\documentclass[12pt]{amsart}
\usepackage{amsmath, amsthm, amssymb,cite,enumitem}
\usepackage{fullpage}
\usepackage{color}
\usepackage{hyperref}
\usepackage[makeroom]{cancel}
\usepackage{mathtools}
\mathtoolsset{showonlyrefs}
\usepackage[english]{babel}
\usepackage{framed}
\usepackage[utf8]{inputenc}
\usepackage{amsmath}
\usepackage{amsfonts}
\usepackage{amssymb}
\usepackage{calligra}
\usepackage{calrsfs}
\usepackage{yfonts}
\usepackage[mathscr]{euscript}
\usepackage{enumitem}
\usepackage{xcolor}
\usepackage{tikz}

\usepackage[normalem]{ulem}

\usepackage{bbm}
\usepackage{hyperref}

\allowdisplaybreaks   

\theoremstyle{plain}
\newtheorem{de}{Definition}[section]
\newtheorem{lem}[de]{Lemma}
\newtheorem{prop}[de]{Proposition}
\newtheorem{cor}[de]{Corollary}
\newtheorem{thm}[de]{Theorem}

\theoremstyle{definition}
\newtheorem{rem}[de]{Remark}

\numberwithin{equation}{section}

\newcommand{\Id}{\mathbf{I}}
\newcommand{\covD}{\mathbf{D}}	

\newcommand{\imu}{i}
\renewcommand{\epsilon}{\varepsilon}

\renewcommand{\Im}{\operatorname{Im}}
\renewcommand{\Re}{\operatorname{Re}}

\newcommand{\R}{{\mathbb{R}}}
\newcommand{\N}{{\mathbb{N}}}
\newcommand{\C}{{\mathbb{C}}}
\newcommand{\Z}{{\mathbb{Z}}}

\renewcommand{\S}{{\mathbb{S}}}

\newcommand{\cH}{{\mathcal{H}}}

\renewcommand{\H}{\mathcal H} 

\begin{document}
 \title[Modified scattering for the three dimensional Maxwell-Dirac system]{Modified scattering for the three dimensional Maxwell-Dirac system}
 
 \author{Sebastian Herr}
\address[S.~Herr]{Fakult\"at f\"ur Mathematik, Universit\"at Bielefeld, Postfach 10 01 31, 33501 Bielefeld, Germany}
\email{herr@math.uni-bielefeld.de}
 \author{Mihaela Ifrim}
  \address[M.~Ifrim]{Department of Mathematics,
480 Lincoln Drive,
213 Van Vleck Hall,
Madison, WI 53706, USA}
\email{ifrim@math.wisc.edu}
\author{Martin Spitz}
 \address[M.~Spitz]{Fakult\"at f\"ur Mathematik, Universit\"at Bielefeld, Postfach 10 01 31, 33501
  Bielefeld, Germany}
\email{mspitz@math.uni-bielefeld.de}


\begin{abstract}

In this work we prove global well-posedness for the massive Maxwell-Dirac system in the Lorenz gauge in $\mathbb{R}^{1+3}$, for
small, sufficiently smooth and decaying initial data, as well as modified scattering for the solutions. Heuristically we exploit the close connection between the massive  Maxwell-Dirac and the wave-Klein-Gordon equations, while 
developing a novel approach which applies directly at the level of the Dirac equations. The modified scattering result follows from a precise description of the asymptotic behavior of the solutions inside the light cone, which we derive via the method of testing with wave packets of Ifrim-Tataru.

\end{abstract}

\maketitle

\setcounter{tocdepth}{1}
\tableofcontents

\section{Introduction and main results}

The problem we will address in this work is the Cauchy problem for the Maxwell-Dirac system on the Minkowski space-time $\mathbb{R}^{1+3}$. In the Lorenz  gauge 
	 $$\partial^\mu A_\mu = 0,$$ it has the form
\begin{equation}
\left\{
	\begin{aligned}
	\label{eq:MD}
	- i& \gamma^\mu \partial_\mu \psi + \psi = \gamma^\mu A_\mu \psi, \\
	& \Box A_\mu = - \overline{\psi} \gamma_\mu \psi.
	 \end{aligned}
  \right.
\end{equation}
This is a fundamental model arising  from relativistic field theory, and it describes the interaction of an electron with its self-induced electromagnetic field. 
The main interest here is on the long time dynamics of the Cauchy problem
with prescribed initial data at time $t = 0$,
\begin{equation}
\label{id}
\psi(0, x) = \psi_0(x), \quad A_{\mu}(0,x)=a_\mu(x),\quad \partial_tA_{\mu}(0,x)=\dot{a}_\mu(x) .
\end{equation}

\medskip

The unknowns are the spinor field  $\psi=\psi(t,x)$, taking values in $\mathbb{C}^4$, and the real-valued potentials $A_{\mu}(t,x)$, with $\mu=\overline{0,3}$.
Without loss of generality we have set the mass in the Dirac equation to be equal to $1$.

\medskip

The main results we present in this paper address  two fundamental questions: (i)  the  global
existence of solutions to the above system, for small initial data satisfying some mild regularity and spatial decay assumptions, and (ii)  the asymptotic description of the solutions.
Examining more closely the asymptotic behavior of the solutions at infinity,
we will show that a modified scattering phenomenon occurs. Precisely, we prove that inside the light cone the following hold:
\begin{enumerate}[label=(\roman*)]
\item  $A$ has $t^{-1}$ decay,
\item $\psi$ does decay at the dispersive
$t^{-\frac32}$ rate, but with a logarithmic
phase correction.
\end{enumerate}

\smallskip

Compared with prior related  works, our novel contributions here include the following:
\begin{itemize}
\item Even though our problem is semilinear in nature, the asymptotic description of the solutions gives a modified scattering result, that reveals a stronger coupling between the Dirac and the Maxwell equation,  more than one suspects when taking a first glance at  the nonlinearity.
\item To a large extent our estimates  are Lorentz invariant, which reflects the full Lorentz symmetry of the Maxwell-Dirac system in the Lorenz gauge, and  is a consequence of having derived the Lorentz vector fields that commute with the linear component of our system \eqref{eq:MD}. 
\item We make no assumptions on the support of the initial
  data. Furthermore, we make very mild decay assumptions on the
  initial data at infinity. In particular, we use only three Lorentz vector fields in the analysis, which is close to optimal and significantly below all previous results on Maxwell-Dirac.
\item Rather than using arbitrarily high regularity, here we work with very limited regularity for the initial data. We only require energy-space bounds on three vector fields which are complemented with bounds on a maximum of nine derivatives without any further decay assumption.
\item To compare with the best results on the massive Maxwell-Klein-Gordon equation \cite{KWY19 ,FWY21}, the global results require up to two derivatives with an additional decay of $r^{-\frac52-}$ in the energy space, which is different from the one for Maxwell-Dirac.
\item In terms of methods, our work employs a combination
  of energy estimates localized to dyadic space-time regions, and
  pointwise interpolation type estimates within the same regions. This is  akin
  to ideas previously used by Metcalfe-Tataru-Tohaneanu \cite{mtt12} in
  a linear setting, and then later refined to apply to a quasilinear setting in the work of Ifrim-Stingo \cite{ifrims}.
  \item Similar to the hyperboloidal foliation method developed in~\cite{Kl85, tataru-hyp, LFM16, LFM24}, we make essential use of the hyberbolic nature of the problem. In order to minimize the regularity and decay assumptions and to effectively employ the testing by wave packet method, we proceed differently in using space-time estimates in the $C_{TS}$ regions.
  \item The asymptotic description of the spinor vector field $\psi$ is obtained using the wave packet testing method of Ifrim-Tataru \cite{IT4, IT1, IT2, IT3}, combined with a novel set of projections that we uncovered in the analysis of the Dirac equation.
\item We identify an asymptotic system for $\psi$ and $A$ inside the light cone,
which has a very clean expression in hyperbolic coordinates.
\end{itemize}

\subsection{Previous work}
A brief survey of previous results on the massive Maxwell-Dirac system and related equations is in order.  We would like to include a more exhaustive list of works, in order to create a context of ideas and results that have emerged in this line of research, in higher dimension, as well as works that address related models like the massless Maxwell-Dirac system or Maxwell-Klein-Gordon systems.
We start with a brief survey of previous results on \eqref{eq:MD} and related equations, namely with the early work on local well-posedness of \eqref{eq:MD} on $\mathbb{R}^{1+3}$ by Gross \cite{Gr} and Bournaveas \cite{Bou}, followed by the more recent work of D'Ancona--Foschi--Selberg \cite{DFS2}  where they established local well-posedness of \eqref{eq:MD} on $\mathbb{R}^{1+3}$ in the Lorenz gauge $\partial^{\mu} A_{\mu} = 0$ for data $\psi(0) \in H^{\epsilon}, A_{\mu}[0] \in H^{\frac{1}{2} + \epsilon} \times H^{-\frac{1}{2}+\epsilon}$, which is almost optimal. Relevant to their method of proof is their  discovery of a deep system null structure of \eqref{eq:MD} in the Lorenz gauge. A very interesting complementary result of Selberg and Tesfahun is the work in \cite{ST21}, which proves ill-posedness below the charge class in space dimensions three and lower. Also, we mention the work on uniqueness of Masmoudi--Nakanishi \cite{MN03}. In more recent work, \cite{GO},  Gavrus and Oh  obtained global well-posedness of the massless Maxwell-Dirac equation in Coulomb gauge on $\mathbb{R}^{1+d}$ $(d \geq 4)$ for data with
small scale-critical Sobolev norm, as well as modified scattering of the solutions. In \cite{lee2023scattering}, Lee has obtained linear scattering for solutions of \eqref{eq:MD} on $\mathbb{R}^{1+4}$.
\smallskip

In terms of global well-posedness, D'Ancona--Selberg \cite{DS} have already obtained a global result for  \eqref{eq:MD} on $\mathbb{R}^{1+2}$ and proved global well-posedness in the charge class.  Regarding work in  $\mathbb{R}^{1+3}$ for\eqref{eq:MD}, we also mention the work of Georgiev \cite{Geo}, Flato, Simon, and Taflin \cite{FST}, and  Psarelli \cite{Psa05}  on global well-posedness for small, smooth and localized data,  as well as the works  \cite{BMS, NM1} on the non-relativistic limit and \cite{NM2} on unconditional uniqueness at regularity $\psi \in C_{t} H^{1/2}, \ (A,\partial_t A) \in C_{t} (H^{1} \times L^{2})$ in the Coulomb gauge. Simplified versions of~\eqref{eq:MD}   were studied in~\cite{P2014, C2020, CKLY2022}. 
Also, stationary solutions have been constructed by Esteban--Georgiev--S\'er\'e \cite{EGS96}.

The next few paragraphs will also discuss related models as they played a crucial role in the ideas that emerged in the study of the Maxwell-Dirac system. For example, a scalar counterpart of \eqref{eq:MD} is the \emph{Maxwell--Klein--Gordon equations} (MKG). We first discuss the massless (MKG). In $d = 3$, Klainerman--Machedon \cite{KM94} proved global well-posedness in the Coulomb and temporal gauge. The asymptotic behavior was investigated by Candy--Kauffman-Lindblad~\cite{CKL19}. Further recent work studying these models should be mentioned: In the energy critical case $d =4$, local well-posedness results for (MKG) were proved by Krieger--Sterbenz--Tataru \cite{KST}. Global well-posedness of (MKG) for arbitrary finite energy data was recently established by Oh and Tataru \cite{OT2, OT3, OT1}, and independently by Krieger--L\"uhrmann \cite{KL}.

Recent contributions to the massive (MKG) in $d = 3$ include the work by Klainerman--Wang--Yang~\cite{KWY19}, establishing asymptotic properties for non-compactly supported data, and Fang--Wang--Yang~\cite{FWY21}, allowing for large Maxwell fields. A precise modified scattering result in the spirit of our Theorem~\ref{thm:ms} for the massive (MKG) in Lorenz gauge was established by Chen~\cite{Ch24} and similarly for the wave-Klein-Gordon system by Chen--Lindblad in~\cite{CL23}. Both these systems and the Maxwell-Dirac system share the dispersive decay heuristics and the strong influence of the hyperbolic nature of the problem. We emphasize that in the Maxwell-Dirac case the Dirac component satisfies a first order equation, resulting in a different nonlinear coupling. Moreover, the intricate spinorial structure of the Maxwell-Dirac system requires new insights, as we will describe in further detail below Theorem~\ref{thm:ms}.

Another model which contributed to the circle of ideas  later circulated in this research direction is provided by the works on the \emph{Dirac-Klein-Gordon systems}. Here, recent work includes the work of D'Ancona-Foschi  \cite{DAncona2007}, as well as the most recent work of Bejenaru and Herr \cite{bejher}, where under a non-resonant condition on the masses, they proved global well-posedness and scattering for the massive Dirac-Klein-Gordon system with small initial data of subcritical regularity in $d=3$.

Work on Dirac equations was also influential in the results obtained for Maxwell-Dirac equations. Notable recent results here are the optimal small data global well-posedness works which were proved recently for the \emph{cubic Dirac equation} in $\mathbb{R}^{1+2}$ and $\mathbb{R}^{1+3}$ by Bejenaru--Herr \cite{BH1, BH2} (massive) and Bournaveas--Candy \cite{BC} (massless). The references in this paragraph  make use of  a feature that the Dirac equation possesses, namely a spinorial null structure.

We also insist on mentioning that our list of references, and the references within these works, is by no means exhaustive; the interested reader can see it as a suggestion of most relevant works related to our current work.   

\smallskip

We would like to mention that our  work is very  different from previous works, in that it does not make use of a spinorial null structure  which traditionally has been developed in order to relate the Dirac equation to the Klein-Gordon models; more so this connection was exploited in scattering results that have emerged for Maxwell-Dirac equations.  Instead,  we work  directly at the level of the Dirac equation in order to uncover the modified scattering behavior. In doing so we reveal a new structural property of the Dirac equations, which is more suitable to global dynamic purposes, explicitly in deriving the asymptotic equation for the spinor vector field $\psi$.

Recent work of the second author with Tataru on modified scattering for a series of relevant models  \cite{IT4, IT1, IT2} played a crucial role in this novel approach we present here.  A comprehensive and exhaustive expository work on recent developments on  modified scattering   is due to the second author and Tataru; see \cite{IT4}. A second important reference that played a direct  role in the energy and pointwise estimates we perform here is the work of the second author with Stingo \cite{ifrims} on almost global existence for wave-Klein-Gordon systems.  

\subsection{The Maxwell-Dirac system}

We consider the Maxwell-Dirac system  on the Minkowski space-time $\mathbb{R}^{1+d}$ for space dimension $d=3$. The space-time coordinates are denoted by $x^\alpha$ with $\alpha =\overline{0,3}$ and 
$t = x^0$, and the Minkowski metric and its inverse are
\[
(g_{\alpha\beta} ) := \mbox{diag}(-1, 1, 1, 1),
\qquad 
(g^{\alpha\beta} ):= \mbox{diag}(-1, 1, 1, 1),
\]
with standard conventions for raising and lowering indices, and $\Box=-\partial^\nu\partial_\nu$.

\bigskip

The Dirac equation is described using the ``gamma matrices'', which are $4 \times 4$ complex-valued matrices $\gamma^{\mu}$ with $\mu$ ranging from $0$ to $3$, 
\begin{align*}
	\gamma^0 := \begin{pmatrix}
		\Id_{2} &0 \\
		0 &-\Id_{2}
	\end{pmatrix}, \qquad
	\gamma^j := \begin{pmatrix}
		0 &\sigma^j \\
		-\sigma^j &0
	\end{pmatrix}
\end{align*}
with the Pauli matrices given by
\begin{align*}
	\sigma^1 := \begin{pmatrix}
		0 &1 \\
		1 &0
	\end{pmatrix}, \qquad
	\sigma^2 := \begin{pmatrix}
		0 &-\imu \\
		\imu &0
	\end{pmatrix}, \qquad
	\sigma^3 := \begin{pmatrix}
		1 &0 \\
		0 &-1
	\end{pmatrix},
\end{align*}
and satisfying the anti-commutation relations
\begin{equation} \label{eq:gmmRel}
	 (\gamma^{\mu} \gamma^{\nu} + \gamma^{\nu} \gamma^{\mu}) = -2 g^{\mu \nu} \, \Id_{4},
\end{equation}
where $\Id_{4}$ is the $4 \times 4$ identity matrix; if no confusion is created, a handy short hand notation we will be using is $\Id_{4}=:\Id$.

 Given a vector valued function (spinor field) $\psi$  on $\mathbb{R}^{1+3}$  that takes values in $\mathbb{C}^{4}$, on which $\gamma^{\mu}$ acts as multiplication,  we define the following \emph{conjugation operation}
\begin{equation}
\label{def:bar for vect}
\overline{\psi}: =\psi^\dag \gamma^0,
\end{equation}
where  $\psi^\dag$ is the Hermitian adjoint of $\psi$.
The same  conjugation relation defined for vectors in equation \eqref{def:bar for vect} extends to general $4\times 4$ matrices $\gamma$ 
\begin{equation*} 
	\overline{\gamma}: = \gamma^0 \gamma ^\dag \gamma^{0}. 
\end{equation*}
In particular for the matrices $\gamma^{\alpha} $ above one easily verifies that
\begin{equation*}
\overline{\gamma^{\alpha}}=\gamma^{\alpha}.
\end{equation*}

\bigskip

A \emph{spinor field} $\psi$ is a function on $\mathbb{R}^{1+3}$ or on any  open subset of $\mathbb{R}^{1+3}$ that takes values in $\mathbb{C}^{4}$. 
Given a real-valued 1-form $A_{\mu}$ (connection 1-form), we introduce the \emph{gauge covariant derivative} on spinors
\begin{equation*}
	\covD_{\mu} \psi := \partial_{\mu} \psi - i A_{\mu} \psi,
\end{equation*}
and the associated \emph{curvature 2-form} 
\begin{equation*}
	F_{\mu \nu} : = \partial_{\mu} A_{\nu} - \partial_{\nu} A_{\mu}=(\mbox{d}A)_{\mu\nu}.
\end{equation*}

The Maxwell–Dirac equations describe the relativistic quantum electrodynamics of particles within self-consistent generated and external electromagnetic fields.
The relativistic Lagrangian field describing the interaction between a connection 1-form $A_{\mu}$, representing an electromagnetic potential, and a spinor field $\psi$, modeling a charged fermionic field is a  space-time integral that takes the form
\begin{equation*}
	\mathcal{S}[A_{\mu}, \psi] = \iint_{\mathbb{R}^{1+3}} \frac{1}{4} F_{\mu \nu} F^{\mu \nu} + i \langle\gamma^{\mu} \covD_{\mu} \psi, \gamma^{0} \psi\rangle -  \langle \psi,  \gamma^{0}\psi\rangle \, dtdx .
\end{equation*}
Here $\langle \psi^{1}, \psi^{2}\rangle := (\psi^{2} )^\dagger \psi^{1}$ is the usual inner product on $\mathbb{C}^{4}$. The Euler--Lagrange equations for $\mathcal{S}[A_{\mu}, \psi]$ take the form
\begin{equation} \label{eq:MD-long}
\left\{
\begin{aligned}
	&\partial^{\nu} F_{\mu \nu} = \langle\psi, \gamma^{0} \gamma_{\mu} \psi\rangle \\
	&i \gamma^{0} \gamma^{\mu} \covD_{\mu} \psi = \gamma^{0} \psi.
\end{aligned}
\right.
\end{equation}
We will refer to \eqref{eq:MD-long} as the \emph{Maxwell--Dirac equations}.

A key   feature of \eqref{eq:MD-long} is its  invariance under gauge transformations meaning that  given any solution $(A, \psi)$ of \eqref{eq:MD-long} and a real-valued function $\chi$, called gauge transformation, on $I \times \mathbb{R}^{3}$, the gauge transform $(\tilde{A}, \tilde{\psi}) = (A + \mbox{d} \chi, e^{i \chi} \psi)$ of $(A, \psi)$ is also a solution to \eqref{eq:MD-long}. This in fact says that relative to this gauge transform we should think of a solution as being  an  equivalence class of functions that are solutions to our problem.  

In order to address the well-posedness theory we need to remove the ambiguity arising from this invariance,  for our system \eqref{eq:MD-long}, and fix the gauge. Traditionally there are several gauges that have been used to address this issue. This includes for instance the \emph{Coulomb gauge} $\partial_j A_j = 0$,
which leads to a  mix of hyperbolic
and elliptic equations. Another possible gauge choice  is the \emph{temporal gauge} $A_0=0$, which retains finite speed of propagation but loses some ellipticity.
In our paper we impose the  \emph{Lorenz gauge} condition, which 
reads
\begin{equation*}
	\partial^\mu A_\mu =  0,
\end{equation*}
and has the advantage  that it is Lorentz invariant, resulting in a more symmetric form of the equations (nonlinear wave equations) compared to the other choices discussed above.

When applied to \eqref{eq:MD-long}, the Lorenz gauge leads us to the system
\begin{equation}\label{eq:MD-lg}
\left\{
	\begin{aligned}
	- i \gamma^\mu \partial_\mu \psi + \psi ={}& \gamma^\mu A_\mu \psi \\
	 \Box A_\mu ={}& - \overline{\psi} \gamma_\mu \psi \\
	 \partial^\mu A_\mu ={}& 0.
	 \end{aligned}
  \right.
\end{equation}

The main interest here is on the long time dynamics of the Cauchy problem
with prescribed initial data at time $t = 0$,
given by \eqref{id}.

If one considers only the (self-contained) system formed by the two equations in \eqref{eq:MD-lg}, then the
initial data above can be chosen arbitrarily. However, if in addition one also adds the third equation, then the initial data is required to satisfy the following constraint equations
\begin{equation}\label{eq:constr}
\left\{
	\begin{aligned}
\dot a_0 ={}&  \ \partial_j a_j 
\\
\Delta a_0 ={}&  \ \partial_j \dot a_j +|\psi_0|^2 ,
 \end{aligned}
  \right.
\end{equation}
which are then propagated to later times by the flow generated by the first two equations.

\bigskip

\subsection{Functional spaces} 

In this section, we introduce the main function spaces we use to prove our main results. As a guideline we use the scaling of the massless Maxwell-Dirac system, which is known to be  invariant under the scaling $(\lambda > 0)$
\begin{equation*}
(\psi, A_{\mu})\rightarrow (\lambda^{-\frac{3}{2}}\psi (\lambda^{-1}t, \lambda^{-1}x), \lambda^{-1}A_{\mu} (\lambda^{-1}t, \lambda^{-1}x) ).
\end{equation*}
This leads to the critical Sobolev space $\mathcal{H}^{0}:=L^2 \times \dot{H}^{\frac{1}{2}}\times\dot{H}^{-\frac{1}{2}}$; the first space measures $\psi$ and the remaining spaces measure position and velocity respectively. In terms of interesting quantities let us state the ones that are available for this model, but emphasize that none of them will play a role in our analysis:
\begin{enumerate}
\item the \textbf{charge conservation}
\begin{equation*}
    {\bf q} := - \int |\psi|^2\, dx  =- \|\psi_0\|^2_{L^2},
\end{equation*}

\item  the \textbf{energy}
\begin{equation*}
{\bf E} := \int i \overline{\gamma^j \mathbf{D}_j \psi} \psi +  \overline{\psi} \psi + \frac12 (|E|^2+|B|^2) \, dx, \qquad 
\end{equation*}
\end{enumerate}
with the electric field $E=(F_{10},F_{20},F_{30})$ and the magnetic field $B=(F_{23},F_{31},F_{12})$. Note that ${\bf E} $ is not sign definite.
In terms of terminology, our problem is called \emph{charge critical}, and this is because the charge is measured in the critical space $L^2$. In $d=4$, the critical Sobolev space would change and the energy will be expressed in terms of these critical Sobolev spaces, leading to the terminology \emph{energy critical Maxwell-Dirac system}.

\subsection{Main results} 

 To study the small data long time well-posedness problem for the
nonlinear evolution  \eqref{eq:MD} one needs to add some decay assumptions for the initial data to the
mix. Before doing so we need to introduce two  small pieces of notations. Section~\ref{s:not} will contain the bulk of the notations and definitions pertaining to this work.  
\begin{itemize}
\item We make the convention of using upper-case letters for multi-indices, e.g.\
$\partial^I=\partial_{x_0}^{i_0}\cdots \partial_{x_d}^{i_d}$ and $x^I={x_0}^{i_0}\cdots {x_d}^{i_d}$, where $I=(i_0,\ldots,i_d)$, and we write $I_0$ if $i_0=0$. 
We also use the shorthand notation $\partial^k=\{\partial^I\}_{|I|=k}$ and $\partial^{\leq k}=\{\partial^I\}_{|I|\leq k}$, for $k\in \N_0$. Similar as above, $k_0$ indicates that there are no $\partial_t(=\partial_{x_0})$ derivatives.

\item We also recall the vector fields (denoted here by) $\Omega_{\alpha\beta}$,
\[
\Omega_{\alpha\beta}:= x_{\beta} \partial_{\alpha} - x_{\alpha}\partial_{\beta}, \quad \alpha, \beta =\overline{0,3},
\]
which represent the generators of the Lorentz group.
\end{itemize}

At this point we are ready to state our  first main theorem, which
clarifies the type of initial data we are considering:

\begin{thm}
\label{thm:gwp-easy}
Assume that the initial data $(\psi_0,a,\dot a)$ for the system \eqref{eq:MD}
satisfies the smallness and decay conditions
\begin{equation}\label{small-data}
\sum_{3j_0 + k_0 \leq 9} 
\| |x|^{j_0} \partial^{\leq j_0+k_0} \psi_0\|_{L^2}
+ \| |x|^{j_0} \partial^{j_0+k_0} a\|_{\dot H^\frac12} + \| |x|^{j_0} \partial^{j_0+k_0} \dot a\|_{\dot H^{-\frac12}} \leq \epsilon,
\end{equation}
as well as the additional low frequency bound 
\begin{equation}\label{small-data-nu}
\|a^\mu\|_{\dot{H}^{\frac12 - \nu}} + 
\|\dot a^\mu\|_{\dot{H}^{-\frac12 - \nu}} \leq \epsilon, \qquad \nu > 0.
\end{equation}

If $\epsilon$ is small enough, then the solution $(\psi,A)$ is global in time, and satisfies the vector field 
bounds
\begin{equation*}
\sum_{3|J| + |K| \leq 9}
\| \Omega^{J} \partial^{K} \psi\|_{L^2}
+ \| \Omega^J \partial^{K} A\|_{\dot H^\frac12} + \| \Omega^J \partial^{K} \partial_t A\|_{\dot H^{-\frac12}} \lesssim \epsilon \langle t\rangle ^{C\epsilon},
\end{equation*}
as well as the pointwise bounds
\begin{equation}\label{thm-point}
|\psi(t,x)| \lesssim  \frac{\epsilon}{\langle t+|x|\rangle^{\frac32-C\epsilon} \langle t-|x|\rangle^{C\epsilon} }, \qquad  |A(t,x)| \lesssim \frac{\epsilon}{\langle t+|x|\rangle }, \; (t>0)
\end{equation}
and, in addition, when $|x|-t\geq \langle t\rangle^{\frac13}$,
\begin{equation*}
    |\psi(t,x)| \lesssim \frac{\epsilon}{{\langle t+|x|\rangle }^{\frac{3}{2}}\langle |x|-t\rangle^{\delta}}.
\end{equation*}
\end{thm} 

 Concerning the Lorenz gauge condition, we point out that the compatibility conditions \eqref{eq:constr} between $a_0$ and $\psi_0$ give rise to the so called charge-problem, see Remark \ref{rem:charge-problem} below.

\begin{rem} For smooth and localized initial data the existence of a unique global solution of~\eqref{eq:MD} was shown in Theorem~1 in the work of Georgiev~\cite{Geo}. On the other hand the work in Psarelli in \cite{Psa05} provides a lower regularity global well-posedness result, though working with compactly supported initial data which is a very restrictive assumption to make.   The same result also includes pointwise decay bounds for the solutions, however no asymptotic equations are derived. By contrast our result applies at low regularity  without using any support assumptions, and additionally we derive clean asymptotic equations for the solutions; see Theorem~\ref{thm:ms} below. 
\end{rem}

We comment here on the decay rates for $\psi$ and $A$ in the above theorem. Beginning with $\psi$, we see that we have the standard dispersive decay rate of $t^{-\frac32}$ inside the cone, but a better decay rate outside. The latter happens simply because of the initial data localization, as the group velocities for $\psi$ waves lie 
inside the cone, and approach the cone only in the high frequency limit.
However, because of the $t^{-1}$ size of $A$  there are strong nonlinear interactions that happen inside the cone which prevent  standard scattering and instead remodulate the $\psi$ waves, suggesting there should be a modified scattering asymptotic.

Turning our attention to $A$, if one were to naively think of the $A$ equation  as a  linear homogeneous wave then the bulk of it would be localized near the cone, with better decay inside, and would have a minimal interaction with $\psi$. 
 However, as it turns out, the bulk of $A$  inside the cone  comes from solving the wave equation with a $\psi$ dependent quadratic source term. This is what produces the exact $t^{-1}$ decay rate.  However, we do get the expected decay  estimates  for $\nabla A$ both outside  and inside the cone. 

To capture the asymptotic behavior of $\psi$  and $A$ at infinity,  and also understand the coupling between $A$ and $\psi$ in time-like directions, one needs to make the above heuristic discussion rigorous. We do this in the next theorem, which describes the asymptotic profiles for $\psi$ and $A$ as well as the modified scattering asymptotics.

\begin{thm}\label{thm:ms}
There exists $\delta > 0$ so that, for all solutions $(\psi,A)$ for the Maxwell-Dirac equations as in Theorem~\ref{thm:gwp-easy}, 
there exist asymptotic profiles
\begin{equation*}
(\rho^{\pm}_{\infty},a^\mu_\infty) \in C^\frac12(B(0,1)), 
\end{equation*}
vanishing at the boundary,
so that inside the light cone we have the asymptotic expansions
\begin{equation}
\label{A-asymptotic}
A^\mu(t,x) = (t^2-x^2)^{-\frac12} a^{\mu}_\infty(x/t) + O(\epsilon\langle t\rangle ^{-1}
\langle t-r \rangle^{-\delta}) ,
\end{equation}
respectively
\begin{equation}
\label{psi-asymptotic}
\psi(t,x) = (t^2-x^2)^{-\frac34} 
\sum_{\pm} e^{\pm i\sqrt{t^2-x^2}}
e^{i \frac{x_\mu a^\mu_{\infty}(x/t)}{2 \sqrt{t^2-x^2}}  \log (t^2-x^2)}
\rho^\pm_\infty(x/t) + O(\epsilon \langle t\rangle^{-\frac32} \langle t-r\rangle^{-\delta}) , \end{equation}
where $a_\infty^\mu$ is uniquely determined by $\rho^\pm_\infty$ via the elliptic equation
\begin{equation}\label{coupling}
  (-1-\Delta_H) a^\mu_{\infty}= 
  - \overline{\rho^\pm_{\infty}} \gamma^\mu  \rho^\pm_{\infty},
\end{equation}
for the hyperbolic Laplacian $\Delta_H$ in the Klein--Beltrami disk model.
\end{thm}

\medskip

The statement of the last theorem is somewhat brief since many notations are introduced 
later. However, some comments may be helpful.

\begin{enumerate}
\item Modified scattering: the asymptotic 
expansion for $\psi$ in \eqref{psi-asymptotic}
departs from the corresponding linear asymptotic
due to the logarithmic phase correction.
This is in turn generated by the exact $t^{-1}$
decay rate for $A$ inside the cone, which is also not consistent with the linear theory.
    \item Hyperbolic geometry:  the asymptotic profiles should be best viewed as functions
    on the hyperbolic space $H$, with the Klein-Beltrami representation via the
    velocity coordinate $v = x/t \in B(0,1)$.
\item Profile regularity: the $C^\frac12$ bound represents just the simplest common regularity property for $\rho^{\pm}$ and $a^{\mu}_\infty$, but in fact we prove an expanded set of bounds, which are best expressed in the hyperbolic setting, where the Lorentz vector fields $\Omega$ play the role of normalized derivatives:
\begin{equation*}
|\Omega^{\leq 2} a^\mu_\infty(v) | \lesssim \epsilon^2 (1-v^2)^\frac12,    
\end{equation*}
\begin{equation*}
| \rho^{\pm}_\infty(v) | \lesssim \epsilon(1-v^2)^{1-C\epsilon} ,   
\end{equation*}
\begin{equation*}
    \| (1-v^2)^{-\frac32+C\epsilon}
\Omega^{\leq 2} \rho^{\pm}_\infty\|_{L^2} 
 \lesssim \epsilon  .
\end{equation*}
We refer the reader to the last section for more details.

\item Higher regularity: If the initial data for $(\psi,A)$ has additional regularity, 
then the hyperbolic space regularity
of $(\rho^{\pm}_\infty, a^{\mu}_\infty)$ 
can be improved, as well as the decay rate for 
$\rho^{\pm}_\infty$ at the boundary of the unit ball. However, there is no improved decay rate
for $a^{\mu}_\infty$; instead,  $
(1-v^2)^{-\frac12} a^{\mu}_\infty$ will always have a limit at the boundary. This can be easily seen  as  a slight extension of the proof of Proposition~ \ref{prop:Ainhom}, or it follows from the asymptotics of the fundamental solution for $(-1-\Delta_{H})$ in Remark~\ref{r:fundamental}. A similar result has been proven in the Maxwell-Klein-Gordon case in \cite{Ch24} and in the Wave-Klein-Gordon case in \cite{CL23}.

\item Low frequency assumption: the additional condition \eqref{small-data-nu} on the initial data for $A$ is necessary in order to obtain the expansion \eqref{A-asymptotic} even if $\psi = 0$. Otherwise, as $\nu \to 0$, we correspondingly must have $\delta \to 0 $ in 
\eqref{A-asymptotic}.

\item Connection to Klein-Gordon: the Dirac waves are closely related to Klein-Gordon waves,
and this is reflected in the form of the asymptotic expansion for $\psi$. 
The two  components $\rho^{\pm}_{\infty}$
correspond exactly to the two Klein-Gordon 
half-waves, as it can be readily seen by examining the phases of the associated terms in the $\psi$ expansion. In a related vein, 
the ranges of $\rho^{\pm}_\infty(v)$ are restricted to $v$ dependent but Lorentz invariant subspaces $V^{\pm}_v$, see \eqref{eq:V+}, which are orthogonal with respect to the $\langle \cdot,\cdot \rangle_H $ inner product defined in \eqref{eq:DefInnerProductH}. With these notations, the source term in the coupling equation \eqref{coupling} takes the form 
\begin{equation*}
   \overline{\rho^\pm_{\infty}} \gamma^\mu  \rho^\pm_{\infty} = \frac{v^\mu}{\sqrt{1-v^2}} (\| \rho^+_\infty\|_{H}^2 + \| \rho^-_\infty\|_{H}^2).
\end{equation*}

\item Charge conservation: this is reflected in the asymptotic profile via the identity
\begin{equation*}
 \| \rho^+_\infty\|_{L^2(H)}^2 + \| \rho^-_\infty\|_{L^2(H)}^2 = \|\psi_0\|_{L^2}^2.
\end{equation*}

\end{enumerate}

\medskip

 Finally we comment on the low frequency assumption \eqref{small-data-nu}:

 \begin{rem}
 The result in Theorem~\ref{thm:gwp-easy} also holds
 without the assumption \eqref{small-data-nu}
 if one is willing to slightly relax the pointwise bound for $A$ to
\[
|A(t,x)| \lesssim  \ \frac{\epsilon}{\langle t+|x|\rangle } \log \frac{2\langle t+|x| \rangle} {\langle t-|x|\rangle}.
\] 
See also Proposition~\ref{prop:Ahom} and the following Remark~\ref{rem:mu0} later on, which is the only place in the paper where 
this assumption is needed and used. 

One possible approach to achieve this is to rely on the weaker BMO bound in \eqref{eq:A-hom-BMO} for $A$; this in turn would require replacing 
the $L^\infty$ endpoint with a BMO endpoint in some of the vector field interpolation Lemmas. Alternatively, one can slightly rebalance the bootstrap bounds for $A$ and $\psi$, from $L^\infty$ and $L^6$ to $L^{\infty-}$ and $L^{6+}$, with appropriate
changes in the powers of $t$.

We chose not to pursue either alternative here because on one hand this assumption 
turns out to be needed for Theorem~\ref{thm:ms}, and on the other hand, it allows for a more streamlined argument.
 \end{rem}

 \begin{rem}\label{rem:charge-problem}
 The main results of this paper, Theorem \ref{thm:gwp-easy} and Theorem \ref{thm:ms}, apply to the PDE system \eqref{eq:MD}, with conditions on the full set of initial data $(a,\dot{a})$ and $\psi_0$. To obtain a Lorenz gauge solution to Maxwell-Dirac, the compatibility conditions \eqref{eq:constr} have to be satisfied. However, if $\psi_0\ne 0$, the equation $\Delta a_0=\partial_j \dot{a}_j + |\psi_0|^2$ shows that $a_0$ cannot decay faster than $1/|x|$, which is incompatible with the assumptions on $a_0,\dot{a}_0$ in the above Theorems. This is known as the charge problem, see e.g.\ \cite[Section 4]{CKL19} in the Maxwell-Klein-Gordon case. There is a standard solution to this, which we sketch here: With the conserved charge $\mathbf{q}=-\int |\psi_0|^2$ as above, define $Q(t,x)=\varrho(|x|-t)\frac{\mathbf{q}}{4\pi |x|}$ for a smooth function $\varrho$ with $\varrho(s)=1$ for $s\geq 1$ and $\varrho(s)=0$ for $s<1/2$. Then, the key property is that $\Box Q=0$ and this function is zero inside the light cone. The modified gauge field $\tilde{A}_\mu:=A_\mu -\delta_{0\mu} Q$ satisfies the same wave equations as $A_\mu$, but with modified initial data $(\tilde{a}_\mu,\dot{\tilde{a}}_\mu)$. Now, assume that $(a_j,\dot{a}_j,\psi_0)$ satisfy the assumptions of Theorem \ref{thm:gwp-easy} and $(a_0,\dot{a}_0)$ satisfy \eqref{eq:constr}. Then,
 $$\tilde{a}_0=a_0-\varrho(|x|)\frac{\mathbf{q}}{4\pi |x|}, \quad \dot{\tilde{a}}_0=a_0+\varrho'(|x|)\frac{\mathbf{q}}{4\pi |x|}$$
 satisfies the assumptions. Note that $(\psi,\tilde{A}_\mu)$ satisfies the modified system
 \begin{equation}
	\label{eq:ModMD}
	\begin{aligned}
	- i \gamma^\mu \partial_\mu \psi &= \gamma^\mu \tilde{A}_\mu \psi + \gamma^0 Q \psi, \\
	\Box \tilde{A}_\mu &= - \overline{\psi} \gamma_\mu \psi.
	\end{aligned}
\end{equation}
 Then, the same conclusion of Theorem \ref{thm:gwp-easy} holds true for $(\psi,\tilde{A}_\mu)$. In particular,  in view of the decay of $Q$, the same pointwise bounds hold for $(\psi,A_\mu)$. The analysis remains unchanged, except for the estimates for $F$ in Section \ref{s:vf}. To estimate the additional contribution in the Dirac equation, we use that $|\Omega^J\partial^K Q | \lesssim \epsilon^2 \langle t+|x|\rangle^{-1}$.

 Similarly, the conclusion of Theorem \ref{thm:ms} holds for the modified system. Here, the statement remains unchanged, as the system has not changed inside the light cone and we only use the bounds established in Theorem \ref{thm:gwp-easy}.
 \end{rem}

In the direction of our main result, we also mention the recent preprint \cite{CL24}, which was posted on the ArXiv a couple of weeks after our paper. Therein, the authors obtain less precise asymptotics under stronger assumptions on the initial data (which are necessary for the space-time resonance method).

\subsection{Outline of the paper} 
The paper is structured in a modular fashion. This in particular means that each section can be understood separately and only the main result 
carries forward.  We distinguish  four main steps:
\begin{enumerate}
\item  energy estimates for the linearized equation, 
\item vector field energy estimates,
\item 
pointwise bounds  derived from energy estimates (sometimes called Klainerman-Sobolev
inequalities), 
\item asymptotic and wave packet analysis.
\end{enumerate}
While this may seem like a standard approach, there are a number of technical difficulties that prevent us from carrying a straightforward analysis, and also there are several improvements 
we bring to the analysis. 

After Section \ref{s:not} which contains notations and definitions we use throughout our work, we structure the proof of the global result as  a bootstrap argument. But unlike the classical 
approach where a large number of vector field bounds are needed, here our bootstrap assumption involves only pointwise bounds on the solutions,
precisely it has the form
\[
\| \psi(t)\|_{L^6} + \|A(t)\|_{L^\infty} \lesssim 
\frac{C \epsilon}{\langle t \rangle}, 
\]
which is consistent with the linear dispersive decay bounds for the Dirac, respectively the 
wave equation. Then the final objective becomes 
to show that we can improve this bound. This is accomplished in several steps:

\medskip

\emph{ I. Energy estimates for the linearized equation.} These are relatively straightforward,
as they are carried out in our base Sobolev space
$L^2 \times \dot H^\frac12 \times \dot H^{-\frac12}$. Nevertheless, their proof is still 
instructive in understanding how a minimal 
$t^{C\epsilon}$ energy growth can be derived using only the above bootstrap assumptions.

\medskip

\emph{ II. Energy estimates for the solutions.}
This is again done under the above bootstrap condition, and it yields energy bounds with a $t^{C\epsilon}$ growth.
It includes vector field bounds, and for clarity
are separated into several steps. They are first proved for the solution and its higher derivatives in Subsection~\ref{ss:higher derivatives}, second for vector fields in Subsection \ref{ss:vector fields}, 
and finally for both vector fields and derivatives applied to the solution in Subsection~\ref{ss:mixed}. While
using just interpolation inequalities and Gronwall type inequalities in time work in the 
first case, in order to obtain vector field
energy bounds using only our bootstrap assumptions
we work instead in dyadic time slabs denoted by $C_T$, which with the proper set-up enable us to 
optimize the interpolation of vector field bounds.
In this we follow the lead of the earlier work of Ifrim-Stingo~\cite{ifrims}.

\medskip

\emph{III. Pointwise (Klainerman-Sobolev) bounds.}
These are derived from the previous energy bounds,
and are akin to classical Sobolev embeddings but 
on appropriate scales. For this purpose we separate the dyadic time slabs $C_T$ above into 
smaller sets, namely the  dyadic regions
$C^{\pm}_{TS}$, where $T$ stands for dyadic time, $S$ for the dyadic distance to the cone, and $\pm$ for the interior/exterior cone, plus an additional interior region $C_T^{int}$ and 
an exterior region $C_T^{ext}$. Then it becomes 
important, as an intermediate step, to derive
 space-time $L^2$ \emph{local energy bounds}, localized to these sets.  Then our pointwise bounds are akin to Sobolev embeddings or interpolation inequalities in these regions,
 with the extra step of  also using the linear equation in several interesting cases.
We note that these bounds inherit the $t^{C\epsilon}$ extra growth from the energy estimates, so they do not suffice in order to close the bootstrap. 

\medskip

\emph{IV. Asymptotic profiles and the asymptotic equation for the spinor field $\psi$}.
Heuristically one expects a Klein-Gordon type 
asymptotic expansion  for the spinor field,
\[
\psi(t,x) \approx t^{-\frac32} \sum_{\pm}
e^{\pm i \sqrt{t^2-x^2}} \rho^{\pm}(t,x)
\]
with well chosen slower varying \emph{asymptotic profiles} $\rho^{\pm}$.  In the case of the linear
Dirac flow one may choose $\rho^{\pm}$ to depend only on the velocity $v = x/t$, but for our nonlinear flow this is no longer possible. Then we need (i) to identify good asymptotic profiles, 
and (ii) to study their time dependence on rays 
(asymptotic equation). This is carried out in 
Section \ref{s:wp} using the method of 
wave packet testing of Ifrim-Tataru \cite{IT1}, \cite{IT3}, 
\cite{IT2}, \cite{IT4}.
However, the wave packet analysis is carefully adapted to the Dirac system, which is novel and quite interesting.
The asymptotic equation turns out to be an ode  of the form
\[
i \partial_t \rho^{\pm}(t,v) \approx v_\alpha A^\alpha  \rho^{\pm}(t,v).
\]
Since the connection coefficients $A^\alpha$ 
are real, this equation  allows us to propagate uniform pointwise bounds for $\rho^{\pm}$, which are then transferred to $\psi$. Thus, by the end of this section we are able to close the $\psi$ part of the bootstrap loop.

\medskip

\emph{V.  Uniform bounds for $A$.}
The $t^{-1}$ decay bounds for $A$ are obtained in Section \ref{s:A}, directly from the wave equation for $A$. Here one needs to separately estimate the contributions of the initial data and of the source term, where for the latter we use  the $t^{-\frac32}$ decay bounds for $\psi$ from the previous section.  

\medskip

\emph{VI.  Radiation profiles for $\psi$ and  $A$
inside the cone.} 
These are constructed in the last section of the paper, whose final objective is to prove the modified scattering result in Theorem \ref{thm:ms}.  This is achieved in several steps,
where we successively construct

a) an initial radiation profile $\rho^\pm_\infty$
for $\psi$, which is only accurate up to a phase rotation, but suffices for the next step.

b) a radiation profile $a_\infty^{\mu}(v)$ for $A$, which can be thought of as the limit of 
$(t^2-x^2) A^\mu$ along rays $x = vt$.

c) Using the result in part (b) we refine 
the choice of the radiation profile $\rho^\pm_\infty$
for $\psi$, removing the  phase rotation
ambiguity in (a).

\subsection*{Acknowledgments}
The second author, M.I.,  was supported by the Sloan Foundation,  by an NSF CAREER grant DMS-1845037, by a Simons Foundation via a Simons Fellowship, and by a Visiting Miller Professorship via the Miller Foundation.
 M.S.\ and S.H.\ acknowledge support by the Deutsche Forschungsgemeinschaft (DFG, German Research
Foundation) – Project-ID 317210226 – SFB 1283.
We would like to thank the Erwin Schr\"odinger International Institute in Vienna as well the University of California at Berkeley, where part of this paper has been written, for their hospitality.
The authors would like to thank the anonymous referees for their careful remarks which helped a lot to improve the presentation.
\section{Preliminaries and notations}
\label{s:not}

\subsection{Notations}

To express the equations we first need to establish some notations; we begin by recalling the standard rectilinear coordinates $x:=(x^0, x^1, x^2, x^3)$ which is a position in time and space vector; in particular $x^0 :=t$ denotes  the time and  $(x^1, x^2, x^3)$ stands for the spatial position. We will use $x^{\alpha}$, with $\alpha =\overline{0,3}$, to denote  the entries of the vector $x$.  

For indices we have the following traditional convention: (i) Greek indices range over $0, 1, \dots, d$, (ii) Latin indices over $1, \dots, d$, (iii) Einstein summation
convention of summing repeated upper and lower indices over these ranges, and (iv)  raising
and lowering indices is performed using the Minkowski metric.

The equations will be written in covariant form on $\mathbb{R}^{1+3} = \mathbb{R}_t \times \mathbb{R}^3_x$
 with the Minkowski metric 
\[
(g_{\alpha\beta} ) := \mbox{diag}(-1, 1, 1, 1),
\]
which admits the inverse metric 
\[
(g^{\alpha\beta} ):= \mbox{diag}(-1, 1, 1, 1).
\]
We raise and lower the indices with respect to this metric, which in particular calls for the following notation
\[
x_{\alpha} := g_{\alpha\beta}x^{\beta},
\]
where we call $x_{\alpha}$ a \emph{covector}. 
We also can reverse the action with the help of the inverse metric, and  raise the indexes, so that we obtain a vector
\[
 g^{\alpha\beta}x_{\beta} =x^{\alpha}.
\]

Lastly we also recall the multi-index notation we will be using throughout the paper, namely we use upper-case letters for multi-indices, e.g.\
$\partial^I_x=\partial_{x_0}^{i_0}\cdots \partial_{x_d}^{i_d}$ and $x^I={x_0}^{i_0}\cdots {x_d}^{i_d}$, where $I=(i_0,\ldots,i_d)$, and we write $I_0$ if $i_0=0$. Similarly, when using  the multi-index notation for vector fields
\[
\Omega^I=\Omega_1^{i_1} \cdots \Omega_n^{i_n}. 
\]
\subsection{Vector fields}
 To state the main results of this paper we have already used  
 in the introduction the vector fields associated to the symmetries of the Minkowski space-time.
Recall that rotation vector fields
and Lorentz boosts where denoted by 
$\Omega_{\alpha \beta}$, 
\begin{equation*}
\Omega_{\alpha\beta}:=x_\beta \partial_\alpha -x_\alpha \partial_\beta , \qquad \alpha, \beta =\overline{0,3}.
\end{equation*}
Together with the translations, these Lorentz generators will be denoted by $\Gamma$, which we define as
\begin{equation*}
\Gamma:= \left\{ \partial_0,  \partial_1,  \partial_2, \partial_3, \Omega_{\alpha\beta} \right\}.
\end{equation*}
As defined above, $\Omega_{\alpha \beta} $ do not commute with the linear component of the Dirac equation in \eqref{eq:MD} due 
to the vectorial structure of the spinors.
Instead we need to consider a correction to the Lorentz vector fields, which represents the Lie derivative of the spinor field 
with respect to the Lorentz vector fields:
\begin{equation*}
    \hat{\Omega}_{\alpha \beta} := \Omega_{\alpha \beta} + \frac{1}{2} \gamma_\alpha \gamma_\beta, \quad \mbox{ for all } \  0 \leq \alpha < \beta \leq 3.
\end{equation*}
This indeed satisfies
\begin{equation*}
    [\hat{\Omega}_{\alpha \beta}, \imu \gamma^\mu \partial_\mu] = 0.
\end{equation*}
See also~\cite{Ba88, Geo} for the use of $\hat{\Omega}_{\alpha \beta}$ and their property of commuting with the Dirac operator.

We will later apply $\hat\Omega_{\alpha \beta}$ to the Dirac component of the Maxwell-Dirac system~\eqref{eq:MD}. However, this is not the end of the story as we want to apply these vector fields to the nonlinear system, which itself has Lorentz invariance.
 Explicitly, applying $\hat\Omega_{\alpha \beta}$ to the Maxwell-Dirac system~\eqref{eq:MD} implies, for instance, that for the first equation we should formally  be able to express the  RHS as follows
\begin{equation}
\label{eq:Lebnitz psi}
- i \gamma^\mu \partial_\mu \hat\Omega_{\alpha \beta} \psi + \hat\Omega_{\alpha \beta}\psi =  \hat\Omega_{\alpha \beta} \left(\gamma^\mu A_\mu \psi \right) = \tilde\Omega_{\alpha \beta}  A_\mu \gamma^\mu \psi  +   A_\mu \gamma^\mu \hat\Omega_{\alpha \beta}\psi.
\end{equation}
Here the only thing we did was to distribute $\hat\Omega_{\alpha \beta}$, observing that one potential outcome would be to have the corresponding vector field applied to $A_{\mu}$, which is 
naturally different from the vector field applied to $\psi$.  At the same time, this new vector field, denoted here by $\tilde\Omega_{\alpha \beta}$, should  be commuting with the linear component of the second equation. More so, it should distribute itself according to the product rule in the nonlinearity of the wave equation, namely, we should have
 \begin{equation}
 \label{eq:Lebnitz A}
     \Box \tilde\Omega_{\alpha \beta}A_\mu = - \overline{ \hat\Omega_{\alpha \beta}\psi} \gamma_\mu \psi 
	 - \overline{\psi} \gamma_\mu  \hat\Omega_{\alpha \beta}\psi 
	.
 \end{equation}
 Indeed, a direct computation leads to the following 
 expressions for the generators of the Lorentz group 
 of symmetries for the full Maxwell-Dirac system:
 
 \begin{lem}\label{lem:com} The family of vector fields $\left\{ \hat\Omega_{\alpha \beta}, \tilde\Omega_{\alpha \beta}\right\}$, with $\alpha, \beta =\overline{0,3}$, and so that 
 \begin{equation}
  \label{eq:new vf}
 \left\{
\begin{aligned}
&\hat\Omega_{\alpha \beta}:=  \Omega_{\alpha \beta} + \frac{1}{2} \gamma_\alpha \gamma_\beta\\
&\tilde\Omega_{\alpha \beta} A_{\delta}:=  \Omega_{\alpha \beta} A_{\delta} +  g_{\beta\delta} A_{\alpha} -g_{\alpha \delta}A_{\beta} ,
\end{aligned}
\right.
 \end{equation}
 commute with the linear Maxwell-Dirac equations and satisfy the product rule in  \eqref{eq:Lebnitz psi} and \eqref{eq:Lebnitz A}.
 \end{lem}

\begin{proof}
The proof relies on a direct computation where one can take the vector fields in \eqref{eq:new vf} and apply them to the equation \eqref{eq:MD} and show that  both \eqref{eq:Lebnitz psi} and \eqref{eq:Lebnitz A} hold true. The details are left to the interested reader.
\end{proof}

For both the Dirac and the wave components  of \eqref{eq:MD} we have defined ten vector fields and in the following we denote these generalized vector fields by $\Gamma_1$ to $\Gamma_{10}$ (omitting the hat and tilde) and employ multi-index notation in the following, i.e.,
\begin{align*}
	\Gamma^J = \Gamma_1^{j_1} \cdots \Gamma_{10}^{j_{10}}, \qquad J \in \N_0^{10}.
\end{align*}
Separating derivatives and vector fields,
we will weight differently the two kinds of derivatives, and set 
\[
\Gamma^{\leq k} = \{  \Omega^J \partial^I\}_{|I|+3|J| \leq k}.
\]

\subsection{Energies for the Dirac equation
on hyperboloids and orthogonal decompositions in $\C^4$}

Suppose $\psi$ is a solution for the homogeneous Dirac equation.
We can write the $L^2$-conservation law of the Dirac equation in the density-flux form
	\begin{equation}
 \label{dens-flux}
		\partial_t |\psi|^2 + \partial_j ( \psi^\dag  \gamma^0 \gamma^j \psi ) = -2 \Im (\psi^\dag \gamma^0 F ).
	\end{equation}
An immediate consequence of this is the conservation of the $L^2$ norm of the solution
on time slices. However, in this article we will also need to use energy functionals 
on hyperboloids 
\[
H := \{(t,x)\, | \,  t^2 -x^2 = c^2 > 0 \}.
\]
Integrating the density-flux relation 
within the region between $H$ and the 
initial surface $t = 0$ we obtain the energy relation
\[
\|\psi(0)\|_{L^2}^2 = E_H(\psi),
\]
where the energy of $\psi$ on the hyperboloid 
$H$ is given by 
\begin{equation*}
E_H(\psi) = \int_{H} (\nu_0 |\psi|^2 + \nu_j \psi^\dag \gamma^0 \gamma^j \psi)\,  d \sigma.
\end{equation*}
The density for this energy is positive definite since the normal vector to the hyperboloid is time-like.

We first diagonalize the above density 
\[
e_H(\psi) = \nu_0 |\psi|^2 + \nu_j \psi^\dag \gamma^0 \gamma^j \psi
\]
with respect to  the Euclidean metric, by writing
\[
e_H(\psi) = \frac{t}{\sqrt{t^2+x^2}} 
|\psi|^2 - \frac{\langle \psi, x_j \gamma^0\gamma^j \psi\rangle}{\sqrt{t^2+x^2}} = \frac{1}{\sqrt{t^2+x^2}}
( t |\psi|^2 - |x| \langle \psi, \gamma^0\gamma^\theta \psi \rangle ),
\]
where using polar coordinates we have denoted
\begin{equation*}
\gamma^\theta:= \theta_{j} \gamma^j, \qquad \theta = \frac{x}{|x|}.    
\end{equation*}
To complete our diagonalization we need to 
consider the spectral properties of the matrix
$\gamma^0\gamma^\theta$. The matrices $\gamma^\theta$ share with $\gamma^j$ the following properties:

\begin{lem}
\label{lem:gammatheta}
For each $\theta \in \S^2$, the matrix $\gamma^0 \gamma^\theta$ is Hermitian and has double eigenvalues $\pm 1$. 
\end{lem}
\begin{proof}
The starting point is the observation that 
the matrices $\gamma^0 \gamma^j$ have the properties in the lemma, and in particular
$(\gamma^0 \gamma^j)^2= I_4$, whereas different
$\gamma^0 \gamma^j$ anticommute.

Then the fact that $\gamma^0 \gamma^\theta$ is Hermitian is immediate. Since they are smooth in $\theta$, it only remains to show that $(\gamma^0 \gamma^\theta)^2 = I_4$. But this is a direct computation,
\[
(\gamma^0 \gamma^\theta)^2 = \sum_{j,k}
\theta_j \theta_k \gamma^0 \gamma^j \gamma^0 \gamma^k  = - \sum_{j,k}
\theta_j \theta_k  \gamma^j \gamma^k
= -\frac12 \sum_{j,k} \theta_j \theta_k  (\gamma^j \gamma^k + \gamma^k \gamma^j),
\]
where we next  relay on the anti-commutation properties of the matrices $\gamma$ stated in \eqref{eq:gmmRel} to conclude the proof of the lemma
\[
(\gamma^0 \gamma^\theta)^2 =  -\frac12 \sum_{j,k} \theta_j \theta_k  (\gamma^j \gamma^k + \gamma^k \gamma^j) =
\sum \theta_j^2 I_4 = I_4.
\]
\end{proof}

Motivated by this lemma, in order to better describe the energy on hyperboloids it is useful to introduce projectors
\[
P^\theta_{\pm}:= \frac12(I_4 \pm \gamma^0 \gamma^\theta)
\]
on the positive, respectively the negative eigenspaces of $\gamma^0\gamma^{\theta}$. Correspondingly, we split 
\[
\psi = \psi_+ + \psi_{-} := P^\theta_+ \psi + P^\theta_- \psi,
\]
where we can think of the two components as ``outgoing", respectively ``incoming".
Then we can rewrite  the energy density on the hyperboloid  $H$ as 
\begin{equation}\label{eH}
e_H(\psi):= \frac{t-r}{\sqrt{t^2+r^2}} |\psi_+|^2
+ \frac{t+r}{\sqrt{t^2+r^2}} |\psi_-|^2.
\end{equation}
The two components $\psi_{\pm}$ of $\psi$ will play different roles in our decay bounds for the Dirac field.
\bigskip

Another interpretation of the energy density 
on the hyperboloids can be naturally obtained
by using the hyperbolic metric and volume element. The invariant measure on the hyperbolic space is related to the 
above Euclidean measure by
\[
d \sigma = \sqrt{t^2+x^2}\,(t^2-x^2) \, dV_H.
\]
Then the above energy is rewritten in an invariant form as
\[
E_H(\psi) = - \int_H (t^2-x^2)^\frac32 \langle \gamma^0 \gamma^H \psi, \psi \rangle\,  dV_H, \qquad \gamma^H := \frac{x_\alpha
\gamma^\alpha}{\sqrt{t^2-x^2}}.
\]
Here it is natural to introduce the (positive definite) inner product on $\C^4$
\begin{equation}  
\label{eq:DefInnerProductH}
\langle \psi^1,\psi^2 \rangle_H =  - \langle \gamma^0 \gamma^H \psi^1, \psi^2 \rangle.
\end{equation}
Comparing this with \eqref{eH} we can diagonalize this in terms of the $\psi_{\pm}$
decomposition as
\begin{equation*}
\| \psi\|_H^2 =   \frac{t-r}{\sqrt{t^2-r^2}} |\psi_+|^2
+ \frac{t+r}{\sqrt{t^2-r^2}} |\psi_-|^2.  
\end{equation*}

The matrix $\gamma^H$ above will play 
an important role in the sequel. We begin with 

\begin{lem}
The matrix $\gamma^H$ satisfies $(\gamma^H)^2 = I_4$, and has double eigenvalues $\pm 1$.    
\end{lem}

\begin{proof}
The proof is again a  relatively straightforward computation,  using the relations \eqref{eq:gmmRel}, and is left for the reader.
\end{proof}
Based on this property, we introduce 
the new set of projectors 
\begin{equation}
\label{def proj Pv}
2 P^{\pm}_v := I_4 \mp \gamma^H, \qquad v = x/t \in B(0,1). 
\end{equation}
These generate a decomposition of $\C^4$ 
as a direct sum of two subspaces $V^{\pm}$
defined as 
\begin{equation}
\label{eq:V+}
 V^{\pm}_v  :=  \ker P^{\pm}_v.  
\end{equation}
Since $\gamma^H$ is in general not symmetric,
these projectors are no longer orthogonal in the Euclidean setting. However, the $\langle \cdot,\cdot \rangle_H$ inner product turns out instead to be the  one with respect to which the projectors $P_v^{\pm}$ are indeed orthogonal: 

\begin{lem}
\label{lem:VpmOrthogonal}
The subspaces $V^+$ and $V^-$ are orthogonal with respect 
to the $\langle \cdot,\cdot \rangle_H$ inner product, and 
$P^{\pm}$ are the corresponding orthogonal projectors.
\end{lem}
\begin{proof}
The proof is a straightforward calculation which is left for the reader.
\end{proof}

\section{Energy estimates for the linearized equation }\label{s:lin}
In this section we prove the energy estimates for the linearized equation. We obtain them in the energy space $\mathcal{H}^0$. Such estimates are needed in Section~\ref{s:vf}, where vector fields bounds for the solution to \eqref{eq:MD} will be derived. This section will also contain the bootstrap bounds we rely on in getting the pointwise decay bound via Klainermann-Sobolev inequalities in Section~\ref{s:ks}.
The solutions for the linearized  system around a solution $(\psi, A)$ are
denoted by $(\phi, B )$.

Including also source terms, the linearized system takes the form

\begin{equation}
\label{eq:MD-lin}
\left\{
	\begin{aligned}
- &\imu \gamma^\mu \partial_\mu \phi + \phi = \gamma^\mu A_\mu \phi + \gamma^\mu B_\mu \psi + F \\
	 &\Box B_\mu = - \overline{\phi} \gamma_\mu \psi -\overline{\psi} \gamma_\mu \phi + G_\mu.
	 \end{aligned}
  \right.
\end{equation}

We next prove energy estimates for the linearized system~\eqref{eq:MD-lin}. We assume the bootstrap hypothesis
\begin{equation}
\label{eq:boot-psi-A}
\| \psi(t) \|_{L^6} +\Vert A(t)\Vert_{L^{\infty}} \leq  C_0 \epsilon \langle t \rangle^{-1},
\end{equation}
which is consistent with having minimal assumptions on the control norms used in getting these energy estimates. To keep the ideas simple here, we assume this holds in a time interval $[0,T]$. However, our bootstrap argument 
for the full problem will instead be carried out in the regions $C_{<T}$ which we introduce in the next section, see \eqref{def:C<T}.

\begin{prop}
\label{prop:EnergyEstimate}
Assuming the bootstrap bound~\eqref{eq:boot-psi-A} (on $\psi$),
we have the estimate
\begin{align*}
    \|(\phi,B)(t)\|_{\cH^0}^2 &\leq \|(\phi,B)(0)\|_{\cH^0}^2 + \int_0^t C_1 C_0 \epsilon \langle s\rangle^{-1} \|(\phi, B)(s)\|_{\cH^0}^2 \, d s \nonumber\\
    &\qquad + \Big|\Re \int_0^t \int_{\R^3} (\phi \cdot \overline{\imu \gamma^0 F} + |D|^{-1} \partial_t B_\mu \overline{G_\mu}) \, d x d s \Big|.
\end{align*}
In particular, in the case $F = G = 0$, we get the energy estimate
\begin{equation*}
\| (\phi,B)(t)\|_{\H^0}    
\lesssim \langle t \rangle^{c\epsilon}  \| (\phi,B)(0)\|_{\H^0}.
\end{equation*}
\end{prop}

\begin{proof}
    Computing the time derivative of the $\H^0$-norm of $(\phi,B)$, we get
    \begin{align}
    \label{eq:EnergyDer}
        &\frac{1}{2} \partial_t \|(\phi,B)\|_{\H^0}^2 
        = \frac{1}{2} \partial_t (\|\phi\|_{L^2}^2 + \||D|^{\frac{1}{2}} B\|_{L^2}^2 + \||D|^{-\frac{1}{2}} \partial_t B\|_{L^2}^2) \nonumber\\
        &=\Re \int_{\R^3} \phi \cdot \overline{\partial_t \phi}\, d x + \Re \int_{\R^3} |D|^{\frac{1}{2}} B \cdot \overline{|D|^{\frac{1}{2}} \partial_t B}\,  d x \nonumber\\
        &\qquad + \Re \int_{\R^3} |D|^{-\frac{1}{2}} \partial_t B \cdot \overline{|D|^{-\frac{1}{2}} \partial_t^2 B}\, d x \nonumber\\
        &= \Re \int_{\R^3} \phi \cdot \overline{\partial_t \phi}\, d x  + \Re \int_{\R^3} |D| B_\mu \,\overline{ \partial_t B_\mu} \,d x + \Re \int_{\R^3} |D|^{-1} \partial_t B_\mu \,  \overline{\partial_t^2 B_\mu}\, d x \nonumber\\
        &= \Re \int_{\R^3} \phi \cdot \overline{(- \gamma^0 \gamma^j \partial_j \phi - \imu \gamma^0 \phi + \imu \gamma^0 \gamma^\mu A_\mu \phi + \imu \gamma^0 \gamma^\mu B_\mu \psi + \imu \gamma^0 F)}\, d x \nonumber\\
        &\qquad + \Re \int_{\R^3} |D| B_\mu \, \overline{\partial_t B_\mu} \, d x + \Re \int_{\R^3} |D|^{-1} \partial_t B_\mu \, \overline{(\Delta B_\mu - \overline{\phi} \gamma_\mu \psi - \overline{\psi} \gamma_\mu \phi + G_\mu)} \, d x,
    \end{align}
    where we employed~\eqref{eq:MD-lin} in the last step. Using that $\gamma^0 \gamma^j$ is hermitian, an integration by parts yields
    \begin{align*}
        \Re \int_{\R^3} \phi \cdot \overline{(- \gamma^0 \gamma^j \partial_j \phi)} \, d x 
        = \Re \int_{\R^3} \partial_j (\gamma^0 \gamma^j \phi) \cdot \overline{\phi} \, d x
        = \Re \int_{\R^3} \phi \cdot \overline{\gamma^0 \gamma^j \partial_j \phi} \, d x,
    \end{align*}
    and hence
    \begin{equation}
    \label{eq:EnergyDerVan1}
         \Re \int_{\R^3} \phi \cdot \overline{(- \gamma^0 \gamma^j \partial_j \phi)} \, d x = 0.
    \end{equation}
    Since $\phi \cdot \overline{\gamma^0 \phi}$ is real, we also have
    \begin{equation*}
        \Re \int_{\R^3} \phi \cdot \overline{ (- \imu \gamma^0 \phi )}\,  d x = 0.
    \end{equation*}
    Using once again that $\gamma^0 \gamma^\mu$ is hermitian and that $A_\mu$ is real, we further infer
    \begin{align*}
        \Re \int_{\R^3} \phi \cdot \overline{ \imu \gamma^0 \gamma^\mu A_\mu \phi } \, d x 
        = - \Re \int_{\R^3} \imu \gamma^0 \gamma^\mu \phi \cdot \overline{A_\mu \phi} \, d x 
        =  - \Re \int_{\R^3} \imu \gamma^0 \gamma^\mu A_\mu  \phi \cdot \overline{\phi}\, d x ,
    \end{align*}
    and thus
    \begin{equation*}
         \Re \int_{\R^3} \phi \cdot \overline{ \imu \gamma^0 \gamma^\mu A_\mu \phi } \, d x = 0.
    \end{equation*}
    Finally, we observe that
    \begin{align}
    \label{eq:EnergyDerVan4}
        &\Re \int_{\R^3} |D| B_\mu \, \overline{\partial_t B_\mu} \, d x + \Re \int_{\R^3} |D|^{-1} \partial_t B_\mu \, \overline{\Delta B_\mu } \, d x \nonumber \\
        &= \Re \int_{\R^3} |D| B_\mu \, \overline{\partial_t B_\mu}\,  d x + \Re \int_{\R^3}  \partial_t B_\mu \, \overline{|D|^{-1} (-|D|^2) B_\mu } \, d x = 0.
    \end{align}
    Inserting~\eqref{eq:EnergyDerVan1} to~\eqref{eq:EnergyDerVan4} into~\eqref{eq:EnergyDer}, we arrive at
    \begin{align}
        \label{eq:EnergyDerSimplified}
        \frac{1}{2} \partial_t \|(\phi,B)\|_{\H^0}^2 
        &= \Re \int_{\R^3} \phi \cdot \overline{(\imu \gamma^0 \gamma^\mu B_\mu \psi + \imu \gamma^0 F)} \, d x \nonumber \\
        &\qquad + \Re \int_{\R^3} |D|^{-1} \partial_t B_\mu \, \overline{(- \overline{\phi} \gamma_\mu \psi - \overline{\psi} \gamma_\mu \phi + G_\mu)}\,  d x.
    \end{align}
    The Sobolev embedding $\dot{H}^{\frac{1}{2}}(\R^3) \hookrightarrow L^3(\R^3)$ allows us to estimate the first summand in the first integral by
    \begin{align*}
        \Big|\int_{\R^3} \phi \cdot \overline{\imu \gamma^0 \gamma^\mu B_\mu \psi} \,d x  \Big|
        &\lesssim \|\phi\|_{L^2} \|B\|_{L^3} \|\psi\|_{L^6} \lesssim \|\phi\|_{L^2} \|B\|_{\dot{H}^{\frac{1}{2}}} \|\psi\|_{L^6}\\
        & \lesssim \|\psi\|_{L^6} \|(\phi,B)\|_{\H^0}^2.
    \end{align*}
    Similarly, we derive for the second integral in~\eqref{eq:EnergyDerSimplified}
    \begin{align*}
        &\Big| \int_{\R^3} |D|^{-1} \partial_t B_\mu \, \overline{(- \overline{\phi} \gamma_\mu \psi - \overline{\psi} \gamma_\mu \phi)} \,d x \Big|
        \lesssim \||D|^{-1} \partial_t B\|_{L^3} \|\phi\|_{L^2} \|\psi\|_{L^6} \\
        & \qquad \lesssim \|\partial_t B\|_{\dot{H}^{-\frac{1}{2}}} \|\phi\|_{L^2} \|\psi\|_{L^6} \lesssim \|\psi\|_{L^6} \|(\phi,B)\|_{\H^0}^2.
    \end{align*}
    We denote the maximum of the implicit constants in the above two estimates by $C_1$. Combining the these estimates with the bootstrap hypothesis~\eqref{eq:boot-psi-A}, we get
    \begin{align*}
        \|(\phi,B)(t)\|_{\H^0}^2 &= \|(\phi,B)(0)\|_{\H^0}^2 + \int_0^t (\partial_t \|(\phi,B)\|_{\H^0}^2)(s) \,d s \\
       & \leq \|(\phi,B)(0)\|_{\H^0}^2 + \int_0^t C_1 \|\psi(s)\|_{L^6}  \|(\phi,B)(s)\|_{\H^0}^2\, d s \\
       &\qquad + \Big|\Re \int_0^t \int_{\R^3} (\phi \cdot \overline{\imu \gamma^0 F} + |D|^{-1} \partial_t B_\mu \overline{G_\mu})\, d x d s \Big| \\
       & \leq \|(\phi,B)(0)\|_{\H^0}^2 + \int_0^t C_1 C_0 \epsilon \langle s\rangle^{-1}  \|(\phi,B)(s)\|_{\H^0}^2 \, d s \\
       &\qquad + \Big|\Re \int_0^t \int_{\R^3} (\phi \cdot \overline{\imu \gamma^0 F} + |D|^{-1} \partial_t B_\mu \overline{G_\mu})\, d x d s \Big|,
    \end{align*}
    which is the first part of the Proposition. In the case $F = G = 0$,
    Gronwall's inequality then yields
    \begin{align*}
        \|(\phi,B)(t)\|_{\H^0}^2 \leq  \langle t \rangle^{2 c \epsilon} \, \|(\phi,B)(0)\|_{\H^0}^2 ,
    \end{align*}
    where $c = \frac{1}{2} C_0 C_1$. This shows the assertion of the Proposition.
\end{proof}

\section{Vector fields bounds}\label{s:vf}

The main goal of this section is to establish energy bounds for $(\psi,
A)$ and their higher derivatives as well as  energy bounds for the solution $(\psi,
A)$ to which
we have applied a certain number of vector fields admissible to \eqref{eq:MD}. These functions solve a 
system which is closely related to the linearized system which was studied in
Section~\ref{s:lin}, but the vector fields
bring in additional difficulties which 
require a more complex argument.


\subsection{Energy estimates for vector fields}

In the remainder of this section we prove
vector field energy bounds given the pointwise bootstrap assumption \eqref{eq:boot-psi-A}, which for convenience we recall here
\begin{equation*}
\| \psi (t)\|_{L^6} +\| A(t) \|_{L^\infty} \leq C_0 \epsilon \langle t \rangle^{-1}.
\end{equation*}

Ideally, given such a bootstrap assumption in a time interval $[0,T]$, one would like  to prove a vector field energy bound of the form  
\begin{equation}
\label{eq:EnergyEstimateVectorFields0}
\| \Gamma^{\leq k}(\psi,A)(t)\|_{\H^0}   \lesssim \langle t \rangle^{c\epsilon}  \| \Gamma^{\leq k}(\psi,A)(0)\|_{\H^0},
\end{equation}
where $c\approx C_0$.
Here we would like to  use interpolation and Gronwall's inequality as in the previous section. But since vector fields are time dependent,  the interpolation should happen in a space-time setting. Because of this we will no longer be able to apply directly Gronwall's inequality in time; instead it turns out that a dyadic time  decomposition would address the issue; this is similar to  the work of the second author in \cite{ifrims} and relies on ideas of Metcalfe-Tataru-Tohaneanu~\cite{mtt12}.

\begin{figure}[h]
\begin{center}
\begin{tikzpicture}[scale=1.9]

\draw[->] (0,-0.3) -- (0,2.5);
\draw[->] (-2.5,0) -- (2.5,0);
\node[below] at (2.4,0) {\small $x$};
\node[left] at (0,2.4) {\small $t$};

\draw[red, thick]  (-2.5,0.8) -- (2.5,0.8);
\node[below, left] at (0,0.7) {\tiny $T$};
\draw[red,thick]  (-2.5,1.6) -- (-1.058,1.6);
\node[above, left] at (0,1.7) {\tiny $4T$};
\draw[red,thick]  (1.058,1.6) -- (2.5,1.6);
\draw[blue, thick, dashed] (-1.03,1.6) -- (1.058,1.6);
\node[right] at (.9,1.7) {\tiny $B$};
\node[left] at (-.9,1.7) {\tiny $A$};

\node[left] at (.47, 1.4) {\tiny ``cup'' region};

\draw[red,thick] [domain=-1.058:1.058] plot(\x, {((1.2)^2+(\x)^2)^(0.5)});
\draw[dashed] [domain=-2:-1.058] plot(\x, {((1.2)^2+(\x)^2)^(0.5)});
\draw[dashed] [domain=1.058:2] plot(\x, {((1.2)^2+(\x)^2)^(0.5)});
\node[below, left] at (1.8,2.2) {\tiny $H_\rho$};
\node[red] at (-1.5,1.2) {\tiny $C_T$};

\draw[dashed] [domain=0:2.4] plot(\x,\x);
\draw[dashed] [domain=-2.4:0] plot(\x,-\x);
\node[right] at (2.2,2.2) {\tiny $t=|x|$};

\fill [red!40,nearly transparent, domain=-1.058:1.058, variable=\x]
  (-1.058, 0.8)
  -- plot ({\x},{((1.2)^2+(\x)^2)^(0.5)})
  -- (1.058, 0.8)
  -- cycle;

\fill[red!40,nearly transparent] (-2.5, 0.8) -- (-2.5,1.6) -- (-1.058,1.6) -- (-1.058, 0.8) -- cycle;
\fill[red!40,nearly transparent]  (1.058,1.6) -- (1.058, 0.8) -- (2.5, 0.8) -- (2.5, 1.6) -- cycle;

\end{tikzpicture}
\caption{Region $C_T$ in 3+1 space-time dimension; ``cup'' region definition}
\label{f:CT}
\end{center}
\end{figure}
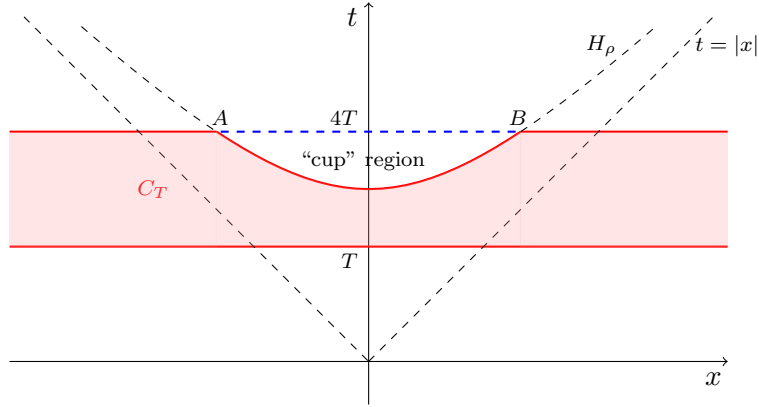


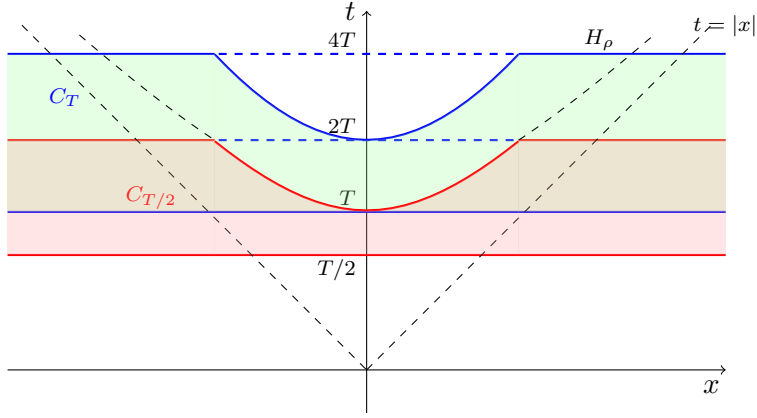
\begin{figure}[h]
\begin{center}
\begin{tikzpicture}[scale=1.9]

\draw[->] (0,-0.3) -- (0,2.5);
\draw[->] (-2.5,0) -- (2.5,0);
\node[below] at (2.4,0) {\small $x$};
\node[left] at (0,2.5) {\small $t$};

\draw[red, thick]  (-2.5,0.8) -- (2.5,0.8);
\node[below, left] at (0,1.2) {\tiny $T$};
\node[below, left] at (0,0.7) {\tiny $T/2$};
\draw[red,thick]  (-2.5,1.6) -- (-1.058,1.6);
\node[above, left] at (0,1.7) {\tiny $2T$};
\node[above, left] at (0,2.3) {\tiny $4T$};
\draw[red,thick]  (1.058,1.6) -- (2.5,1.6);
\draw[blue, thick, dashed] (-1.03,1.6) -- (1.058,1.6);

 \draw[blue, thick]  (-2.5,1.1) -- (2.5,1.1);
\draw[blue,thick]  (-2.5,2.2) -- (-1.058,2.2);
\draw[blue, thick, dashed] (-1.03,2.2) -- (1.058,2.2);
\draw[blue,thick]  (1.058,2.2) -- (2.5,2.2);
\draw[blue,thick] [domain=-1.058:1.058] plot(\x, {((1.4)^2+(\x)^2)^(0.7)});

\fill [green!40,nearly transparent, domain=-1.058:1.058, variable=\x](-1.058, 1.1)
  -- plot ({\x},{((1.4)^2+(\x)^2)^(0.7)})
  -- (1.058, 1.1)
  -- cycle;

\fill[green!40,nearly transparent] (-2.5, 1.1) -- (-2.5,2.2) -- (-1.058,2.2) -- (-1.058, 1.1) -- cycle;

\fill[green!40,nearly transparent]  (1.058,2.2) -- (1.058, 1.1) -- (2.5, 1.1) -- (2.5, 2.2) -- cycle;

\draw[red,thick] [domain=-1.058:1.058] plot(\x, {((1.1)^2+(\x)^2)^(0.55)});

\draw[dashed] [domain=-2:-1.058] plot(\x, {((1.2)^2+(\x)^2)^(0.5)});
\draw[dashed] [domain=1.058:2] plot(\x, {((1.2)^2+(\x)^2)^(0.5)});
\node[below, left] at (1.8,2.3) {\tiny $H_\rho$};
\node[red] at (-1.5,1.2) {\tiny $C_{T/2}$};

\node[blue] at (-2.1,1.9) {\tiny $C_{T}$};

\draw[dashed] [domain=0:2.4] plot(\x,\x);
\draw[dashed] [domain=-2.4:0] plot(\x,-\x);
\node[right] at (2.2,2.4) {\tiny $t=|x|$};

\fill [red!40,nearly transparent, domain=-1.058:1.058, variable=\x]
  (-1.058, 0.8)
  -- plot ({\x},{((1.1)^2+(\x)^2)^(0.55)})
  -- (1.058, 0.8)
  -- cycle;

\fill[red!40,nearly transparent] (-2.5, 0.8) -- (-2.5,1.6) -- (-1.058,1.6) -- (-1.058, 0.8) -- cycle;
\fill[red!40,nearly transparent]  (1.058,1.6) -- (1.058, 0.8) -- (2.5, 0.8) -- (2.5, 1.6) -- cycle;

\end{tikzpicture}
\caption{Overlapping $C_T$  regions}
\label{f:CToverlap}
\end{center}
\end{figure}

Ideally one might want to work in the overlapping dyadic regions $[T,4T] \times \R^3$, 
except that such regions cannot be well foliated by hyperboloids.
So we define instead the regions 
\[
C_T := \{ t \in [T,4T]\,,  \,  t^2 -x^2 \leq 4 T^2\},
\]
and also 
\begin{equation}
\label{def:C<T}
C_{<T} := \{ t \in [0,4T]\,,  \,  t^2 -x^2 \leq 4 T^2\}.
\end{equation}
For interpolation purposes, we will also use a slight enlargement $C_T^+$ of $C_T$ where we  add a lower cap, thus working with the region we define next
\begin{equation}
\label{eq:ct+}
C_T^+ := C_T \cup \{ (t,x) \in [T/2,T] \times \R^3, t^2 - x^2 \geq T^2/4\};
\end{equation}
explicitly, this is a slab with a cap removed on top and   with a similar cup added at the bottom.

Our main bootstrap argument in this paper 
will be carried out exactly in a region of the form $C_{<T}$, where we will assume that the bootstrap assumption \eqref{eq:boot-psi-A} holds with a large constant $C_0$, and then show that we can improve the constant.

The aim of this section is to carry out the first step of this process, namely to prove that, under such a bootstrap assumption in $C_{<T}$, we can obtain vector field energy estimates in the same region $C_{<T}$. For this, the strategy will be to inductively 
prove the vector field bounds in the (overlapping) regions $C_T$
for dyadic $T$, losing a $1+C\epsilon$ factor at every step.

To measure our solutions in $C_T$, it is 
convenient to introduce a stronger norm which does not contain only the energy, but also the size of the source terms 
in the corresponding linear equation.
Thus we define the following norms in the regions $C_T$:
\begin{equation}
\label{XT-def}
\|(\psi,A)\|_{X_T} := \|(\psi,A)(T)\|_{\cH^0} + 
T^\frac12 \| - \imu \gamma^\mu \partial_\mu \psi + \psi    \|_{L^2(C_T)} + T^{\frac13} \| \Box A\|_{L^\frac32(C_T)}.
\end{equation}
Now we are ready to state our main energy estimates:
\begin{prop}
\label{prop:EnergyEstGamma}
Assume that the bootstrap bounds  \eqref{eq:boot-psi-A}, recalled above, hold in $C_{<T}$.
Then in any dyadic region $C_{T_1} \subset C_{<T}$ we have the energy estimates 
\begin{equation*}
\| \Gamma^{\leq k}(\psi,A)\|_{X_{T_1}}   \lesssim \langle T_1 \rangle^{c\epsilon}  \| \Gamma^{\leq k}(\psi,A)(0)\|_{\H^0}, \qquad T_1 \leq T,
\end{equation*}
holding for a total of $k = 9$ vector fields and derivatives, with $c\approx C_0$.
\end{prop}
As we will see immediately in the next subsection, this in particular implies the fixed time energy bounds in \eqref{eq:EnergyEstimateVectorFields0}, 
but also provides additional information which we will use later on for the Klainerman-Sobolev inequality.

The rest of this section is 
devoted to the proof of this proposition. We split the analysis in three parts. First we derive the bounds if only translation vector fields are applied, then if only $\Omega$ (this is a shorthand notation for $\Omega_{\alpha \beta}$ which we will frequently use throughout) vector fields are applied, and finally if a mix of translation and $\Omega$  vector fields is applied.

\subsection{ Linear energy estimates in $C_T$.}
Our preliminary task here is to discuss energy estimates for the linear homogeneous system
\begin{equation}
\label{eq: linear system}
\left\{
\begin{aligned}
&(- \imu \gamma^\mu \partial_\mu + 1) \psi  = F\\
&\Box A = G
\end{aligned}
\right.
\end{equation}
in the cup regions $C_T$, which in particular imply that the fixed time energy bounds in \eqref{eq:EnergyEstimateVectorFields0} follow from Proposition~\ref{prop:EnergyEstGamma}.
These energy estimates are as follows:
\begin{lem}\label{lem:LinearEnergyEstCT} Assume $(\psi,A)$ solve the linear homogeneous system~\eqref{eq: linear system}.
	\begin{enumerate}
		\item 
			For  the spinor $\psi$ we have
			\begin{equation}
				\label{eq:EnergyEstLinearDiracCT}
					\|\psi\|_{L^\infty_t L^2_x(C_T)} \lesssim \|\psi(T)\|_{L^2} + \|F\|_{L^1_t L^2_x(C_T)}.
			\end{equation}
		\item 
			For the electromagnetic field $A$ we  have
			\begin{equation}
   \label{eq:EnergyEstLinearWaveCT-full}
				\|\nabla_{t,x} A\|_{L^\infty_t \dot{H}^{-\frac{1}{2}}_x(C_T)} \lesssim \|\nabla_{t,x} A(T)\|_{\dot{H}^{-\frac{1}{2}}} + \|G\|_{L^1_t L^{\frac{3}{2}}_x(C_T)}.
			\end{equation}
	\end{enumerate}
\end{lem}

Some comments are useful here concerning 
the uniform energy bound for $A$ in \eqref{eq:EnergyEstLinearWaveCT-full}.
This is a nonlocal norm, whose interpretation at times $t \in [T,4T]$ depends on $t$. This is the full $\cH_0$ norm if $t \in [T,2T]$
(which is what allows us to easily advance time in the continuity argument later on), but if $t \in [2T,4T]$ this is only an $\cH_0$ norm in a punctured region with a ball removed. This is defined as the best norm of an extension
to the full $\R^3$, which is less descriptive. However, by applying Sobolev embeddings in all of $\R^3$ for this extensions we obtain as a corollary the local bound 
\begin{equation}
		\label{eq:EnergyEstLinearWaveCT}
				\|A\|_{L^\infty_t L^3_x(C_T)} \lesssim \|\nabla_{t,x} A(T)\|_{\dot{H}^{-\frac{1}{2}}} + \|G\|_{L^1_t L^{\frac{3}{2}}_x(C_T)}.
			\end{equation}
This weaker bound, rather than  \eqref{eq:EnergyEstLinearWaveCT-full}, is exactly what is needed and used in order to close the proof of Proposition~\ref{prop:EnergyEstGamma}.

\begin{proof}
	Let $s \in [T, 4T]$ and let $H_s$ denote the part of the hyperboloid $t^2 - |x|^2 = 4 T^2$ with $t \leq s$ and $C_{T,s } := \{(t,x) \in C_T \colon t = s\}$. Finally, we write $\Sigma_s := C_{T,s} \cup H_s$. If $s \leq 2T$, $\Sigma_s = \{(s,x) \colon x \in \R^3 \}$ and the analysis simplifies. Throughout  this section, we thus assume without loss of generality that  $s > 2T$.
	
	For the first part we write the $L^2$-conservation of the Dirac equation in the density-flux form
	\begin{align*}
		\partial_t |\psi|^2 + \partial_j ( \psi^\dag  \gamma^0 \gamma^j \psi ) = -2 \Im (\psi^\dag \gamma^0 F ),
	\end{align*}
	and integrate this identity over the domain, denoted by $D_s$, between $\{t = T\}$ and $\Sigma_s$, leading to
	\begin{equation}
		\int_{C_{T,s}} |\psi(s)|^2\, dx - \int_{\R^3} |\psi(T)|^2 \, d x + \int_{H_s} (\nu_0 |\psi|^2 + \nu_j \psi^\dag \gamma^0 \gamma^j \psi) \, d \sigma  = -2 \int_{D_s} \Im(\psi^\dag \gamma^0 F) \, dx d t. \label{eq:EnergyIdentDirac}
	\end{equation}
	Employing Lemma~\ref{lem:gammatheta}, we have
	\begin{align*}
		\int_{H_s} (\nu_0 |\psi|^2 + \nu_j \psi^\dag \gamma^0 \gamma^j \psi)\,  d \sigma \geq 0
	\end{align*}
	as the unit outer normal $\nu$ of the hyperboloid is time-like. Discarding this term in~\eqref{eq:EnergyIdentDirac}, we obtain
	\begin{align*}
		\int_{C_{T,s}} |\psi(s)|^2 \, dx \lesssim \|\psi(T)\|_{L^2}^2 + \|\psi\|_{L^\infty_t L^2_x(D_s)} \|F\|_{L^1_t L^2_x(D_s)}.
	\end{align*}
	Taking the supremum over $s \in [T,4T]$, we arrive at 
	\begin{align*}
		\|\psi\|_{L^\infty_t L^2_x(C_T)} \lesssim \|\psi(T)\|_{L^2} + \|F\|_{L^1_t L^2_x(C_T)}.
	\end{align*}
	
	For the wave part we proceed similarly.
	Let $\tilde{G}$ denote an extension of $G$ to the strip $[T, 4T] \times \R^3$ with norm on the full strip bounded by the norm on of $G$ on $C_T$, see Lemma \ref{lem:Extension}. As $C_T$ is a domain of determination, the corresponding solution $\tilde{A}$ of the linear wave equation provides an extension of $A$ to the full strip. Applying a Littlewood-Paley decomposition, the standard energy inequality for the linear wave equation yields
	\begin{align*}
		\lambda^2 \|\tilde{A}_\lambda(s)\|_{L^2}^2 +  \|\partial_t \tilde{A}_\lambda(s)\|_{L^2}^2 &\lesssim \lambda^2 \|A_\lambda(T)\|_{L^2}^2 +  \|\partial_t A_\lambda(T)\|_{L^2}^2 \\
		&\qquad + \int_{T}^s \|\partial_t \tilde{A}_\lambda(t)\|_{L^2} \|G_\lambda(t)\|_{L^2}\, dt,
	\end{align*}
	where the subscript $\lambda$ denotes the part of the function with frequency localized around $\lambda \in 2^\Z$. Dividing by $\lambda$ and summing up, we get
	\begin{align*}
		\| \tilde{A}(s)\|_{\dot{H}^{\frac{1}{2}}}^2 + \|\partial_t \tilde{A}(s)\|_{\dot{H}^{-\frac{1}{2}}}^2
		\lesssim \| \nabla_{t,x} A(T)\|_{\dot{H}^{-\frac{1}{2}}}^2 + \|\partial_t \tilde{A} \|_{L^\infty_t \dot{H}^{-\frac{1}{2}}_x} \|\tilde{G}\|_{L^1_t \dot{H}^{-\frac{1}{2}}_x},
	\end{align*}
	where the space-time norms are taken over $[T,s] \times \R^3$. Taking the supremum over $s \in [T, 4T]$, an application of Young's inequality and the embeddings $\dot{H}^{\frac{1}{2}} \hookrightarrow L^3$ and $L^{\frac{3}{2}} \hookrightarrow \dot{H}^{-\frac{1}{2}}$ yield
	\begin{align*}
		\|\nabla_{t,x} A\|_{L^\infty_t \dot H^{-\frac12}_x(C_T)} &\lesssim \|\nabla_{t,x} \tilde{A}\|_{L^\infty_t \dot{H}^{-\frac{1}{2}}_x([T,4T] \times \R^3)} \lesssim \| \nabla_{t,x} A(T)\|_{\dot{H}^{-\frac{1}{2}}} +  \|\tilde{G}\|_{L^1_t L^{\frac{3}{2}}_x([T,4T]\times \R^3)} \\
		&\lesssim \| \nabla_{t,x} A(T)\|_{\dot{H}^{-\frac{1}{2}}} +  \|G\|_{L^1_t L^{\frac{3}{2}}_x(C_T)}.
	\end{align*}
\end{proof}

\begin{rem}
    Being only interested in the energy estimates at this stage, we have simply discarded the integral over the hyperboloid in~\eqref{eq:EnergyIdentDirac}. In Section~\ref{ss:vector fields}, however, the control this integral provides will be crucial in order to recover the pointwise bounds.
\end{rem}

\subsection{Bounds for derivatives}
\label{ss:higher derivatives}

Let $I \in \N_0^4$. Applying $\partial^I$ to the Maxwell-Dirac system~\eqref{eq:MD}, we infer that $(\partial^I \psi, \partial^I A)$ solves the linear system~\eqref{eq: linear system} with source terms
\begin{align}
    &F = F_I := \sum_{0 \leq I' \leq I} \binom{I}{I'} \gamma^\mu \partial^{I'} A_\mu \partial^{I - I'} \psi, \label{eq:DefFalpha}\\
    &G_\mu = G_{\mu,I} := - \sum_{0 \leq I' \leq I} \binom{I}{I'} \partial^{I'} \overline{\psi} \gamma^\mu \partial^{I - I'} \psi. \label{eq:DefGmualpha}
\end{align}

Next, using Gagliardo-Nirenberg estimates in space-time, we derive energy estimates for $(\partial^I \psi, \partial^I A)$.

\begin{lem}
    \label{lem:EnergyEstHigherDeriv}
    Assuming the bootstrap hypothesis~\eqref{eq:boot-psi-A} in $C_{<T}$, we obtain
    \begin{equation}
        \label{eq:EnergyEstHigherDeriv}
        \|\partial^k (\psi, A) \|_{X_{T_1}} \lesssim \langle T_1\rangle^{c \epsilon} \| \partial^{\leq k} (\psi, A)(0)\|_{\cH^0}
    \end{equation}
    for all $k \in \N$, $T_1 \leq T$  and  $c\approx C_0$.
\end{lem}

\begin{proof}
We prove the lemma using a discrete 
version of Gronwall's inequality, starting from the following bound in 
a region  of the  form $C_T$:
\begin{equation}\label{eq:FG-deriv}
T_1^{\frac12}\| F_I\|_{L^2_{t,x}(C_{T_1})} + T_1^{\frac13}\|G_I\|_{L^\frac32_{t,x}(C_{T_1})}
\lesssim C_0 \epsilon  \| \partial^{\leq k} (\psi,A)\|_{X_{T_1}}, \qquad |I| \leq k.
\end{equation}
We first show that this bound implies
\eqref{eq:EnergyEstHigherDeriv}. 

\medskip

By definition of the $X_{T_1}$-norm, the above bound implies that 
\begin{equation}\label{eq:XT-deriv-CT}
 \| \partial^{\leq k} (\psi,A)\|_{X_{T_1}} \lesssim \|\partial^{\leq k} (\psi,A)(T_1)\|_{\H^0}, 
\end{equation}
and returning to \eqref{eq:FG-deriv},
\[
T_1^{\frac12}\| F_I\|_{L^2_{t,x}(C_{T_1})} + T_1^{\frac13}\|G_I\|_{L^\frac32_{t,x}(C_{T_1})} \lesssim C_0\epsilon \|\partial^{\leq k} (\psi,A)(T_1)\|_{\H^0}.
\]
Then direct energy estimates for $(\psi,A)$, again as in Lemma~\ref{lem:LinearEnergyEstCT},
give 
\begin{equation}\label{eq:Gronwall-CT}
\|\partial^{\leq k} (\psi,A)(2T_1)\|_{\H^0} \leq (1+ c\epsilon) \|\partial^{\leq k} (\psi,A)(T_1)\|_{\H^0}, \qquad c \approx C_0.
\end{equation}
This last inequality forms the basis
for our discrete Gronwall argument. For any $t \in [1,T]$, we fix $l \in \N$ with $2^{l-1} \leq t < 2^l$ and set $t_1 = 2^{-l} t \in [0,1)$. Iterating the above estimate, we infer
\begin{align*}
    \|\partial^{\leq k} (\psi,A)(t)\|_{\H^0} \leq (1 + c\epsilon)^l \|\partial^{\leq k} (\psi,A)(t_1)\|_{\H^0}.
\end{align*}
Since $(1 + c\epsilon)^l \lesssim t^{c' \epsilon}$ with $c' \approx c \approx C_0$ and $\|\partial^{\leq k} (\psi,A)\|_{L^\infty_t([0,1])\H^0} \lesssim \|\partial^{\leq k} (\psi,A)(0)\|_{\H^0}$ by local well-posedness, we arrive at
\[
\|\partial^{\leq k} (\psi,A)(t)\|_{\H^0} \lesssim \langle t\rangle^{c\epsilon} \|\partial^{\leq k} (\psi,A)(0)\|_{\H^0}, \qquad 0 \leq t \leq T
\]
with a constant $c \approx C_0$. Together with \eqref{eq:XT-deriv-CT}, this implies the desired estimate \eqref{eq:EnergyEstHigherDeriv}.

\medskip

It now remains to prove the bound \eqref{eq:FG-deriv}. Given the bootstrap norms in \eqref{eq:boot-psi-A} and Lemma \ref{lem:LinearEnergyEstCT}, by H\"older's inequality in time it suffices to show that we have
\begin{equation}\label{eq:FI}
\|F_I\|_{L^2_{t,x}(C_{T_1})} \lesssim \| A\|_{L^\infty_{t,x}(C_{T_1})} \|\partial^{\leq k}\psi \|_{L^2_{t,x}(C_{T_1})}+ \| \psi\|_{L^6_{t,x}(C_{T_1})} \| \partial^{\leq k} A \|_{L^3_{t,x}(C_{T_1})},    
\end{equation}
respectively
\begin{equation}\label{eq:GI}
\|G_I\|_{L^\frac32_{t,x}(C_{T_1})} \lesssim \| \psi\|_{L^6_{t,x}(C_{T_1})} \|\partial^{\leq k}\psi \|_{L^2_{t,x}(C_{T_1})}.   
\end{equation}
We prove this separately for all the 
terms in $F_I$ and $G_I$. It suffices to assume that $|I|=k$.
This is obvious if $I'= 0$ and if 
$I'=I$. For intermediate values of $I'$ we denote $|I'|= j$, and define exponents $p_1$ and $p_2$ by
    \begin{align*}
        \begin{aligned}
        \frac{1}{p_1} &= \frac{j}{3k} = \frac{j}{4} + \theta_1 \Big(\frac{1}{3} - \frac{k}{4} \Big), \qquad &\theta_1 &=  \frac{j}{k}, \\
        \frac{1}{p_2} &= \frac{k-j}{3k} + \frac{1}{6} = \frac{k-j}{4} + \theta_2 \Big( \frac{1}{2} - \frac{k}{4} \Big) + \frac{1 - \theta_2}{6}, \qquad &\theta_2 &= \frac{k-j}{k}. 
        \end{aligned}
    \end{align*}
    Then $\theta_1 + \theta_2 = 1$, $\frac{1}{p_1} + \frac{1}{p_2} = \frac{1}{2}$, and the classical Gagliardo-Nirenberg estimates applied in $C_{T_1}$ yield
 \begin{align*}
        \|\partial^{I'} A_\mu \partial^{I - I'} \psi\|_{L^2_{t,x}(C_{T_1})} \leq & \ \|\partial^{I'} A_\mu\|_{L^{p_1}_{t,x}(C_{T_1})} \|\partial^{I - I'} \psi\|_{L^{p_2}_{t,x}(C_{T_1})}
        \\ 
        \lesssim & \  \|\partial^{\leq k} A\|_{L^3_{t,x}(C_{T_1})}^{\theta_1} \|A\|_{L^\infty_{t,x}(C_{T_1})}^{1-\theta_1} \|\partial^{\leq k} \psi\|_{L^2_{t,x}(C_{T_1})}^{\theta_2} \|\psi\|_{L^6_{t,x}(C_{T_1})}^{1-\theta_2} \\
    \lesssim & \ 
       \| A\|_{L^\infty_{t,x}(C_{T_1})} \|\partial^{\leq k}\psi \|_{L^2_{t,x}(C_{T_1})}+ \| \psi\|_{L^6_{t,x}(C_{T_1})} \| \partial^{\leq k} A \|_{L^3_{t,x}(C_{T_1})}         
    \end{align*}
as needed.

The bound for the corresponding term in $G_I$ is similar.
This time we use the exponents
\begin{align*}
    \begin{aligned}
    \frac{1}{p_1} &= \frac{j}{3k} + \frac{1}{6} =\frac{j}{4} + \theta_1 \Big(\frac{1}{2} - \frac{k}{4} \Big) + \frac{1-\theta_1}{6}, \qquad &\theta_1 &= \frac{j}{k}, \\
    \frac{1}{p_2} &= \frac{1}{2}- \frac{j}{3k} = \frac{k-j}{4} + \theta_2 \Big(\frac{1}{2} - \frac{k}{4}\Big) + \frac{1-\theta_2}{6}, &\theta_2 &= \frac{k-j}{k},
    \end{aligned}
\end{align*}
which satisfy $\frac{1}{p_1} + \frac{1}{p_2} = \frac{2}{3}$, and again we have $\theta_1 + \theta_2 = 1$. The Gagliardo-Nirenberg estimates and the bootstrap hypothesis~\eqref{eq:boot-psi-A} thus yield
\begin{align*}
    \|\partial^{I'} \overline{\psi} \gamma^\mu \partial^{I - I'} \psi\|_{L^{\frac{3}{2}}_{t,x}(C_{T_1})}
    \lesssim & \ \|\partial^{I'} \psi \|_{L^{p_1}_{t,x}(C_{T_1})} \|\partial^{I - I'} \psi \|_{L^{p_2}_{t,x}(C_{T_1})}
    \\
    \lesssim & \ \|\partial^{\leq k}\psi \|_{L^2_{t,x}(C_{T_1})}^{\theta_1} \|\psi\|_{L^6_{t,x}(C_{T_1})}^{1-\theta_1} \|\partial^{\leq k} \psi \|_{L^2_{t,x}(C_{T_1})}^{\theta_2} \|\psi\|_{L^6_{t,x}(C_{T_1})}^{1-\theta_2} \\
    \lesssim & \ \|\partial^{\leq k} \psi\|_{L^2_{t,x}(C_{T_1})} \|\psi\|_{L^6_{t,x}(C_{T_1})},
\end{align*}
again as needed. This concludes the proof of \eqref{eq:FI} and \eqref{eq:GI} and thus the proof of the lemma.
\end{proof}

\begin{rem}
    \label{rem:EnergyEstNumberDer}
    We provided the energy estimates for an arbitrary number of derivatives, though we only use a maximum of three vector fields or nine regular derivatives.
\end{rem}

\subsection{Bounds for vector fields}
\label{ss:vector fields}

Here again we want to use some sort of interpolation inequalities, but since vector fields are time dependent  the interpolation should happen in a space-time setting.

In order to estimate the source terms in the proof of the energy estimates for the $\Omega$ vector fields (again omitting hat and tilde notation), the following Gagliardo-Nirenberg type interpolation result for vector fields is crucial. As already mentioned above, this interpolation has to happen in a space-time region since the vector fields are time dependent.
\begin{lem}[Higher order interpolation for vector fields]\label{lem:high-ord}
    Let $0 \leq j \leq m$ be integers and $\theta = \frac{j}{m}$. Let $p,q,r \in [1,\infty]$ satisfy
    \begin{align*}
        \frac{1}{p} = \frac{\theta}{r} + \frac{1-\theta}{q}.
    \end{align*}
    Then, for $|J|=j$,
    \begin{align}
        \label{eq:EstVectorFieldCT}
        \|\Omega^J u\|_{L^p_{t,x}(C_T^+)} \lesssim \|\Omega^{\leq m} u\|_{L^r_{t,x}(C_T^+)}^\theta \|u\|_{L^q_{t,x}(C_T^+)}^{1-\theta}.
    \end{align}
\end{lem}

\begin{proof}
    Using appropriate hyperbolic coordinates, the statement reduces to standard Gagliardo-Nirenberg estimates.
\end{proof}

We are now ready to prove Proposition~\ref{prop:EnergyEstGamma} for the $\Omega$ vector fields.
\begin{lem}
	\label{lem:EnergyEstOmegavf}
	Assuming the bootstrap hypothesis~\eqref{eq:boot-psi-A} in $C_{<T}$, we obtain
    \begin{equation}
        \label{eq:EnergyEstHighervf}
        \|\Omega^J (\psi, A) \|_{X_{T_1}} \lesssim \langle T_1\rangle^{c \epsilon} \| \Omega^{\leq k} (\psi, A)(0)\|_{\cH^0}
    \end{equation}
    for all $|J|=k \in \N$ with $T_1\leq T$ and $c\approx C_0$.
\end{lem}

\begin{proof}
Here we rely  again on the estimates for the linear system  \eqref{eq: linear system}. We directly apply $\Omega^J$ with a multi-index $|J| = k$ to obtain
\begin{equation}
\left\{
	\begin{aligned}
	(&- \imu \gamma^\mu \partial_\mu + 1)\, \Omega^J \psi  = F_J  \\
	 &\Box \Omega^J A_\mu = G_J, \\
	 \end{aligned}
  \right.
\end{equation}
with source terms
\begin{align*}
	F_J &= \sum_{J_1 \leq J} \binom{J}{J_1} \Omega^{J_1} A_\mu \gamma^\mu \Omega^{J - J_1} \psi, \qquad
	G_J = - \sum_{J_1 \leq J} \binom{J}{J_1} \Omega^{J_1} \overline{\psi} \gamma_\mu \Omega^{J - J_1} \psi.
\end{align*}
We will estimate the source terms in terms of the energy, i.e.,
\begin{align}
\label{est perturb 1}
& T_1^{\frac12}\|F_J\|_{L^2_{t,x}(C_{T_1})} \lesssim  \epsilon  \, (\| \Omega^{\leq k} \psi\|_{L^\infty_t L^2_x(C_{< T_1})} + \|\Omega^{\leq k} A\|_{L^\infty_t L^3_x(C_{< T_1})} ),
\\
\label{est perturb 2}
&T_1^{\frac13}\| G_J \|_{L^\frac32_{t,x} (C_{T_1})}  \lesssim \epsilon \, \| \Omega^{\leq k} \psi\|_{L^\infty_t L^2_x (C_{< T_1})}.
\end{align}

Another important remark concerns the use of $C_{T_1}$ and $C_{< T_1}$ slabs in estimates \eqref{est perturb 1} and \eqref{est perturb 2}: on the LHS we use the $C_{T_1}$ region but in order to be able to use interpolation inequalities on hyperboloids  in the proof, we also need to add a cup region under $C_{T_1}$, see the interpolation Lemma~\ref{lem:high-ord}. To address this issue we are working instead with $C_{< T_1}$ on the right; one could also equivalently work with $C_{T_1} \cup C_{T_1/2}$.

 Using the linear energy estimates from Lemma~\ref{lem:LinearEnergyEstCT} for every $|J| \leq k$, we deduce
\begin{align*}
\| \Omega^{\leq k} \psi \|_{L^\infty_t L^2_x (C_{T_1})} + \| \Omega^{\leq k} A \|_{L^\infty_t L^3_x(C_{T_1})} &\lesssim  \| \Omega^{\leq k} \psi(T_1) \|_{L^2} + \| \nabla_{t,x} \,\Omega^{\leq k} A (T_1) \|_{\dot H^{-\frac12}} \\
& \qquad + \sum_{|J| \leq k} (\|F_J\|_{L^1_t L^2_x(C_{T_1})}
+ \| G_J \|_{L^1_t L^\frac32_x (C_{T_1})}).
\end{align*}
Assuming the bounds~\eqref{est perturb 1} and~\eqref{est perturb 2} for a moment, we obtain with H{\"o}lder's inequality in time
\begin{align}
\label{eq:VecFieldEnergyEstCT}
&\| \Omega^{\leq k} \psi \|_{L^\infty_t L^2_x (C_{T_1})} + \| \Omega^{\leq k} A \|_{L^\infty_t L^3_x(C_{T_1})}\lesssim \| \Omega^{\leq k} \psi(T_1) \|_{L^2} + \| \nabla_{t,x} \,\Omega^{\leq k} A (T_1) \|_{\dot H^{-\frac12}}.
\end{align}
Plugging this estimate into~\eqref{est perturb 1} and~\eqref{est perturb 2}, we infer
\begin{align*}
T_1^{\frac12}\|F_J\|_{L^2_{t,x}(C_{T_1})} + T_1^{\frac13}\| G_J \|_{L^\frac32_{t,x} (C_{T_1})}
\lesssim \epsilon (\| \Omega^{\leq k} \psi(T_1) \|_{L^2} + \| \nabla_{t,x} \, \Omega^{\leq k} A(T_1)  \|_{\dot H^{-\frac{1}{2}}}).
\end{align*}
Now we use the linear energy estimates in the strip $[T_1,2T_1] \times \R^3$, cf. the proof of Lemma~\ref{lem:LinearEnergyEstCT},
to arrive at
\begin{align*}
&\| \Omega^{\leq k} (\psi,A) \|_{L^\infty \cH^0([T_1,2T_1]\times \R^3)} \leq  \| \Omega^{\leq k} (\psi,A)(T_1) \|_{\cH^0} \\
&\qquad + C \sum_{|J| \leq k}(\|F_J\|_{L^1_t L^2_x([T_1,2T_1] \times \R^3)} + \|G_J\|_{L^1_t L^{\frac{3}{2}}_x([T_1,2T_1] \times \R^3)}) \\
 &\leq (1 + C \epsilon) \| \Omega^{\leq k} (\psi,A)(T_1)\|_{\cH^0}.
\end{align*}
Iterating this argument and applying a discrete Gronwall as in the proof of Lemma~\ref{lem:EnergyEstHigherDeriv}, we obtain
\begin{align*}
	\| \Omega^{\leq k} (\psi,A)(t) \|_{\cH^0} \lesssim  \langle t\rangle^{c\epsilon} \| \Omega^{\leq k} (\psi,A)(0)\|_{\cH^0}, \qquad 1 \leq t \leq T,
\end{align*}
which implies the assertion of the lemma in view of~\eqref{est perturb 1} to~\eqref{eq:VecFieldEnergyEstCT} and the definition of the $X_{T_1}$-norm. It only remains to prove~\eqref{est perturb 1} and~\eqref{est perturb 2}.

After switching to space-time norms with the same exponents for space and time, we can basically proceed as in the case for regular derivatives due to the interpolation Lemma~\ref{lem:high-ord}. Take a multiindex $J$ with $|J| = k$ and a multiindex $J_1$ with $J_1 \leq J$. We set $j = |J_1|$ and note that $|J - J_1| = k - j$. Setting
\begin{align*}
	\frac{1}{p_1} = \frac{j}{3k}, \qquad \text{and} \qquad \frac{1}{p_2} = \frac{k-j}{3k} + \frac{1}{6} = \frac{k-j}{k}\cdot \frac{1}{2} + \frac{j}{k} \cdot \frac{1}{6},
\end{align*}
we have $\frac{1}{p_1} + \frac{1}{p_2} = \frac{1}{2}$, and hence
\begin{align*}
	 T_1^{\frac{1}{2}} \|\Omega^{J_1} A_\mu \gamma^\mu \Omega^{J - J_1} \psi \|_{L^2_{t,x}(C_{T_1})} \lesssim T_1^{\frac{1}{2}} \|\Omega^{J_1} A\|_{L^{p_1}_{t,x}(C_{T_1})} \| \Omega^{J - J_1} \psi \|_{L^{p_2}_{t,x}(C_{T_1})}.
\end{align*}
Applying Lemma~\ref{lem:high-ord} and the bootstrap hypothesis~\eqref{eq:boot-psi-A}, we obtain
\begin{align*}
	\|\Omega^{J_1} A\|_{L^{p_1}_{t,x}(C_{T_1})} &\lesssim \|\Omega^{\leq k} A\|_{L^3_{t,x}(C_{T_1})}^{\frac{j}{k}} \|A\|_{L^\infty_{t,x}(C_{T_1})}^{1-\frac{j}{k}} \\
	&\lesssim T_1^{\frac{1}{3} \frac{j}{k}} \|\Omega^{\leq k} A\|_{L^\infty_t L^3_x(C_{T_1})}^{\frac{j}{k}} (C_0 \epsilon T_1^{-1})^{1 - \frac{j}{k}},
\end{align*}
and
\begin{align*}
	\| \Omega^{J - J_1} \psi \|_{L^{p_2}_{t,x}(C_{T_1})} &\lesssim \| \Omega^{\leq k} \psi \|_{L^2_{t,x}(C_{T_1})}^{1 - \frac{j}{k}} \|\psi\|_{L^6_{t,x}(C_{T_1})}^{\frac{j}{k}} \\
	&\lesssim T_1^{\frac{1}{2}(1 - \frac{j}{k})} \|\Omega^{\leq k} \psi \|_{L^\infty_t L^2_x(C_{T_1})}^{1 - \frac{j}{k}} T_1^{\frac{1}{6} \frac{j}{k}} (C_0 \epsilon T_1^{-1})^{\frac{j}{k}}.
\end{align*}
Plugging the last two estimates into the previous one, we arrive at
\begin{align*}
	 T_1^{\frac{1}{2}} \|\Omega^{J_1} A_\mu \gamma^\mu \Omega^{J - J_1} \psi \|_{L^2_{t,x}(C_{T_1})} 
	&\lesssim C_0 \epsilon  \|\Omega^{\leq k} A\|_{L^\infty_t L^3_x(C_{T_1})}^{\frac{j}{k}} \|\Omega^{\leq k} \psi \|_{L^\infty_t L^2_x(C_{T_1})}^{1 - \frac{j}{k}} \\
	&\lesssim C_0 \epsilon (\|\Omega^{\leq k} \psi \|_{L^\infty_t L^2_x(C_{T_1})} + \|\Omega^{\leq k} A\|_{L^\infty_t L^3_x(C_{T_1})}).
\end{align*}
Summing over $J_1 \leq J$ yields~\eqref{est perturb 1}.

To prove~\eqref{est perturb 2}, we proceed analogously. With the same notation as above, we set
\begin{align*}
	\frac{1}{p_1} = \frac{j}{k} \cdot \frac{1}{2} + \frac{k-j}{k} \cdot \frac{1}{6} \qquad \text{and} \qquad 
	\frac{1}{p_2} = \frac{k-j}{k} \cdot \frac{1}{2} + \frac{j}{k} \cdot \frac{1}{6}.
\end{align*}
We then have $\frac{1}{p_1} + \frac{1}{p_2} = \frac{2}{3}$ and thus
\begin{align*}
	T_1^{\frac13}\|\Omega^{J_1} \overline{\psi} \gamma_\mu \Omega^{J - J_1} \psi\|_{L^{\frac{3}{2}}_{t,x}(C_{T_1})} \lesssim T_1^{\frac{1}{3}} \|\Omega^{J_1} \psi\|_{L^{p_1}_{t,x}(C_{T_1})} \| \Omega^{J - J_1} \psi \|_{L^{p_2}_{t,x}(C_{T_1})}.
\end{align*}
Applying Lemma~\ref{lem:high-ord} and the bootstrap hypothesis~\eqref{eq:boot-psi-A} once more, we get
\begin{align*}
	 \|\Omega^{J_1} \psi\|_{L^{p_1}_{t,x}(C_{T_1})} &\lesssim \| \Omega^{\leq k} \psi \|_{L^{2}_{t,x}(C_{T_1})}^{\frac{j}{k}} \|\psi\|_{L^6_{t,x}(C_{T_1})}^{1 - \frac{j}{k}} \\
	 &\lesssim T_1^{\frac{1}{2} \frac{j}{k}} \|\Omega^{\leq k} \psi \|_{L^\infty_t L^2_x(C_{T_1})}^{\frac{j}{k}} T_1^{\frac{1}{6} (1-\frac{j}{k})} (C_0 \epsilon T_1^{-1})^{1-\frac{j}{k}},
\end{align*}
and analogously
\begin{align*}
	\| \Omega^{J - J_1} \psi \|_{L^{p_2}_{t,x}(C_{T_1})} \lesssim T_1^{\frac{1}{2} (1-\frac{j}{k})} \|\Omega^{\leq k} \psi \|_{L^\infty_t L^2_x(C_{T_1})}^{1-\frac{j}{k}} T_1^{\frac{1}{6} \frac{j}{k}} (C_0 \epsilon T_1^{-1})^{\frac{j}{k}}.
\end{align*}
The last two estimates combined with the previous one yield
\begin{align*}
	T_1^{\frac13}\|\Omega^{J_1} \overline{\psi} \gamma_\mu \Omega^{J - J_1} \psi\|_{L^{\frac{3}{2}}_{t,x}(C_{T_1})} \lesssim C_0 \epsilon \|\Omega^{\leq k} \psi\|_{L^\infty_t L^2_x(C_{T_1})}.
\end{align*}
Summing over $J_1$ with $J_1 \leq J$, we arrive at~\eqref{est perturb 2}.
\end{proof}

\begin{rem}
    \label{rem:EnergyEstNumberVectorFields}
    We point out that we provide the above energy estimates for an arbitrary number of vector fields, but we will use it only for $k \leq 3$ in order to keep the regularity and decay assumptions in our main results as low as possible.
\end{rem}

\subsection{Bounds for a mix of derivatives and vector fields}
\label{ss:mixed}

For the proof of Proposition~\ref{prop:EnergyEstGamma}, it only remains to treat a mix of regular derivatives and $\Omega$ vector fields. The crucial tool in this case is the following interpolation result, which allows to separate the vector fields and the derivatives.

\begin{lem}
	\label{lem:InterpolationMixedDerivatives}
	\begin{enumerate}
		\item[]
		\item \label{it:InterpolationMixed1vf} Let $2 \leq p, q \leq \infty$ with
				\begin{align*}
					\frac{1}{p} = \frac{1}{6} + \frac{2}{3q}.
				\end{align*}
				Let $J \in \N_0^4$ with $|J|$ even. We then have
				\begin{equation}
					\label{eq:Interpolation1VFDer}
					\| \Omega \partial^{J} \phi\|_{L^p_{t,x}(C_T^+)} \lesssim \|\Omega^{\leq 3} \phi\|_{L^2_{t,x}(C_T^+)}^{\frac{1}{3}} \| \partial^{\leq \frac{3}{2} |J|} \phi \|_{L^q_{t,x}(C_T^+)}^{\frac{2}{3}}.
				\end{equation}
					
		\item \label{it:InterpolationMixed2vf} Let $2 \leq p, q \leq \infty$ with
				\begin{align*}
					\frac{1}{p} = \frac{1}{3} + \frac{1}{3q}.
				\end{align*}
				Let $J \in \N_0^4$. We then have
				\begin{equation}
					\label{eq:Interpolation2VFDer}
					\| \Omega^2 \partial^{J} \phi\|_{L^p_{t,x}(C_T^+)} \lesssim \|\Omega^{\leq 3} \phi\|_{L^2_{t,x}(C_T^+)}^{\frac{2}{3}} \| \partial^{\leq 3 |J|} \phi \|_{L^q_{t,x}(C_T^+)}^{\frac{1}{3}}.
				\end{equation}
	\end{enumerate}
\end{lem}
\begin{proof}
The proof  of the lemma is similar to the interpolation  result of Ifrim-Stingo (see \cite[Lemma~1.1]{ifrims}) with the difference that  we use different analytic families of operators, namely $T_z \phi = e^{(z - \frac{2}{3})^2} |D_y|^{3(1-z)} |D_s|^{\frac{3}{2}|\beta| z} \phi$ for~\eqref{eq:Interpolation1VFDer} and
$T_z \phi = e^{(z - \frac{1}{3})^2} |D_y|^{3(1-z)} |D_s|^{3|\beta| z} \phi$ for~\eqref{eq:Interpolation2VFDer}.
\end{proof}

We next provide the energy estimates for a mix of vector fields and regular derivatives. Together with Lemma~\ref{lem:EnergyEstHigherDeriv} and Lemma~\ref{lem:EnergyEstMixed} this lemma yields Proposition~\ref{prop:EnergyEstGamma}.
\begin{lem}	
	\label{lem:EnergyEstMixed}
	Let $|I|+3|J|\leq k=9$ and $|I| > 0$ and $|J| > 0$. Assuming the bootstrap hypothesis~\eqref{eq:boot-psi-A} in the $C_{<T}$ region, we get
    \begin{equation}
        \label{eq:EnergyEst}
        \|\Omega^J \partial^I (\psi, A) \|_{X_{T_1}} \lesssim \langle T_1\rangle^{c \epsilon} \| \Gamma^{\leq k} (\psi, A)(0)\|_{\cH^0}, \quad \mbox{ for } 0 \leq T_1\leq T \mbox{ and } c\approx C_0.
    \end{equation}
\end{lem}

\begin{proof}
	There are two cases to consider. Either $\Omega^J\partial^I$ contains one $\Omega$ vector field and up to six regular derivatives or it contains two $\Omega$ vector fields and up to three regular derivatives. We start with the case of one vector field and we assume $|I| = 6$, the case of $|I| < 6$ being treated similarly. Applying $\Omega \partial^I$ to the Maxwell-Dirac system~\eqref{eq:MD}, we get
	\begin{equation*}
\left\{
	\begin{aligned}
	&(- \imu \gamma^\mu \partial_\mu + 1) \Omega \partial^I \psi  = F_{1,I} \\
	& \Box \Omega \partial^I A_\mu = G_{\mu,1,I}, \\
	 \end{aligned}
  \right.
\end{equation*}
with source terms
\[
\left\{
\begin{aligned}
&F_{1,I} = \sum_{I_1 \leq I} \binom{I}{I_1} (\partial^{I - I_1} A_\mu \gamma^\mu \Omega \partial^{I_1} \psi + \Omega \partial^{I - I_1} A_\mu \gamma^\mu \partial^{I_1} \psi) \\
	&G_{\mu,1,I} =  -\sum_{I_1 \leq I} \binom{I}{I_1} (\partial^{I - I_1} \overline{\psi} \gamma_\mu \Omega \partial^{I_1} \psi + \Omega \partial^{I - I_1} \overline{\psi} \gamma_\mu \partial^{I_1} \psi).
 \end{aligned}
 \right.
\]
We use a discrete Gronwall inequality as in the proof of Lemma~\ref{lem:EnergyEstOmegavf} and concentrate on the key step of showing the estimates
\begin{align}
	\label{est perturb 1b}
&T_1^{\frac12} \|F_{1,I}\|_{L^2_{t,x}(C_{T_1})} \lesssim  \epsilon  \, (\| \Gamma^{\leq 9} \psi\|_{L^\infty_t L^2_x(C_{T_1})} + \|\Gamma^{\leq 9} A\|_{L^\infty_t L^3_x(C_{T_1})} ),
\\
\label{est perturb 2b}
&T_1^{\frac13}\| G_{1,I} \|_{L^1_t L^\frac32_x (C_{T_1})}  \lesssim \epsilon \, \| \Gamma^{\leq 9} \psi\|_{L^\infty_t L^2_x (C_{T_1})}.
\end{align}
We first consider the case that $|I_1|$ is even. Here we set
\begin{align*}
	\theta_1 = \frac{2|I - I_1|}{3|I|}, \quad \frac{1}{p_1} = \frac{1}{3} \theta_1, \quad \frac{1}{q_1} = \frac{1}{3} \cdot \frac{1}{2} + \frac{2}{3} \cdot \frac{1}{q_2} = \frac{1}{6} + \frac{2}{3q_2}, \quad \frac{1}{q_2} = \frac{1}{2} - \frac{1}{2} \theta_1.
\end{align*}
Then $\frac{1}{p_1} + \frac{1}{q_1} = \frac{1}{2}$ and we infer
\begin{align*}
	T_1^{\frac12} \| \partial^{I - I_1} A_\mu \gamma^\mu \Omega \partial^{I_1} \psi \|_{L^2_{t,x}(C_{T_1})} 
	\lesssim T_1^{\frac{1}{2}} \|\partial^{I - I_1} A\|_{L^{p_1}_{t,x}(C_{T_1})} \|\Omega \partial^{I_1} \psi\|_{L^{q_1}_{t,x}(C_{T_1})}.
\end{align*}
Rescaling to $T_1 = 1$ and using classical extension theorems, we see that we can use Gagliardo-Nirenberg estimates also on $C_{T_1}$. This yields
\begin{align*}
	\|\partial^{I - I_1} A\|_{L^{p_1}_{t,x}(C_{T_1})} &\lesssim \|\partial^{\leq \frac{3}{2} |I|} A\|_{L^3_{t,x}(C_{T_1})}^{\theta_1} \|A\|_{L^\infty_{t,x}(C_{T_1})}^{1 - \theta_1} \\
	&\lesssim T_1^{\frac{1}{3} \theta_1} \|\partial^{\leq 9} A\|_{L^\infty_t L^3_x(C_{T_1})}^{\theta_1} (C_0 \epsilon {T_1}^{-1})^{1 - \theta_1},
\end{align*}
where we also employed the bootstrap hypothesis~\eqref{eq:boot-psi-A}. For the remaining mixed term, we apply the interpolation result from Lemma~\ref{lem:InterpolationMixedDerivatives}~\ref{it:InterpolationMixed1vf}, which yields
\begin{align*}
	\|\Omega \partial^{I_1} \psi\|_{L^{q_1}_{t,x}(C_{T_1})} \lesssim \|\Omega^{\leq 3} \psi \|_{L^2_{t,x}(C_{T_1})}^{\frac{1}{3}} \| \partial^{\leq \frac{3}{2}|I_1|} \psi\|_{L^{q_2}_{t,x}(C_{T_1})}^{\frac{2}{3}}.
\end{align*}
We next take a multiindex $I_2$ with $|I_2| = \frac{3}{2} |I_1|$, the case of $|I_2| < \frac{3}{2} |I_1|$ again being treated similarly. Setting $\theta_2 = \frac{|I_1|}{|I|}$, we have
\begin{align*}
	\frac{\theta_2}{2} + \frac{1-\theta_2}{6} = \frac{1}{q_2},
\end{align*}
so that an application of Gagliardo-Nirenberg estimates leads to
\begin{align*}
 	\| \partial^{\leq \frac{3}{2}|I_1|} \psi\|_{L^{q_2}_{t,x}(C_{T_1})} \lesssim \| \partial^{\leq \frac{3}{2} |I|} \psi \|_{L^2_{t,x}(C_{T_1})}^{\theta_2} \|\psi\|_{L^6_{t,x}(C_{T_1})}^{1 - \theta_2}.	
\end{align*}
We thus infer
\begin{align*}
	\|\Omega \partial^{I_1} \psi\|_{L^{q_1}_{t,x}(C_{T_1})} &\lesssim T_1^{\frac{1}{6}} \|\Omega^{\leq 3} \psi\|_{L^\infty_t L^2_x(C_{T_1})}^{\frac{1}{3}} T_1^{\frac{1}{2} \theta_2 \cdot \frac{2}{3}} \|\partial^{\leq 9} \psi\|_{L^\infty_t L^2_x(C_{T_1})}^{\frac{2}{3}\theta_2}\\
	&\qquad \cdot T_1^{\frac{1}{6}(1 - \theta_2)\cdot \frac{2}{3}} (C_0 \epsilon T_1^{-1})^{(1 - \theta_2)\cdot \frac{2}{3}},
\end{align*}
where we also employed the bootstrap hypothesis~\eqref{eq:boot-psi-A} again in the last step. Combining this estimate with the one for $\partial^{I - I_1} A$, we get
\begin{align}
	T_1^{\frac12} \| \partial^{I - I_1} A_\mu \gamma^\mu \Omega \partial^{I_1} \psi \|_{L^2_{t,x}(C_{T_1})} \lesssim C_0 \epsilon (\|\Gamma^{\leq 9} \psi \|_{L^\infty_t L^2_x(C_{T_1})} + \| \Gamma^{\leq 9} A\|_{L^\infty_t L^3_x(C_{T_1})}). \label{eq:EstMixedEven}
\end{align}
Next we consider the case that $|I_1|$ is odd. We write $\partial^{I_1} = \partial_\nu \partial^{I_2}$ and use that the commutator of $\Omega$ and $\partial_\nu$ is a linear combination of first order derivatives. Terms of the form $\partial^{I - I_1} A_\mu \gamma^\mu \partial^1 \partial^{I_2} \psi$ can be estimated by Gagliardo-Nirenberg estimates so that we concentrate on estimating
\begin{align*}
	T_1^{\frac12}\| \partial^{I - I_1} A_\mu \gamma^\mu \partial_\nu \Omega \partial^{I_2} \psi \|_{L^2_{t,x}(C_{T_1})} 
	\lesssim T_1^{\frac{1}{2}} \|\partial^{I - I_1} A\|_{L^{p_1}_{t,x}(C_{T_1})} \|\partial_\nu \Omega \partial^{I_2} \psi\|_{L^{q_1}_{t,x}(C_{T_1})}
\end{align*}
with the same $p_1$ and $q_1$ as above. Using Gagliardo-Nirenberg estimates, we get
\begin{align*}
	\|\partial_\nu \Omega \partial^{I_2} \psi\|_{L^{q_1}_{t,x}(C_{T_1})} \lesssim \| \Omega \partial^{I_2} \psi \|_{L^{q_1}_{t,x}(C_{T_1})}^{\frac{1}{2}} \| \partial^{\leq 2} \Omega \partial^{I_2} \psi \|_{L^{q_1}_{t,x}(C_{T_1})}^{\frac{1}{2}}.
\end{align*}
Summing first over $\nu$, we can then absorb the first order derivatives on the above right-hand side in the left-hand side, which leads to
\begin{align*}
	\|\partial_\nu \Omega \partial^{I_2} \psi\|_{L^{q_1}_{t,x}(C_{T_1})} \lesssim \| \Omega \partial^{I_2} \psi \|_{L^{q_1}_{t,x}(C_{T_1})} + \| \Omega \partial^{I_2} \psi \|_{L^{q_1}_{t,x}(C_{T_1})}^{\frac{1}{2}}  \| \partial^{2} \Omega \partial^{I_2} \psi \|_{L^{q_1}_{t,x}(C_{T_1})}^{\frac{1}{2}}.
\end{align*}
Using again that the commutator of $\Omega$ with regular derivatives yields regular derivatives which can be dealt with by Gagliardo-Nirenberg estimates, it remains to estimate $\| \Omega \partial^{I_2} \psi \|_{L^{q_1}_{t,x}(C_{T_1})}$ and $\| \Omega \partial^2 \partial^{I_2} \psi \|_{L^{q_1}_{t,x}(C_{T_1})}$. But as $|I_2|$ is even, we can proceed here as in the case $|I_1|$ even. In conclusion, we also obtain~\eqref{eq:EstMixedEven} in the case $|I_1|$ odd.

The remaining terms in $F_{1,I}$ are treated in a similar way, which yields~\eqref{est perturb 1b}. With adaptions analogous to the ones done in the proof of Lemma~\ref{lem:EnergyEstOmegavf} for the source terms of the wave equation, we then also derive~\eqref{est perturb 2b}.

To treat the case of two $\Omega$ vector fields and up to three regular derivatives, we apply $\Omega^2 \partial^I$ to~\eqref{eq:MD} and estimate the arising source terms in a similar way as in the case of one vector field above, employing part~\ref{it:InterpolationMixed2vf} of Lemma~\ref{lem:InterpolationMixedDerivatives}.
\end{proof}

\section{Pointwise bounds}\label{s:ks}  A first step in recovering the bootstrap bounds on the global time scale is to prove appropriate Klainerman-Sobolev inequalities, where the aim is to obtain pointwise bounds from the integral $X_T$ type  bounds
in Proposition~\ref{prop:EnergyEstGamma}. By itself this does not suffice globally in time because the time growth $t^{C\epsilon}$ from the energy estimates will carry over. 
Instead it only suffices almost globally in time.
Nevertheless, the bounds we establish here will suffice in order to estimate the errors in the asymptotic  equations in later sections.

Our main result here is linear (applies to \eqref{eq: linear system}) and applies at a fixed dyadic scale $T$. Because of this, we omit the $T^{C\epsilon}$ factor
in Proposition~\ref{prop:EnergyEstGamma}.
Then we want to show that 
\begin{thm}\label{t:KS}
    Assume that in a time dyadic region $C_T\cup C_{T/2}$ we have
\begin{equation}\label{Ek} 
\| \Gamma^{\leq k}(\psi,A)(t)\|_{X_T} \leq 1, \qquad k = 9. 
\end{equation}
Then in $C_T$ we have
\begin{equation}\label{eq:psi-uniform}
|\psi| \lesssim \langle t+r\rangle^{-\frac32}  \langle t-r \rangle_-^{-\delta} ,   \end{equation}
\begin{equation}\label{eq:A-uniform}
|\partial A| \lesssim t^{-1} \langle t-r \rangle^{-\frac12},
\end{equation}
\begin{equation}\label{eq:A-uniform2}
|\mathcal T A| \lesssim t^{-\frac32} ,
\end{equation}
where $\delta > 0$. In addition, inside the cone we have an improved bound for $\mathcal{T} A$, namely
\begin{equation}
\label{eq:improved TA}
|\mathcal{T} A|\lesssim \langle t \rangle ^{-\frac{3}{2}}\langle t-r\rangle^{-\frac{1}{2}}.
\end{equation} 
\end{thm}
Here $\mathcal{T}$ stands for normalized derivatives in directions which are tangent to the light cones
$\{t-r = const\}$.
These are spanned by $\mathcal T = \{ \partial_r+\partial_t,\slash\!\!\!\!\nabla \}$
where $\slash\!\!\!\!\nabla$ stands for the derivatives in angular directions. Equivalently, in the above theorem one may 
use derivatives tangent to the hyperboloids $\{t^2-r^2 = const\}$.

\begin{rem} What is missing here is the uniform bound 
for $A$,
\[
|A| \lesssim t^{-1},
\]
which would be too much to ask for at this point, using only the information given in the hypothesis of the theorem above. Instead, we will prove pointwise bounds for $A$ later on, in Section~\ref{s:A}, by using the wave equation for 
$A$ and the pointwise bounds for $\psi$.
\end{rem}

\begin{rem}
The improved bound \eqref{eq:improved TA} is due to the fact that our baseline spaces for the wave equation are $\dot{H}^\frac{1}{2} \times \dot{H}^{-\frac{1}{2}}$, as opposed to $ \dot{H}^1\times L^2$ . With additional work one should be able to obtain a similar improved bound for $\nabla A$ 
\begin{equation}
\label{eq:improved NA}
|\nabla A|\lesssim \langle t \rangle^{-1} \langle t-r\rangle^{-1} ,
\end{equation} 
but this would require some adjustments to the $X_T$-norm, which we chose not to pursue here. 
\end{rem}
 
\medskip
\begin{proof} To fix the notations we denote the right hand side of the linear equation for $(\psi, A)$ as follows
\begin{equation}
\left\{
	\begin{aligned}
	\label{eq:MD-inh}
	-& \imu \gamma^\mu \partial_\mu \psi + \psi = F \\
	 &\Box A_\mu = G .
	 \end{aligned}
  \right.
\end{equation}
These equations are assumed to hold in $C_T \cup C_{T/2}$. Because the proof involves fractional and negative Sobolev spaces we would want to use a spatial Littlewood-Paley decomposition. But this dyadic decomposition is not entirely compatible with the geometry of $C_T$, so we would like to extend $(\psi, A)$ to the entire strip. It suffices to appropriately extend $(F, G)$:
\begin{lem} 
\label{lem:Extension}
Let $(F, G)$ be functions in $C_T$ which satisfy the bound
\begin{equation}
T^\frac12 \|\Gamma^{\leq 2k} F   \|_{L^2_{t,x}(C_T)} + T^{\frac13} \| \Gamma^{\leq 2k}G\|_{L^\frac32_{t,x}(C_T)} \lesssim 1.
\end{equation}
Then there exists an extension, still denoted by $(F, G)$, which satisfies the same estimate in the full time slice $[T, 4T]$.
\end{lem}
\begin{proof}
After rescaling to $T = 1$, we use hyperbolic coordinates in order to apply standard extension techniques such as reflection at the boundary, cf.~\cite{AF03} for the case of regular derivatives.
\end{proof} 

The above lemma. i.e., Lemma~\ref{lem:Extension}, allows us to replace the cup region $C_T$ with the full time slice $[T,4T]$.
Adding also the region $C_{T/2}$, we can assume
that $(\psi,A)$ satisfy $X_T$ type bounds in the larger time slab $[T/2,4T]$. However, we only need to prove the pointwise bounds in the theorem in the region $C_T$.  To this we will add the lower cap, thus working with the region $C_T^+$  defined in \eqref{eq:ct+}.

\medskip

Now we consider separately the Dirac and the
wave component, working in the full time slab 
$[T/2,4T]$. In this setting, it suffices to prove the desired pointwise bounds in the region 
$C_T^+$.  Our strategy will be to reduce the proof of the theorem to standard Sobolev embeddings in regions which, in suitable coordinates, have unit size.
To place ourselves in this situation, we decompose
the region $C_T^+$ into smaller regions which have fixed geometry, as follows:
\begin{equation}\label{CT-decomp}
C_T^+ := C_T^{int} \bigcup C_T^{ext} \bigcup_{\pm} \bigcup_{1 \leq S \leq T} C_{TS}^{\pm}  .  
\end{equation}
We now describe the sets in this decomposition:

\begin{itemize}
    \item The interior region $C_T^{int}$ is defined as 
 \[
C_T^{int} := ( [T/2,4T] \times \R^3 ) \cap 
\{ T^2/4 \leq   t^2 -x^2 \leq 4T^2 \}.
 \]
 This region can be foliated with large sections of hyperboloids.

 \item The exterior region $C_T^{ext}$ is far outside the cone, and is described as 
 \[
C_T^{ext} := \{ (t,x); t \in [T,4T]; r \geq 2 t \}.
 \]
 \item The region around the cone, we dyadically decompose with respect to the size
of $t-r$, which  measures how far or close we are to the cone
\begin{equation}
\label{cts}
\begin{aligned}
&C_{TS}^{+}:=\left\{ (t,x)\, :\,   S\leq t-r\leq 2S, \, T\leq t\leq 4T \right\},  \mbox{ where } 1\leq S\lesssim T, \\
&C_{TS}^{-}:=\left\{ (t,x)\, :\,   S\leq r-t\leq 2S, \, T\leq t\leq 4 T \right\}, \mbox{ where } 1\leq S \lesssim T;
\end{aligned}
\end{equation}
see Figure~\ref{f:D0}.
\end{itemize}

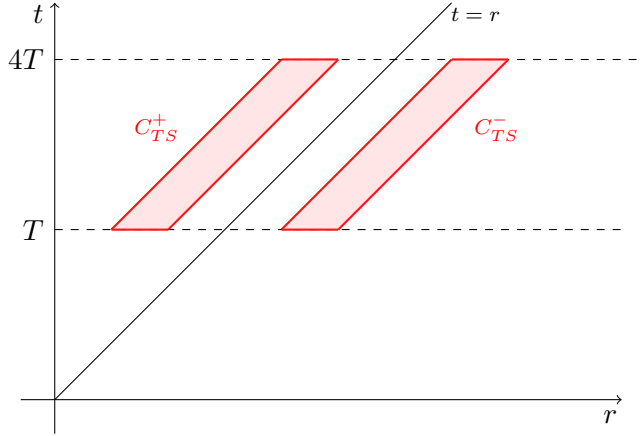
\begin{figure}[h]
\begin{center}
\begin{tikzpicture}[scale=1.5]

\draw[->] (0,-0.3) -- (0,3.5);
\draw[->] (-0.3,0) -- (5,0);
\node[below] at (4.9,0) {\small $r$};
\node[left] at (0,3.4) {\small $t$};

\draw [domain=0:3.5] plot(\x, \x);
\node[right] at (3.4,3.4) {\tiny $t=r$};

\draw[dashed] [domain=0:5] plot (\x, 1.5);
\node[left] at (0,1.5) {\small $T$};
\draw[dashed] [domain=0:5.3] plot (\x, 3);
\node[left] at (0,3) {\small $4T$};
 
\draw[red, thick] [domain=0.5:2] plot(\x+.5, \x+1);
 \draw[red, thick] [domain=.5:2] plot(\x, \x+1);
 \draw[red, thick] [domain=0:0.5] plot(\x+.5, 1.5);
 \draw[red, thick] [domain=1.5:2] plot(\x+.5, 3);
 
  \node[red,thick] at (0.9, 2.4) {\tiny $C^+_{TS}$};
 
 \draw[red, thick] [domain=2.5:4] plot(\x-.5, \x-1);
 \draw[red, thick] [domain=3.5:5] plot(\x-1, \x-2);
\draw[red, thick] [domain=2.5:3] plot(\x-.5, 1.5);
 \draw[red, thick] [domain=4:4.5] plot(\x-.5, 3);
   \node[red,thick] at (3.9, 2.4) {\tiny $C^{-}_{TS}$};

\fill[red!40,nearly transparent] (2,1.5) -- (2.5,1.5) -- (4,3) -- (3.5,3) -- cycle;

\fill[red!40,nearly transparent]  (0.5,1.5) -- (1,1.5) -- (2.5,3) -- (2, 3) -- cycle;

\end{tikzpicture}
\caption{1D vertical section of space-time regions $C^{\pm}_{TS}$}
\label{f:D0}
\end{center}
\end{figure}

Here $C^+_{TS}$ represents a spherically symmetric dyadic region
inside the cone with width $S$, distance $S$ from the cone, and time
length $T$.  $C^-_{TS}$ is the similar region outside the cone where,
far from the cone, we would have $T\lesssim S$. To simplify the
exposition we will use the notation $C_{TS}$ as a shorthand for either
$C^+_{TS}$ or $C^-_{TS}$. 
These regions are also well foliated with sections of hyperboloids.
Such a decomposition has been introduced before 
by Metcalfe-Tataru-Tohaneanu~\cite{mtt12} in a linear setting; we largely follow their notations. 

In the above definition of the $C_{TS}$ sets we limit $S$ to $S \geq 1$ because our assumptions are invariant with respect to unit size translations. In particular,  this leaves out a conical shell region along the side of the cone $t=r$, which intersects both the interior and the exterior of the cone. To also include this region in our analysis we redefine 
\begin{equation}
\label{ct1}
C_{T1}:=\left\{ (t,x)\, :\,   \vert t-r\vert \leq 2, \, T\leq t\leq 2T \right\},  \mbox{ where } S\sim 1.
\end{equation}

\subsection{The bounds for the Dirac equation.}
We will prove our pointwise bounds for the 
Dirac equation separately in each of the regions of the $C_{T}^+$ decomposition in \eqref{CT-decomp}. To do this, we need $L^2$ bounds for extended $X_T$ 
functions $\psi$ in each of these regions. 
Just as in the case of the wave equation, not all components of $\psi$ will satisfy the same bounds. To understand this, we write the $L^2$
conservation for Dirac in a density flux form using \eqref{dens-flux}, and which, for convenience, we recall below
\begin{equation*}
\partial_t |\psi|^2 + \partial_j ( \psi^\dag \gamma^0 \gamma^j \psi ) = -2 \Im (\psi^\dag \gamma^0 F ).   
\end{equation*}
We note that here it is important that the matrices $\gamma^0 \gamma^j$ are all Hermitian.

The first use we have for this relation is to integrate it in the region $D$ between the time-slice $t = T$ and a hyperboloid $H$. Then we 
obtain an expression for the  energy  on the hyperboloid,
\[
E_H(\psi):= \int_{H\cap C_T} \nu_0 |\psi|^2 + \nu_j ( \psi^\dag \gamma^0 \gamma^j \psi)\, d\sigma = \int_{H\cap C_T} e_H(\psi) \, dV,
\]
namely
\[
E_H(\psi) =\int_{t=T} |\psi|^2\, dx - 
\iint_D 2 \Im ( \psi^\dag \gamma^0 F)\, dx dt -\int_{\left\{t=4T\right\}\setminus \mbox{cup}} |\psi|^2\, dx,
\]
which yields the bound
\begin{equation}
\label{Energy-H}
E_H(\psi) \lesssim \|\psi\|_{X_T}^2.   
\end{equation}
Here the cup region is the region above the hyperboloid $H$. 
The energy density on hyperboloids was computed earlier in \eqref{eH} in terms of the $\psi_{\pm}$ decomposition of $\psi$. Precisely, we can rewrite  the energy on hyperboloids
as 
\[
E_H(\psi):= \int_{H\cap C_T}  \frac{t-r}{\sqrt{t^2+r^2}} |\psi_+|^2
+ \frac{t+r}{\sqrt{t^2+r^2}} |\psi_-|^2\, d\sigma.
\]

In a strictly smaller angle inside the cone both coefficients have size one, so we control the full $L^2$ norm of $\psi$ on $H\cap C_T$. This immediately allows us to bound $\psi$ in the interior region $C_T^{int}$, which is well foliated by sections of hyperboloids with uniform size $O(T)$,
\[
\| \psi\|_{L^\infty(C_T^{int})} 
\lesssim \sup_{H} \| \psi\|_{L^\infty(H \cap C_T^{int})} \lesssim 
T^{-\frac32} \sup_H  \| \Omega^{\leq 2} \psi\|_{L^2(H \cap C_T^{int})} .
\]

Next we consider the $C_{TS}^{\pm}$ regions. As we approach the cone, applying the same argument as above loses a factor of $(T/S)^\frac12$ and is no longer sufficient. Outside the cone this cannot be done at all. Instead, we prove 
$L^2$ bounds in $C_{TS}$, both inside and outside the cone:

\begin{lem}\label{lem:CTS}
We have the following estimates:
\begin{equation}\label{psip-est}
\| \psi_+\|_{L^2_{t,x}
(C_{TS})} \lesssim T^\frac12 \|\psi\|_{X_T} ,   
\end{equation}
respectively
\begin{equation}\label{psim-est}
\| \psi_-\|_{L^2_{t,x}(C_{TS})} \lesssim S^\frac12 \|\psi\|_{X_T}  .
\end{equation}
\end{lem}
\begin{proof}
The bound for $\psi_+$ follows trivially from the energy estimate and H\"older's inequality, as 
\[
\| \psi_+\|_{L^2_{t,x}(C_{TS})} \lesssim 
\| \psi_+\|_{L^2_{t,x}(C_{T})} \lesssim T^\frac12 
\| \psi_+\|_{L^\infty_t L^2_{x}}.
\]
For $\psi_-$ we integrate the density-flux relation with a well chosen weight $\chi(t-r)$
which is bounded, nonnegative and a nonincreasing function of its argument.
This yields
\[
\int_{C_T} - \chi'(t-r) (|\psi|^2 - \langle \gamma^0 \gamma^\theta \psi,\psi \rangle) \, dx dt 
\lesssim \|\psi \|_{L^\infty_t L^2_x}^2 +
\| \psi\|_{L^\infty_t L^2_x} \|F\|_{L^1_tL^2_x} \lesssim \|\psi\|_{X_T}^2,
\]
which gives
\[
\int_{C_T} - \chi'(t-r) |\psi_-|^2 \, dx dt 
\lesssim \|\psi\|_{X_T}^2.
\]
To get \eqref{psim-est} it suffices to choose 
$\chi$ with the property that $\chi' \approx - S^{-1}$ in the interval $[S,2S]$ for $C_{TS}^+$,
respectively $[-2S,-S]$ for $C_{TS}^-$.
    
\end{proof}

In the context of Theorem~\ref{t:KS}, similar bounds apply for vector fields applied to $\psi_{\pm}$:

\begin{cor}\label{c:CTS}
Assume that \eqref{Ek} holds. Then with $k = 9$ we also have
\begin{equation}
\begin{aligned}
\|\Gamma^{\leq k} \psi_+\|_{L^2_{t,x}(C_{TS})} \lesssim & \  T^\frac12 \|\Gamma^{\leq k} \psi\|_{X_T}  ,  
\\
\|\Gamma^{\leq k} \psi_-\|_{L^2_{t,x}(C_{TS})} \lesssim & \ S^\frac12 \|\Gamma^{\leq k} \psi\|_{X_T} .
\end{aligned}
\end{equation}
\end{cor}

\begin{proof}
From Lemma~\ref{lem:CTS} we have the bounds
   \begin{equation}
\|P^{\theta}_+ \Gamma^{\leq k} \psi\|_{L^2_{t,x}(C_{TS})} \lesssim T^\frac12 \|{\Gamma^{\leq k}}\psi\|_{X_T} ,   
\end{equation}
respectively
\begin{equation}
\|P_{-}^{\theta} \Gamma^{\leq k} \psi\|_{L^2_{t,x}(C_{TS})} \lesssim S^\frac12 \|\Gamma^{\leq k} \psi\|_{X_T}  .
\end{equation} 
The conclusion of the corollary would then immediately follow if we knew that the projectors
$P_{\pm}^{\theta}$ commute with our vector fields $\Gamma$. This is unfortunately not true, but we 
have the next best thing, namely that the commutators are perturbative. To see this it suffices to consider spatial derivatives and Lorentz boosts, as the angular vector fields are easily seen to commute with $P_{\pm}^{\theta}$. For regular derivatives we
compute the commutator
\[
Q^{\theta} :=[ \partial_x, P_+^{\theta}] .
\]
Since $P_+^{\theta}$ is smooth and zero-homogeneous in $x$ it follows that $Q^{\theta}$ is also smooth, and $-1$-homogeneous in $x$, and in particular it has pointwise size $ O(| x|^{-1})$ which in $C_{TS}$ means $ O(T^{-1})$.
For Lorentz boosts it is convenient 
to switch to the radial operator
\[
\Omega_{0r} := \theta^j \Omega_{0j},
\]
which together with the rotations 
generates all the others, with bounded coefficients, whose derivatives gain $T$ factors.
The advantage is that $\Omega_{0r}$ commutes 
with $P_{\pm}^{\theta}$, which immediately implies the conclusion of the corollary.
\end{proof}

Our next objective is to obtain pointwise bounds in $C^{\pm}_{TS}$ from the $L^2$ bounds.
For this we cannot just rely on vector fields, as these only span the directions tangent to hyperboloids. Instead, we also need to capture the transversal direction, which we do using the Dirac equation.
For this, it is convenient to write the Dirac equation in the polar frame, which has the form
\begin{equation}
-i \gamma^0 \partial_0 \psi - i \gamma^\theta \partial_r \psi - \frac{i}{r} \Omega \psi + \psi = F ,    
\end{equation}
where $\Omega$ captures the angular part.
We can also use the Lorentz boost to further 
simplify this as follows 
\begin{equation}
-i (- \frac{t}{r} \gamma^0 + \gamma^\theta) \partial_r \psi -\frac{i}{r} \Omega \psi + \psi = F .    
\end{equation}

Our vector field bounds in Corollary~\ref{c:CTS} allow us to include the $\Omega$ component part into $F$, so it suffices to consider the radial ordinary differential equation
\begin{equation}
- i (- \frac{t}{r} \gamma^0 + \gamma^\theta)\partial_r \psi + \psi = F_1 .
\end{equation}
Since we have already used one vector field, 
the bounds we have for $F_1$ have the form
\begin{equation}\label{reduced-source}
 \| \Gamma^{\leq 6} F_1\|_{L^2_{t,x}(C_{TS})} \lesssim T^{-\frac12}.   
\end{equation}

In the last equation we multiply by $\gamma^0$ and project, noting that $\gamma^0$ interchanges the projectors, and thus the corresponding subspaces,
\[
\gamma^0 P^\theta_{+} =  P^\theta_- \gamma^0. 
\]
We arrive at the coupled system
\begin{equation} \label{coupled}
\left\{
\begin{aligned}
&\frac{t-r}{r} \partial_r \psi_+ - i \gamma^0 \psi_- = (i\gamma^0 F_1)_+ := F_+\\
& \frac{t+r}{r} \partial_r \psi_- - i \gamma^0 \psi_+ =  (i\gamma^0 F_1)_- := F_-.
\end{aligned}
\right.
\end{equation}
This is hyperbolic inside the cone and elliptic outside. Apriori, for the entries above  
we have the estimates \eqref{psip-est},
\eqref{psim-est} and \eqref{reduced-source}, which we recall below:
\begin{equation}\label{system-bds}
\|\Gamma^{\leq 9} \psi_-\|_{L^2_{t,x}(C_{TS})} \lesssim S^\frac12, \quad 
\|\Gamma^{\leq 9} \psi_+\|_{L^2_{t,x}(C_{TS})} \lesssim T^\frac12, 
\quad \|\Gamma^{\leq 6} F_{\pm}\|_{L^2_{t,x}(C_{TS})} \lesssim T^{-\frac12}.
\end{equation}

To prove the pointwise bounds we need the estimates
\begin{align}
    \label{eqn:undiff}
    \| \Gamma^{\leq 6} \psi_{\pm} \|_{L^2_{t,x}(C_{TS})} \lesssim S T^{-\frac{1}{2}}.
\end{align}
 A direct application of~\eqref{coupled} and~\eqref{system-bds} yields the required estimate for $\Gamma^{\leq 6} \psi_{-}$, but only $\| \Gamma^{\leq 6} \psi_+ \|_{L^2_{t,x}(C_{TS})} \lesssim S^{\frac{1}{2}}$. To improve on that bound, we look at the first equation in~\eqref{coupled} for $\Gamma^{\leq 7} \psi_{-}$ and note that the corresponding right-hand side is given by
\begin{align*}
    i P^\theta_+ \gamma^0 \Gamma^{\leq 7} F - P^\theta_+ \gamma^0 \frac{1}{r} \Omega \Gamma^{\leq 7} \psi \sim i P^\theta_+ \gamma^0 \Gamma^{\leq 7} F - \gamma^0 \frac{1}{r} \Omega \Gamma^{\leq 7} \psi_{-}
\end{align*}
up to lower order commutator terms which are of size $T^{-\frac{1}{2}}$. For the first term we have $\| \Gamma^{\leq 7} F\|_{L^2(C_{TS})} \lesssim T^{-\frac{1}{2}}$, while for the second one, we get the following bound
\begin{align*}
    \|r^{-1} \Omega \Gamma^{\leq 7} \psi_{-}\|_{L^2_{t,x}(C_{TS})} &\lesssim \frac{1}{T} \|\partial \Omega \Gamma^{\leq 6} \psi_{-} \|_{L^2_{t,x}(C_{TS})} + T^{-\frac{1}{2}} \\ 
    &\lesssim \frac{1}{T} \|\partial^{\leq 2} \Omega \Gamma^{\leq 6} \psi_{-}\|_{L^2_{t,x}(C_{TS})}^{\frac{1}{2}} \|\Omega \Gamma^{\leq 6} \psi_{-} \|_{L^2_{t,x}(C_{TS})}^{\frac{1}{2}} + T^{-\frac{1}{2}} \\
    &\lesssim S^{\frac{1}{4}} T^{-\frac{1}{2}} \|T^{-1}\partial^{\leq 2} \Omega \Gamma^{\leq 6} \psi_{-}\|_{L^2_{t,x}(C_{TS})}^{\frac{1}{2}} +  T^{-\frac{1}{2}}\\
    &\lesssim S^{\frac{1}{4}} T^{-\frac{1}{2}} \|T^{-1} \Omega \Gamma^{\leq 8} \psi_{-}\|_{L^2_{t,x}(C_{TS})}^{\frac{1}{2}} + T^{-\frac{1}{2}}.
\end{align*}
Here we used interpolation as well as the bounds in ~\eqref{system-bds}. In getting our estimates we commuted $\Omega$ with regular derivatives $\partial$ which lead to  just regular derivatives of first order, which in turn could be bounded. Finally, we exploit that in $C_{TS}$ we can estimate $T^{-1} \Omega$ by first order derivatives, i.e.,
\begin{align*}
    \|T^{-1} \Omega \Gamma^{\leq 8} \psi_{-}\|_{L^2_{t,x}(C_{TS})}^{\frac{1}{2}} \lesssim \|\Gamma^{\leq 9} \psi_{-}\|_{L^2_{t,x}(C_{TS})}^{\frac{1}{2}} \lesssim S^{\frac{1}{4}}.
\end{align*}
System~\eqref{coupled} and the bounds~\eqref{system-bds} thus imply
\begin{align*}
    \| \Gamma^{\leq 7} \psi_{-} \|_{L^2_{t,x}(C_{TS})} \lesssim S T^{-\frac{1}{2}},
\end{align*}
and then
\begin{align*}
    \| \Gamma^{\leq 6} \psi_+ \|_{L^2_{t,x}(C_{TS})} \lesssim S T^{-\frac{1}{2}},
\end{align*}
finishing the proof of~\eqref{eqn:undiff}.

At this stage we are ready to consider separately the interior case $C_{TS}^+$ and the exterior case $C_{TS}^-$.

\bigskip

We first consider \textbf{ the interior case  $C_{TS}^+$}.
There we exploit the hyperbolic structure 
of the system \eqref{coupled} by introducing an approximately conserved energy density
\[
e := (t-r) |\psi_+|^2 + (t+r) |\psi_-|^2,
\]
where 
\[
|\partial_r e| \lesssim |\psi_+|^2 + 
|\psi_{-}|^2 + r |\psi_+| |F_{+}|  + r |\psi_-| |F_{-}|.
\]

For this density, employing~\eqref{eqn:undiff}, we first estimate its integral
\[
\int_{C_{TS}} e \, dx dt \lesssim S^2,
\]
and then the integral of its radial derivative,
\[
\int_{C_{TS}} | \partial_r e| \, dx dt \lesssim S,
\]
where we used~\eqref{eqn:undiff} and the bound~\eqref{system-bds} for $F_\pm$.
Now we think of $C_{TS}^+$ as being foliated by hyperboloids, with the transversal direction 
given by the radial direction. In the radial direction the set $C_{TS}^+$ has thickness $O(S)$,
so the  two estimates above allow us to bound the trace of $e$ on hyperboloids,
\[
\int_{H \cap C_{TS}^+} e \, d \sigma \lesssim S,
\]
or equivalently 
\begin{equation}
\label{eq:L2BoundHyperboloid}
\int_{H \cap C_{TS}^+} S |\psi_+|^2 + T |\psi_-|^2 \, d\sigma\lesssim S.
\end{equation}

 The same estimate applies to $\Omega^{\leq 2} \psi_{\pm}$. Therefore using Sobolev embeddings on  hyperboloid slices $H \cap C_{TS}^+$
we arrive at the pointwise bounds
\begin{equation}
    \label{eq:PointwiseBoundspsiplpsimin}
| \psi_{+}| \lesssim T^{-\frac32}, \qquad 
|\psi_-| \lesssim S^\frac12 T^{-2} \qquad \text{in } C_{TS}^+.
\end{equation}

\bigskip

Next we consider the case of \textbf{the exterior regions $C_{TS}^-$}.
To get a good elliptic bound we square the two equations in \eqref{coupled} with appropriate weights and a smooth cutoff $\chi(t-r)$ which selects a slight enlargement of $C_{TS}^-$,
\[
\int_{C_T} (t+r) \chi(t-r)|(t-r) \partial_r \psi_+ - i r \gamma^0 \psi_-|^2 \, dx dt  = \int_{C_T} (t+r) r^2 \chi(t-r)|F_+|^2 \, dx dt,
\]
respectively
\[
\int_{C_T} (r-t) \chi(t-r)|(t+r) \partial_r \psi_- - i r \gamma^0 \psi_+|^2 \, dx dt  = \int_{C_T} (r-t) r^2 \chi(t-r) |F_-|^2 \, dx dt .
\]
Now we add them up, expanding the squares and integrating by parts to cancel the cross terms.
We note that the weights were matched exactly so that we can achieve this last cancellation.
We obtain
\[
\begin{aligned}
\int_{C_T} \chi(t-r) (
S^2 T |\partial_r \psi_-|^2
+ T^3 |\psi_+|^2+ 
ST^2 |\partial_r \psi_+|^2
+ ST^2 |\psi_-|^2 )\, dx dt \\
\qquad \qquad \lesssim
\int_{2C_{TS}} T^2 |\psi|^2 + T^3 (|F_+|^2 +|F_-|^2) \, dx dt.
\end{aligned}
\]
As $|\psi|^2 = |\psi_+|^2 + |\psi_{-}|^2$, the coefficients of $\psi_+$ and $\psi_{-}$ on the left are larger than the one on the right by at least a factor of $S$. So reiterating with a slightly wider bump $\chi$ 
we arrive at 
\[
\begin{aligned}
\int_{C_T} \chi(t-r) (
S^2 T |\partial_r \psi_-|^2
+ T^3 |\psi_+|^2+ 
ST^2 |\partial_r \psi_+|^2
+ ST^2 |\psi_-|^2 ) \, dx dt
\\
\qquad \qquad \lesssim
\int_{3C_{TS}} S^{-1} T^2 |\psi|^2 + T^3 (|F_+|^2
+|F_-|^2 ) \, dx dt.
\end{aligned}
\]
We can use \eqref{eqn:undiff} to estimate the right hand side, which yields
\begin{equation}\label{eqn:elliptic}
\int_{C_{TS}}  \left(
\frac{S^2}{T} |\partial_r \psi_-|^2
+ T  |\psi_+|^2+ 
 S |\partial_r \psi_+|^2
+ S |\psi_-|^2 \right) \,dx dt  \lesssim 1.
\end{equation}

Now we are able to obtain $L^2$ bounds on the hyperboloids, by using an interpolation inequality of the form
\begin{equation}
\label{interp-H}
\| \psi_-\|_{L^2(H \cap C_{TS})}^2 
\lesssim \| \psi_-\|_{L^2_{t,x}( C_{TS})} \| \partial_r \psi_-\|_{L^2_{t,x}(C_{TS})} + S^{-1}  \| \psi_-\|_{L^2_{t,x}( C_{TS})}^2,
\end{equation}
and similarly for $\psi_+$, where 
the $S$ factor represents the thickness of $C_{TS}$ in the $r$ direction. We combine this interpolation inequality with  \eqref{eqn:elliptic}, and also 
with \eqref{eqn:undiff} in the case of $\psi_-$
in order to get the following bounds on the hyperboloids,
\[
\| \psi_-\|_{L^2(H \cap C_{TS})}^2 
\lesssim \min\left\{\frac{S}{T^\frac12}, \frac{1}{S^\frac12}\right\}  \frac{T^\frac12}{S} = \min\left\{ 1, \frac{T^\frac12}{S^\frac32}\right\},
\]
where we have better decay as we get farther 
away from the cone. Similarly,
\[
\| \psi_+\|_{L^2(H \cap C_{TS}^-)}^2 
\lesssim \| \psi_+\|_{L^2_{t,x}( C_{TS}^-)} \| \partial_r \psi_+\|_{L^2_{t,x}(C_{TS}^-)} \lesssim  \frac{1}{S^\frac12 T^\frac12}  \frac{1}{S} 
\]
which is even better away from the cone.

In view of our starting point in \eqref{system-bds}, we similarly obtain the same bounds for 
$\Omega^{\leq 2} \psi_{\pm}$. Then we can again apply Sobolev embeddings on the hyperboloids
to arrive at the pointwise estimates
\begin{equation}
\label{eq:PointwBoundCTSmin}
|\psi_-|  \lesssim T^{-\frac32}   \min\left\{ 1, \frac{T^\frac12}{S^\frac32}\right\}, \qquad 
|\psi_+| \lesssim T^{-\frac32} \frac{1}{S^\frac12 T^\frac12}  \frac{1}{S} \qquad \text{in } C_{TS}^-,
\end{equation}
which suffices for \eqref{eq:psi-uniform}.

Finally, we need to consider the exterior region $C_{T}^{ext}$. The argument is similar to the one 
in $C_{TS}^-$, but simpler since we no longer 
need to differentiate between $\psi_+$ and $\psi_-$ and between the factors $t+r$ and $t-r$. For brevity we leave the detailed computation to the interested reader. This concludes the proof of the 
uniform bounds for $\psi$ in the theorem.

\bigskip

\subsection{The bounds for the wave equation.}
To prove the pointwise bounds for the 
wave equation in \eqref{eq:A-uniform},
\eqref{eq:A-uniform2} and \eqref{eq:improved TA} we start by localizing in frequency,
decomposing dyadically 
\[
A = \sum A_\lambda.
\]

The $X_T$ vector field bounds still hold for all $A_\lambda$, with appropriate frequency factors. Then we prove pointwise bounds for each component 
separately in $C_{T}^+$, and simply add them up. 

We begin our analysis with a brief discussion of energy estimates, for which  we write the energy
in a density-flux form,
\begin{equation}\label{energy-A}
\partial_t T^{00} + \partial_j T^{0j} = G \cdot \partial_t A,   
\end{equation}
where the corresponding components of the energy momentum tensor are
\[
T^{00} := \frac12(|\partial_t A|^2 + |\nabla_x A|^2), \qquad T^{0j} := - \partial_0 A \partial_j A.
\]
We first integrate it in the region, which we denote by $D$, between the time-slice $t = T$ and a hyperboloid $H$. Then we 
obtain an energy bound on the hyperboloid,
\[
\begin{aligned}
E_H(A):=& \int_{H\cap C_T} \nu_0 T^{00} + \nu_j T^{0j}\, d\sigma  +\int_{\left\{t=4T\right\}\setminus \mbox{cup}} T^{00}\, dx\\
= & \int_{t=T} T^{00} dx + 
\iint_D G \cdot \partial_t A \,dx dt,
\end{aligned}
\]
which yields the bound
\begin{equation}
E_H(A) \lesssim \|A\|_{X_T}^2.    
\end{equation}
Here the energy on $H$ can be written as
\[
E_H(A):= \int_{H\cap C_T} \frac12\nu_0 (|\partial_t A|^2+|\nabla_x A|^2) - |\nu_x|  \theta \cdot \nabla_x A \partial_t A \, d\sigma ,
\]
which we can rewrite 
as 
\[
E_H(A):= \int_{H\cap C_T} (\nu_0 - |\nu_x|) |(\partial_t - \partial_r) A|^2
+
(\nu_0 + |\nu_x|) (|(\partial_t + \partial_r) A|^2+ |\slash\!\!\!\!\nabla A|^2) \, d\sigma,
\]
where the last term represents derivatives in angular directions. Here the two coefficients have size
\[
\nu_0 + |\nu_x| \approx 1, \qquad \nu_0 - |\nu_x|
\approx \frac{|t-r|}{t+r}.
\]

Inside $C_T^{int}$ both coefficients have size one, so we control the full $L^2$ norm of $\nabla A$ on $H$. This immediately allows us to bound $\nabla A$ in the interior region,
\[
\begin{aligned}
\| \nabla A\|_{L^\infty(C_T^{int})} 
\lesssim \sup_{H} \| \nabla A\|_{L^\infty(H \cap C_T^{int})} \lesssim 
T^{-\frac32} \sup_H  \| \Omega^{\leq 2} \nabla A\|_{L^2(H \cap C_T^{int})}. 
\end{aligned}
\]
We can apply this separately to each $A_\lambda$
noting that 
\[
\| \Omega^{\leq 2} \nabla A_\lambda\|_{L^2(H \cap C_T^{int})} \lesssim \lambda^{\frac12} 
\| \Gamma^{\leq 6} A_\lambda\|_{X_T} \lesssim \lambda^{-\frac52} 
\| \Gamma^{\leq 9} A_\lambda\|_{X_T}.
\]
As we approach the cone, the same argument can be applied to tangential derivatives but not to normal ones. Outside the cone this cannot be done at all. Instead, we prove $L^2$ bounds in $C_{TS}$, both inside and outside the cone:

\begin{lem}\label{lem:CTS-A}
We have the following estimates:
\begin{equation}
\| \nabla  A\|_{L^2_{t,x}(C_{TS})} \lesssim T^\frac12 (\|\nabla A\|_{L^\infty_t L^2_x} + \|\Box A\|_{L^1_tL^2_x}) ],   
\end{equation}
respectively
\begin{equation}
\| \mathcal T A\|_{L^2_{t,x}(C_{TS})} \lesssim S^\frac12  (\|\nabla A\|_{L^\infty_t L^2_x} + \|\Box A\|_{L^1_tL^2_x}). 
\end{equation}
\end{lem}
\begin{proof}
The first bound is a direct consequence of H\"older's inequality in time. The second one
is obtained exactly as in the Dirac case by integrating the density-flux relation \eqref{energy-A} against a radial weight $\chi(t-r)$.    
\end{proof}

Similar bounds hold for $A_\lambda$, 
as 
\[
\|\nabla A_{\lambda}\|_{L^\infty_t L^2_x} + \|\Box A_\lambda\|_{L^1_tL^2_x}
\lesssim \lambda^{\frac12} \| A_\lambda\|_{X_T},
\]
and even for vector fields applied to $A$,
for which we can write
\begin{equation}
\| \Gamma^{\leq 9} \nabla  A_\lambda\|_{L^2_{t,x}(C_{TS})} \lesssim T^\frac12 \lambda^{\frac12} ,  
\end{equation}
respectively
\begin{equation}
\| \Gamma^{\leq 9}\mathcal T A_\lambda\|_{L^2_{t,x}(C_{TS})} \lesssim S^\frac12  \lambda^{\frac12}.
\end{equation}
We first address the pointwise bounds for $\mathcal T A_\lambda$.
We observe that $\partial_t + \partial_r = \frac{1}{r} \Omega_{0r} - \frac{t-r}{r} \partial_r$ with $\Omega_{0r} = \frac{x_i}{r} \Omega_{0i}$. Using the pointwise estimate which we show for $\partial_r A$ in~\eqref{dr4A} below, we obtain $\|\frac{t-r}{r} \partial_r A\|_{L^\infty(H \cap C_{TS})} \lesssim S T^{-1} S^{-\frac{1}{2}} T^{-1} \lesssim T^{-\frac{3}{2}}$. For the remaining components $\frac{1}{r} \Omega$ of the tangential derivatives, we use
\begin{equation}
\label{eq:CommTangentialDerivatives}
\partial_r (r^{-1} \Omega A_\lambda) = r^{-1} \partial_r \Omega A_\lambda - r^{-2} \Omega A_\lambda.
\end{equation}
Using twice Lemma~\ref{lem:CTS-A},
we obtain
\[
\| \Omega^{\leq 2} \mathcal T A_\lambda\|_{L^2(C_{TS})} \lesssim  S^\frac12  \min \left\{
\lambda^{-\frac52}, \lambda^{\frac{1}{2}}\right\}, \qquad \|\partial_r \Omega^{\leq 2}  \mathcal T A_\lambda\|_{L^2(C_{TS})} \lesssim 
T^{-\frac12} \lambda^\frac12,
\]
where interpolating we get the $L^2$ bound on hyperboloids
\begin{equation}
 \| \Omega^{\leq 2} \mathcal T A_\lambda\|_{L^2(H\cap C_{TS})} \lesssim  \min \left\{ \lambda^{\frac{1}{2}}, \lambda^{-1} \right\},   
\end{equation}
and conclude using Sobolev embeddings on hyperboloids to obtain
\begin{equation}
 \|  \mathcal T A_\lambda\|_{L^\infty(H\cap C_{TS})} \lesssim T^{-\frac32} \min \left\{ \lambda^{\frac{1}{2}}, \lambda^{-1} \right\}.  
\end{equation}

We now turn our attention to the bounds for $\nabla A$, where to get pointwise bounds in $C_{TS}$ from the $L^2$ bounds we cannot just rely on vector fields, as these only span the directions tangent to hyperboloids. Instead, we also need to capture the transversal direction, which we do using the wave equation.
For this, it is convenient to write the wave equation in the polar frame, which has the form
\begin{equation}
\left(\partial_t^2 - \partial_r^2 -\frac{2}{r} \partial_r + \frac{1}{r^2} \Omega^2\right) A_\lambda  = G_\lambda.
\end{equation}
Using also the Lorentz boosts, we can rewrite this as 
\begin{equation}
\frac{(t^2-r^2)}{r^2} \partial_r^2 A_\lambda + \frac{1}{r} \Omega \nabla A_\lambda + \frac{1}{r} \nabla A_\lambda    = G_\lambda.
\end{equation}

Using Lemma~\ref{lem:CTS-A}, we can include 
all but the  first term on the left into the right hand side, arriving at an equation of the 
form
\begin{equation}\label{dr2}
   \frac{(t^2-r^2)}{r^2} \partial_r^2 A_\lambda = G_\lambda^1, 
\end{equation}
where $\partial_r A_\lambda$ and $G_\lambda^1$ satisfy
\begin{equation}\label{dr3}
   \| \Omega^{\leq 2} \partial_r A_\lambda\|_{L^2_{t,x}(C_{TS})}
   \lesssim T^{\frac12}  \min \left\{
\lambda^{-\frac52}, \lambda^{\frac{1}{2}}\right\},\qquad 
   \|\Omega^{\leq 2} G_\lambda^1\|_{L^{\frac32}_{t,x}(C_{TS})} \lesssim T^{-\frac13} + S^{\frac{1}{6}}\lambda^{\frac{1}{2}}.
\end{equation}
Writing $\lambda_* := \min\{\lambda^{-\frac{5}{2}}, \lambda^{\frac{1}{2}}\}$ and interpolating as in \eqref{interp-H} in order to bound the traces on hyperboloids, we obtain
\begin{align*}
 \| \Omega^{\leq 2} \partial_r A_\lambda\|_{L^{\frac{5}{3}}(H \cap C_{TS})}^{\frac{5}{3}}
   &\lesssim \| \Omega^{\leq 2} \partial_r A_\lambda\|_{L^2_{t,x}(C_{TS})}^{\frac{2}{3}} \| \Omega^{\leq 2} \partial_r^2 A_\lambda\|_{L^{\frac{3}{2}}_{t,x}(C_{TS})} \\
   &\qquad + S^{-1} (ST^3)^{\frac{1}{6}} \| \Omega^{\leq 2} \partial_r A_\lambda\|_{L^2_{t,x}(C_{TS})}^{\frac{5}{3}} \\
   &\lesssim (T^{\frac{1}{2}} \lambda_*)^{\frac{2}{3}} (S^{-1} T (T^{-\frac{1}{3}} + S^{\frac{1}{6}} \lambda^{\frac{1}{2}})) + S^{-\frac{5}{6}} T^{\frac{1}{2}} (T^{\frac{1}{2}} \lambda_* )^{\frac{5}{3}}
   \\
   &\lesssim  S^{-1} T \lambda_*^{\frac{2}{3}}  + S^{-\frac{5}{6}} T^{\frac{4}{3}} (\lambda_*^{\frac{2}{3}} \lambda^{\frac{1}{2}} + \lambda_*^{\frac{5}{3}})  \lesssim S^{-\frac{5}{6}} T^{\frac{4}{3}} \min\{\lambda^{-\frac{7}{6}}, \lambda^{\frac{5}{6}} \}.
\end{align*}
Finally, we apply Sobolev embeddings on hyperboloids to get
\begin{equation}
 \label{dr4A}
 \| \partial_r A\|_{L^\infty(H\cap C_{TS})} \lesssim T^{-\frac{9}{5}} \| \Omega^{\leq 2} \partial_r A\|_{L^{\frac{5}{3}}(H \cap C_{TS})} \lesssim T^{-\frac{9}{5}} S^{-\frac{1}{2}} T^{\frac{4}{5}} \lesssim S^{-\frac{1}{2}} T^{-1}.
\end{equation}

\subsection{An improved bound for the wave equation inside the cone}

Here we consider improved bounds for $\mathcal T A$ inside the cone. Our goal will be to prove the following
bound inside the cone
\begin{equation}
|\mathcal T A | \lesssim \frac{1}{\langle t \rangle^\frac32  \langle t-r \rangle^\frac12}.
\end{equation}
To accomplish this, we exploit energy  estimates on hyperboloids.
Precisely, we have 
\begin{equation}
\|\Omega^{\leq 3} \nabla A_\lambda\|_{L^2(H \cap C_{TS})}
\lesssim \frac{T^\frac12}{S^\frac12} \lambda^\frac12,
\end{equation}
which by Sobolev embeddings gives
\begin{equation}
\| \Omega \nabla A_\lambda \|_{L^\infty(H \cap C_{TS})}
\lesssim \frac{1}{T S^\frac12} \lambda^\frac12 .
\end{equation}
Commuting  $\Omega$ with $\nabla$  gives 
\begin{equation}
\| \nabla \mathcal T A_\lambda \|_{L^\infty(H \cap C_{TS})}
\lesssim \frac{1}{T^2 S^\frac12} \lambda^\frac12 , 
\end{equation}
and finally 
\begin{equation}\label{DTA}
|\nabla \mathcal T A_\lambda |
\lesssim \frac{\lambda^\frac12 }{\langle t \rangle^2  \langle t-r \rangle^\frac12} ,
\end{equation}
which is valid inside the cone.

On the other hand, our prior bound was 
\begin{equation}\label{TA}
| \mathcal T A_\lambda |
\lesssim \frac{\min\{\lambda^\frac12,\lambda^{-\frac12}\} }{\langle t \rangle^{\frac32}} .
\end{equation}
We seek to combine the last two bounds 
using the localization at frequency $\lambda$.
Without loss of generality we will assume $\lambda \lesssim 1$, as the bound for $\lambda > 1$ is similar to 
$\lambda = 1$, but better.

Using the localization at frequency $\lambda$, \eqref{DTA} would naively  imply
\begin{equation}
| \mathcal T A_\lambda |
\lesssim \frac{\lambda^{-\frac12} }{\langle t \rangle^2  \langle t-r \rangle^\frac12} .
\end{equation}
However we have to be careful that we have \eqref{DTA}
only inside the cone. So this is useful only 
in the region where $S \gg \lambda^{-1}$, while outside 
we have to rely on \eqref{TA}. Precisely, we 
consider a cutoff function $\chi$ which selects the region $t-r > \lambda^{-1}$, so that 
$\nabla \chi$ has size $\lambda$ and support 
in the region $t-r \approx \lambda^{-1}$. Then applying an inverse derivative at frequency $\lambda$ we write
\[
\mathcal T A_\lambda = \lambda^{-1} K_\lambda \ast \nabla \mathcal T A_\lambda ,
\]
where $K_\lambda$ is a regularizing kernel on the $\lambda^{-1}$ scale. Using the cutoff $\chi$
we write this as 
\[
\begin{aligned}
\mathcal T A_\lambda =& \lambda^{-1} K_\lambda \ast \nabla \chi \mathcal T A_\lambda
+ \lambda^{-1} K_\lambda \ast \nabla (1-\chi) \mathcal T A_\lambda\\
= &\lambda^{-1} K_\lambda \ast \nabla \chi \mathcal T A_\lambda
+ \lambda^{-1} \nabla K_\lambda \ast  (1-\chi) \mathcal T A_\lambda\\
=& 
\lambda^{-1} K_\lambda \ast  \chi \nabla \mathcal T A_\lambda
+
\lambda^{-1} K_\lambda \ast (\nabla \chi) \mathcal T A_\lambda
+ \lambda^{-1} \nabla K_\lambda \ast  (1-\chi) \mathcal T A_\lambda .
\end{aligned}
\]
We evaluate this in the interior region $t-r > \lambda^{-1}$.
For the first term we use \eqref{DTA} to obtain
\[
\lesssim \frac{\lambda^{-\frac12}}{\langle t \rangle^2  \langle t-r \rangle^\frac12} .
\]
For the other two terms we use \eqref{TA}, which gives 
\[
\lesssim \frac{\min\{\lambda^\frac12,\lambda^{-\frac12}\} }{\langle t \rangle^{\frac32}} 
\left( \frac{\lambda^{-1}}{\langle t-r \rangle}\right)^N, 
\]
which decays rapidly away from $t-r \approx \lambda^{-1}$. We have proved that 
\[
|\mathcal T A_\lambda| 
\lesssim \frac{\lambda^{-\frac12}}{\langle t \rangle^2  \langle t-r \rangle^\frac12} + \frac{\min\{\lambda^\frac12,\lambda^{-\frac12}\} }{\langle t \rangle^{\frac32}} 
\left( \frac{\lambda^{-1}}{\langle t-r \rangle}\right)^N, \quad \mbox{when } t-r>\lambda^{-1}.
\]

We combine this with \eqref{TA} in the region $ 0 < t-r <\lambda^{-1}$. Considering all cases 
and summing up with respect to $\lambda < 1$
we obtain 
\[
|\mathcal T A_\lambda| \lesssim \frac{1}{\langle t \rangle^\frac32  \langle t-r \rangle^\frac12}.
\]

\end{proof}

\section{Asymptotic analysis for \protect{$\psi$}}\label{s:wp} 

As noted earlier, for our
main bootstrap argument we assume that our solution $(\psi,A)$ satisfies the bootstrap assumption \eqref{eq:boot-psi-A}
in a region $C_{<T}$, and then we seek to improve the constant in the same region.
The goal of this section is to provide 
a more accurate asymptotic analysis for $\psi$, which in particular will allow 
us to close the $\psi$ component of the 
bootstrap, as well as to prove the pointwise bound for $\psi$ in \eqref{thm-point} in Theorem~\ref{thm:gwp-easy}.

For motivation, we begin by observing that combining the energy estimates for $\psi$ in Section~\ref{s:vf} with the Klainerman-Sobolev inequalities in Section~\ref{s:ks} implies a 
$t^{-\frac32+ C\epsilon}$ decay bound for $\psi$
inside the cone, which is not sufficient in order to close the bootstrap bound. 
This is what makes the analysis in this section necessary.

We will study the asymptotic behavior of the spinor $\psi$ via the method of testing by wave packets, introduced in~\cite{IT1, IT2} in the context of the nonlinear Schr\"odinger equation and water waves, respectively. The idea is to capture the asymptotic profile $\gamma(t,v)$ of the solution $\psi$ by testing it with a wave packet, i.e. an approximate solution of the corresponding linear problem, which travels along the ray $x = v t$. Here, $v \in B(0,1)$ spans the range of allowed group velocities for the massive Dirac flow.

\subsection{A heuristic derivation of the asymptotic equation}
Before we delve into the wave packet analysis, 
we begin with a heuristic computation of the asymptotic equation, which will serve as a guide for the more precise analysis later on.

Our starting point is the Klein-Gordon equation, 
which is satisfied by solutions to the linear homogeneous Dirac flow. The fundamental solution 
has $t^{-\frac32}$ decay inside the cone, and two phases
of the form 
\[
\phi_{\pm} = \pm \sqrt{t^2-x^2}
\]
associated to the two half-waves. Hence, if a function $u$ solves the Klein-Gordon equation 
\[
(\Box + 1) u = 0,
\]
and has smooth and localized initial data, then it will have an asymptotic expansion of the form
\[
u(t,x) \approx \sum_{\pm} (t^2-x^2)^{-\frac34} e^{i\phi_{\pm}} \rho^{\pm}(v), \qquad v = x/t, 
\]
with  nice functions $\rho^{\pm}$ with support inside the unit ball and rapid decay at the boundary. Here for the amplitude we have preferred the Lorentz invariant factor $t^2-x^2$ over the power of $t$,
which will be compensated by the decay of $\rho$ at the
boundary of the ball.

Now we turn our attention to the linear Dirac equation, beginning with the homogeneous one
\[
- i \gamma^\alpha \partial_\alpha \psi + \psi = 0.
\]
Assuming nice and localized initial data, 
its solutions will also have an asymptotic expansion as in the Klein-Gordon case above,
\begin{equation}\label{Dirac-asymptotic}
\psi(t,x) \approx \sum_{\pm} (t^2-x^2)^{-\frac34} e^{i\phi_{\pm}} \rho^{\pm}(v), \qquad v = x/t. 
\end{equation}
But unlike the Klein-Gordon case where the asymptotic 
profiles $\rho^{\pm}$ are independent complex valued
functions tied to the initial position and velocity,
here we have instead a first order system, which will 
yield algebraic constraints on $\rho^{\pm}$. To see
how this works, we apply the Dirac operator to the 
function $\psi$ in \eqref{Dirac-asymptotic}. This 
will be a consistent expansion if the error has at least $t^{-1-}$ better decay at infinity.  We organize
the terms by the decay rate at infinity:
\[
\begin{aligned}
  (-i \gamma^\alpha \partial_\alpha + 1) \psi   = & \ 
(t^2-x^2)^{-\frac34} e^{i\phi_{\pm}} (1 \mp \frac{x_\alpha}{\sqrt{t^2-x^2}} \gamma^\alpha)   \rho^{\pm}(v) 
\\
& \ - i \frac32 (t^2-x^2)^{-\frac74} e^{i\phi_{\pm}} {x_\alpha} \gamma^\alpha \rho^{\pm}(v) 
\\
& \ - i  (t^2-x^2)^{-\frac34} e^{i\phi_{\pm}} \gamma^\alpha \partial_\alpha \rho^{\pm}(v)
\\
:= & \ R_1^\pm + R_2^\pm + R_3^\pm,
\end{aligned}
\]
where the three terms $R_1$, $R_2$, $R_3$ have decay 
$t^{-\frac32}$, $t^{-\frac52}$, $t^{-\frac52}$, neither 
of which is an acceptable error. The first priority here
is to cancel the $R_1$ term, which involves the  projectors  $P^{\pm}_v$ defined in \eqref{def proj Pv}, which for convenience we recall here
\[
2 P^{\pm}_v = 1 \mp \frac{x_\alpha}{\sqrt{t^2-x^2}} \gamma^\alpha .
\]
Our cancellation condition thus has the form
\begin{equation}\label{cancel}
P^{\pm}_v \rho^{\pm}(v) = 0.    
\end{equation}

Assuming the cancellation condition \eqref{cancel}, 
the two asymptotic profiles $\rho^{\pm}(v)$ are restricted to the respective subspaces $V^{\pm}_{v}$.
We can also simplify the expression 
\[
R_2^{\pm} = \mp  i\frac32 (t^2-x^2)^{-\frac54} e^{i\phi_{\pm}} \rho^{\pm}(v),
\]
which is then seen to belong to $V^{\pm}_v$.

Finally we consider $R_3^{\pm}$, which apriori has  
both a $V^\pm_v$ and a $V^{\mp}_v$ component. The $V^\mp_v$
component is less important as it is nonresonant 
and can be eliminated with a lower order $t^{-\frac52}$
correction to $\psi$. But the $V^{\pm}_v$ component 
has to cancel $R_2^\pm$. Hence we compute 
\begin{equation}\label{R3pm}
\begin{aligned}
P^\mp_v R_3^\pm = & \ - i  (t^2-x^2)^{-\frac34} e^{i\phi_{\pm}} P^\mp_v \gamma^\alpha \partial_\alpha \rho^{\pm}(v)
\\
= & \ - i  (t^2-x^2)^{-\frac34} e^{i\phi_{\pm}}\gamma^\alpha P^{\pm}_v  \partial_\alpha \rho^{\pm}(v)
\pm i  (t^2-x^2)^{-\frac54} e^{i\phi_{\pm}} x^\alpha  \partial_\alpha \rho^{\pm}(v)
\\
= & \  i  (t^2-x^2)^{-\frac34} e^{i\phi_{\pm}}\gamma^\alpha  (\partial_\alpha P^{\pm}_v)  
\rho^{\pm}(v),
\end{aligned}
\end{equation}
where we have used the twisted commutation relation 
\begin{equation}\label{twist-com}
P^{\mp}_v \gamma^\alpha = \gamma^\alpha P^{\pm}_v \mp 
\frac{x^\alpha}{\sqrt{t^2-x^2}} I_4.
\end{equation}
Finally we compute 
\begin{equation}\label{dP-pm}
\gamma^\alpha  \partial_\alpha P^{\pm}_v = \pm
\frac12 \gamma^\alpha \gamma^\beta \left(\frac{g_{\alpha\beta}}{\sqrt{t^2-x^2}}
+ \frac{x_\alpha x_\beta}{(t^2-x^2)^{\frac32}}\right)
= \pm\frac32 \frac{1}{\sqrt{t^2-x^2}} I_4,
\end{equation}
which shows that we indeed have the cancellation
\[
R_2^\pm+ P^\mp_v R_3^\pm= 0.
\]
This concludes our justification of the leading order 
asymptotic expansion \eqref{Dirac-asymptotic}, 
under the necessary and sufficient condition \eqref{cancel}.

Now we turn our attention to the magnetic Dirac equation,
\[
  - i \gamma^\alpha \partial_\alpha \psi + \psi = A_\alpha \gamma^\alpha \psi
\]
with a magnetic potential $A$ which has $t^{-1}$ decay 
at infinity. This decay rate guarantees that the contribution of $A$ cannot be seen as perturbative 
for $\psi$ as in \eqref{Dirac-asymptotic}, which implies that the expansion \eqref{Dirac-asymptotic}
cannot hold at the leading order; in other words, if $A$ has just $t^{-1}$ decay along rays then classical scattering cannot hold for our solutions. 
Instead, we look for a corrected asymptotic expansion of the form
\begin{equation}\label{Dirac-asymptotic-nl}
\psi(t,x) \approx \sum_{\pm} (t^2-x^2)^{-\frac34} e^{i\phi_{\pm}} \rho^{\pm}(t,v), \qquad v = x/t,
\end{equation}
where we allow for a slow remodulation of $\rho^{\pm}$
along rays. For this we can repeat the computation above, with two differences:

i) We need to add the contribution of $A$. This is 
\[
(t^2-x^2)^{-\frac34} e^{i\phi_{\pm}} A_\alpha \gamma^\alpha \rho^{\pm}(t,v).
\]

ii) The scaling derivative of $\rho^{\pm}$ no longer vanishes. This arises in the $R_3^{\pm}$
computation in \eqref{R3pm}, precisely in the term
\[
  \pm i  (t^2-x^2)^{-\frac54} e^{i\phi_{\pm}} x^\alpha  \partial_\alpha \rho^{\pm}(t,v).
\]

The contribution of $A$ has size $t^{-\frac52}$,
so it does not affect the leading order cancellation that gives the relations \eqref{cancel}.  Instead, at the $t^{-\frac52}$ order the two terms need to be equal at least when projected onto $V^\pm_v$. This yields the relation
\[
  \pm i (t^2-x^2)^{-\frac12}  x^\alpha  \partial_\alpha \rho^{\pm}(t,v) = A_\alpha P_v^\mp \gamma^\alpha \rho^{\pm}(t,v).
\]
Using the identity \eqref{twist-com} and the constraint \eqref{cancel} on the right, this becomes
\begin{equation*}
 i t \partial_t \rho^{\pm}(t,v) =  i  x^\alpha  \partial_\alpha \rho^{\pm}(t,v) = - x^\alpha A_\alpha \rho^{\pm}(t,v),
\end{equation*}
i.e. the ode
\begin{equation} \label{asympt-eqn}
  i \partial_t \rho^\pm(t,v) = - v^\alpha A_\alpha \rho^{\pm}(t,v)  
\end{equation}
along rays $x = vt$, which we will refer to as the \emph{modulation equation}, and which  causes the \emph{asymptotic profile} $\rho^{\pm}(t,v)$ to rotate within $V^{\pm}$ on the slow time scale $s = \ln t$.

This concludes our heuristic analysis. The aim of the rest of this section will be to construct a good choice of the asymptotic profiles $\rho^{\pm}$, which has the following key properties:

\begin{itemize}
    \item We have the asymptotic expansion \eqref{Dirac-asymptotic-nl}, with lower order 
    errors.
    \item The asymptotic profiles $\rho^{\pm}$  are approximate solutions to the asymptotic equation \eqref{asympt-eqn}, again with lower order errors.
\end{itemize}
Once this is done, we can use the asymptotic equation \eqref{asympt-eqn} in order to prove uniform bounds for the asymptotic profile, and then transfer these bounds to $\psi$ using the expansion \eqref{Dirac-asymptotic-nl}. This in turn will close the bootstrap hypothesis for $\psi$.

\subsection{Wave packets for the Dirac equation}
The main idea of the method of wave packet testing, as developed in \cite{IT1, IT2, IT3},  is that 
the asymptotic profiles can be obtained by testing the solution $\psi$ with well chosen wave
packets, which are approximate solutions to the 
linear equation traveling along rays with given velocity $v$. Compared with the above references,
here there are two twists to this story:
\begin{itemize}
\item corresponding to each admissible velocity $v$ inside the unit ball there will be two wave  
packets $\psi_v^\pm$, associated to the two phase choices $\phi_{\pm}$.
\item the wave packets have an additional amplitude parameter $\sigma^{\pm} \in V^{\pm}_v$.
\end{itemize}

We begin with the construction of our wave packets. For each velocity $v$ with $|v|<1$ we need to produce
two such packets $\psi_v^\pm$, localized around 
the ray $x = vt$, which we will use in order to track the $\pm$ components of $\psi$ along the same ray. The simplest ansatz we can make is 
\begin{equation}\label{first-ansatz}
\psi^{\pm}_v(t,x) = \chi(y) e^{i\phi_\pm}  \sigma^{\pm}(x/t),
\qquad y = d_H^2(x/t,v)(t^2-x^2)^{\frac{1}{2}},
\end{equation}
where the entries in the above expression are as follows:
\begin{itemize}
\item $\chi$ is a  compactly supported bump function, normalized to have unit integral as a radial function in $\mathbb{R}^3$.

\item $d_H(x/t,v)$ represents the hyperbolic distance between $x/t$ and $v$,  viewed as points in the Beltrami-Klein disc model for the hyperbolic space, which can be 
written explicitly as
\[
\cosh^2 [d_H(v,w)] = 1+\frac{2|v-w|^2}{(1-|v|^2)(1-|w|^2)}
\]
with the simpler case 
\[
 d_H(0,v) = \frac12 \ln \frac{1+|v|}{1-|v|}. 
\]

\item The functions $\sigma^{\pm}$ are chosen to belong 
to the correct subspace, 
\begin{equation}
\sigma^{\pm}(x/t) \in V^{\pm}_{x/t}.
\end{equation}
For a more accurate choice, we select $\sigma^{\pm}(v)$ arbitrarily, and -- employing the Frobenius theorem -- uniquely extend it smoothly nearby respecting the Lorentz invariance, i.e. so that 
\begin{equation}
\label{hat}
\hat \Omega \sigma^{\pm} = 0.
\end{equation}
\end{itemize}
The motivation for these particular choices 
is to simplify our computation by making our 
wave packet choice invariant with respect to Lorentz transformations. We also note that we will use the wave packets only in the region $\{t^2 - x^2 \geq 1\}$, where the compact support of $\chi$ implies that $(t,x)$ stays away from the cone.

\medskip

Now we apply the Dirac operator to $\psi_v^{\pm}$,
which is a computation somewhat similar to the one in the previous subsection. To avoid distracting technicalities, in view of the Lorentz invariance it suffices to consider 
the simplest case when $v=0$, where
we get  
\[
\begin{aligned}
(- i \gamma^\alpha \partial_\alpha + 1)\psi^{\pm}_v
= & \  2 \chi(y)e^{i\phi_\pm}  P_{x/t}^\pm \sigma^{\pm}(x/t)
\\ 
  & \
 - i \chi(y)e^{i\phi_\pm}  \gamma^\alpha \partial_\alpha \sigma^{\pm}(x/t)
\\ & \ 
 + i (t^2-x^2)^{-\frac{1}{2}}  \chi'(y)  e^{i\phi_\pm}
d^2_H x_\alpha \gamma^\alpha \sigma^{\pm}(x/t)
\\ & \ 
 - i (t^2 - x^2)^{\frac{1}{2}}
\chi'(y) e^{i\phi_\pm} \partial_\alpha d^2_H  \gamma^\alpha \sigma^{\pm}(x/t)
\\ =: & R_1^{\pm}+ R_2^{\pm} +R_3^{\pm} +R_4^{\pm}.
\end{aligned}
\]
Now we consider each of these four terms, separating the $V^\pm_{x/t}$ and the $V^\mp_{x/t}$ contributions.
The $V^\pm_{x/t}$ terms are the important ones, which we
need to explicitly determine. The $V^\mp_{x/t}$ terms 
on the other hand can be thought off as nonresonant, and can be removed with a lower order
correction to our initial ansatz in \eqref{first-ansatz}. As a rule of thumb,
the errors which are more than $t^{-1}$ better are purely perturbative, while 
all larger errors are not and will have to have some structure.

\medskip

a) The first term $R_1^\pm$ vanishes since $\sigma^\pm$ is chosen to be in the correct subspace $V^\pm_{x/t}$. 

\medskip
b) The second term also has size $t^{-1}$ better, and a simple $V^\pm_{x/t}$ component. Indeed, using \eqref{twist-com} followed by \eqref{dP-pm}, we get
\[
P_{x/t}^\mp R_2^\pm = \pm \frac{3i}2  (t^2-x^2)^{-\frac{1}{2}}\chi(y) e^{i\phi_\pm} \sigma^{\pm}(x/t).
\]

\medskip 
c) The third term also has size 
$t^{-1}$ better, and its $V^\pm_{x/t}$ projection is
\[
P^{\mp}_{x/t} R_3^\pm = R_3^\pm = \pm i    \chi'(y)  e^{i\phi_\pm}  d^2_H \sigma^{\pm}(x/t).
\]

\medskip
d) The fourth term is the worst, as it only has  size
$t^{-\frac{1}{2}}$ better. However, using again \eqref{twist-com}, its $V^\pm_{x/t}$ projection is better,
\[
P_{x/t}^\mp R_4^{\pm} = \pm i \chi'(y) e^{i\phi_\pm} x^\alpha \partial_\alpha d^2_H\sigma^{\pm}(x/t) = 0.
\]
\bigskip

To summarize, we have established that
\begin{equation}\label{D-packet}
\begin{aligned}
(- i \gamma^\alpha \partial_\alpha + 1)\psi^{\pm}_v
= & \  \frac{i}{2} (t^2-x^2)^{-\frac{1}{2}}\left(   
   \pm 3 \chi(y)   \pm 2\chi'(y)   d^2_H(x/t,0) (t^2-x^2)^{\frac{1}{2}}\right) e^{i\phi_\pm} \sigma^{\pm}(x/t)
\\ & \
+ P_{x/t}^\pm (R_2^\pm  + R_4^\pm).
\end{aligned}
\end{equation}
To eliminate the terms on the second line we redefine our wave packet by setting
\begin{equation}\label{corrected-ansatz}
\begin{aligned}
\tilde{\psi}^{\pm}_v : =  & \ \psi^{\pm}_v -\frac{1}{2} P^{\pm}_{x/t} (R_2^\pm  + R_4^\pm)
\\  =: &\  e^{i\phi_\pm}  (\chi(y)\sigma^{\pm}(x/t)
+ \underline{\sigma}^{\pm}(x/t)),
\end{aligned}
\end{equation}
where $\underline{\sigma}^{\pm}$ has similar regularity and localization as the first term, but has size $(t^2-x^2)^{\frac{1}{4}}$ smaller. Applying the linear Dirac operator to $\psi_v^\pm$ we cancel the expression $P_{x/t}^\pm (R_2^\pm + R_4^\pm)$, obtaining instead additional error terms. These are all $(t^2-x^2)^{-\frac34}$ 
better than $\psi_v^{\pm}$ with a single exception, namely the counterpart of $R_4^{\pm}$,
which we denote by $R_5^{\pm}$,  obtaining
\begin{equation}\label{D-packet+}
\begin{aligned}
(- i \gamma^\alpha \partial_\alpha + 1)\tilde \psi^{\pm}_v
= & \  \frac{i}{2} (t^2-x^2)^{-\frac{1}{2}}\left(   
   \pm 3 \chi(y)   \pm 2\chi'(y)   d^2_H(x/t,0) (t^2-x^2)^{\frac{1}{2}}\right) e^{i\phi_\pm} \sigma^{\pm}(x/t)
\\ & \ 
+  R_5^\pm + O((t^2-x^2)^{-\frac34}) .
\end{aligned}
\end{equation}
The $V^\mp_{x/t}$ component of $R_5^{\pm}$ can be corrected again so it remains to compute the $V^\pm_{x/t}$ component, which is
\[
\begin{aligned}
P^\mp_{x/t} R^\pm_5 =&\frac{1}{2}(t^2-x^2)^{\frac{1}{2}} e^{i\phi_\pm} \left( \chi'(y) \partial_{\alpha}\partial_{\beta} d^2_H(x/t,0) + (t^2-x^2)^{\frac{1}{2}} \chi''(y) \partial_{\beta} d^2_H(x/t,0) \partial_{\alpha} d^2_H(x/t,0)\right)\\
&\hspace*{10.5cm} \cdot P^{\mp}_{x/t} \gamma^{\beta}P^{\pm}_{x/t} \gamma^{\alpha}\sigma^{\pm}(x/t).
\end{aligned}
\]
A direct computation using the relation \eqref{twist-com} gives us
\[
P^{\mp}_{x/t}\gamma^{\beta}P^{\pm}_{x/t}\gamma^{\alpha}\sigma^{\pm}(x/t)=g^{\alpha\beta}\sigma^{\pm}(x/t)\mp \frac{x^{\alpha}x^{\beta}}{t^2-x^2} \sigma^{\pm}(x/t).
\]
The second term on the right does not contribute to the outcome because it yields scaling derivatives of $d_H$, so we are left with
\[
\begin{aligned}
P^\mp_{x/t} R^\pm_5 =&(t^2-x^2)^{\frac{1}{2}} e^{i\phi_\pm} g^{\alpha\beta} \left( \chi'(y) \partial_{\alpha}\partial_{\beta} d^2_H + (t^2-x^2)^{\frac{1}{2}} \chi''(y) \partial_{\beta} d^2_H \partial_{\alpha} d^2_H\right) \sigma^{\pm}(x/t).
\end{aligned}
\]
We can simplify this expression modulo lower order terms, using $d_H^2(x/t,0) \approx 
x^2/t^2$ modulo quartic errors. In three dimensions this gives
\[
g^{\alpha \beta} \partial_{\alpha}\partial_{\beta} d^2_H \approx 
\frac{6}{t^2-x^2}, \qquad g^{\alpha\beta}\partial_{\beta} d^2_H \partial_{\alpha} d^2_H  \approx \frac{4 d_H^2}{t^2-x^2},
\]
which finally yields
\begin{equation}
  P^\mp R^\pm_5 =(t^2-x^2)^{-\frac{1}{2}} e^{i\phi_\pm}  \left( 6\chi'(y)  + 4(t^2-x^2)^{\frac{1}{2}} d^2_H \chi''\right) \sigma^{\pm}(x/t)  + O((t^2-x^2)^{-\frac34}).
\end{equation}
To summarize, our final wave packet choice is 
\begin{equation}\label{psi-tilde-tilde}
\begin{aligned}
\tilde{\psi}^{\pm}_v : =  & \ \psi^{\pm}_v -\frac{1}{2} P^{\pm}_{x/t} (R_2^\pm + R_3^\pm+ R_4^\pm + R_5^{\pm})
\\  =: &\  e^{i\phi_\pm}  (\chi(y)\sigma^{\pm}(x/t)
+ \underline{\sigma}^{\pm}(x/t)),
\end{aligned}
\end{equation}
which when inserted into the linear Dirac equation gives
\begin{equation}
\label{bun}
    \begin{aligned}
(- i \gamma^\alpha \partial_\alpha + 1)\tilde \psi^{\pm}_v
= & \  \frac{i}{2} (t^2-x^2)^{-\frac{1}{2}}\left(   
   \pm 3 \chi(y)   \pm 2\chi'(y)   d^2_H(x/t,0) (t^2-x^2)^{\frac{1}{2}}\right) e^{i\phi_\pm} \sigma^{\pm}(x/t)
\\ & \ 
+  (t^2-x^2)^{-\frac{1}{2}} e^{i\phi_\pm}  \left( 6\chi'(y)  + 4(t^2-x^2)^{\frac{1}{2}} d^2_H \chi''(y)\right) \sigma^{\pm}(x/t)  + O((t^2-x^2)^{-\frac34}).
\end{aligned}
\end{equation}

We drop the tilde in the following and write $\psi_v^\pm$ again for our wave packets. This is a good approximate solution for the linear Dirac system as follows:

\begin{lem}
The wave packets $\psi_v^{\pm}$ defined above solve an equation of the form
\begin{equation}\label{fv}
(- i \gamma^\alpha \partial_\alpha + 1)   \psi_v^{\pm} = (t^2-
x^2)^{-\frac{1}{2}} e^{i\phi_{\pm}} r_v^{\pm} := f_v^{\pm}
\end{equation}
in the region $\{t^2 - x^2 \geq 1\}$, where the errors amplitudes $r_v^\pm$ have a similar size (in the $H$ norm), localization and regularity as our original bump function $ \chi(y)$. Furthermore, $r_v^{\pm}$ 
admit a representation of the form
\begin{equation}\label{rv-rep}
r_v^{\pm} =(t^2-x^2)^{-\frac14} (\hat \Omega \tilde r_v^{\pm} +\tilde{\tilde r}_v^{\pm} )
\end{equation}
with $\tilde r_v^{\pm}$  and $\tilde{\tilde r}_v^{\pm}$ also in the same class.
\end{lem}

\begin{proof}
By Lorentz invariance, it suffices to consider the case $v=0$.  We need to consider the right hand side terms in \eqref{bun}. We begin with the term on the first line, 
\[
r_v^{\pm} =( 3\chi(y)  +2\chi'(y)   d^2_H(x/t,0)
(t^2-x^2)^{\frac{1}{2}} )\sigma^\pm(x/t).
\]
For this one can directly guess the appropriate representation, and compute
the error corresponding to this term
\[
\begin{aligned}
(t^2 - x^2)^{-\frac{1}{4}} \tilde{\tilde r}_v^{\pm} := & 
  \ r_v^{\pm} - \hat\Omega_{0j} \Big(\frac{x^j \chi(y)}{ \sqrt{t^2-x^2}}\sigma^\pm(x/t)\Big)\\
= & \   r_v^{\pm} - \Omega_{0j} \Big(\frac{x_j \chi(y)}{ \sqrt{t^2-x^2}}\Big)\sigma^{\pm}(x/t)
\\ 
= & \ 3\chi(y) \Big(1 -\frac{t}{\sqrt{t^2-x^2}}\Big) \sigma^{\pm}(x/t)  
\\ & +\chi'(y)  ( 2(t^2-x^2)^{\frac{1}{2}} d^2_H(x/t,0) - 
x_j (t \partial_j + x^j \partial_t) d^2_H(x/t,0)) \sigma^{\pm}(x/t).
\end{aligned}
\]
This corresponds to the choices
\[
\tilde r^\pm_{v,j} = (t^2 - x^2)^{\frac{1}{4}} \frac{x^j \chi(y)}{ \sqrt{t^2-x^2}}\sigma^\pm(x/t),
\]
which are acceptable since $|x| \lesssim (t^2-x^2)^{\frac14}$ in the support of $\chi(y)$.

At this point it suffices to observe that 
\[
d^2_H(v,0) = |v|^2 + O(|v|^4),
\]
which shows that the quadratic term cancels 
in the second expression, and we arrive at
\[
\tilde{\tilde r}_v^{\pm} = (t^2 - x^2)^{\frac{1}{4}} 
(\chi(y) O(\frac{x^2}{t^2}) + \chi'(y) O(\frac{x^2}{t^2})).
\]
Since $|x| \lesssim t^\frac12 \approx (t^2-x^2)^\frac14$ in the support of $\chi(y)$ and $\chi'(y)$, this suffices for the conclusion of the Lemma.

The term on the second line of \eqref{bun} is similar, just with $\chi$ replaced by $\chi'$.
\end{proof}

\subsection{Asymptotic profiles  via wave packet testing} Now that we have our wave packets, 
the next step is to use them in order to construct asymptotic profiles for the Dirac component $\psi$ of our solution. 

Departing slightly from the standard approach 
introduced in \cite{IT4}, here we will do the testing on hyperboloids, rather then on time slices. There are two primary reasons for that:
\begin{itemize}
    \item to gain orthogonality between $V^+$ and $V^-$ in the energy functional.
    \item to facilitate integrations by parts 
    for the Lorentz vector fields.
\end{itemize}

We begin by recalling the conserved energy functional associated to the linear Dirac equation
on hyperboloids, which has the form (see Section~\ref{s:not})
\[
E_H(\psi) =  \int_H (t^2-x^2)^\frac32 \langle  \psi, \psi \rangle_H \, dV_H.
\]

\medskip

We use the above energy to define our asymptotic profiles
$\rho^{\pm}$ via wave packet testing:

\begin{de}
Let $t > 0$, $v \in B(0,1)$ so that $(t,vt) \in H$. 
Then the asymptotic profiles $\rho^{\pm}(t,v) \in V^{\pm}_v$ associated to $\psi$ are uniquely determined by
\begin{equation}
\label{eq:Defrhopm}
\langle \rho^{\pm}, \sigma_j^\pm(v)\rangle_H := (t^2-(vt)^2)^\frac32
\int_{H} \langle \psi, \psi_v^{\pm}\rangle_H \, dV_H, \quad j=1,2,
\end{equation}
where $\psi_v^{\pm}$ are as in \eqref{corrected-ansatz} and  $\sigma^\pm_1(v),\sigma^\pm_2(v)$ are chosen to form an orthonormal base with respect to $\langle \cdot , \cdot \rangle_H$ of $V^\pm_v$ and so that the last bullet point below  \eqref{first-ansatz} is satisfied.
\end{de}
In the following, we will sometimes drop the subscript on $\sigma_j^\pm$.

We first consider pointwise and asymptotic bounds for the asymptotic profiles:

\begin{lem}
The asymptotic profiles $\rho^{\pm}$ satisfy the pointwise bounds
\begin{equation}\label{rho-point}
\| \rho^{\pm}(t,v)\|_{H} \lesssim \epsilon t^{-2+C\epsilon} (t^2-(tv)^2), 
\end{equation}
and the $L^2$ bounds
\begin{equation}\label{rho-2}
\| \Omega^{\leq 3} \rho^{\pm}\|_{L^2(H_{\leq t})} \lesssim t^{C \epsilon}    .
\end{equation}
\end{lem}
\begin{proof}
  The estimate \eqref{rho-point} is an immediate consequence of the uniform bound 
\begin{equation}\label{psi-point}
\|\psi(t,x) \|_{H} \lesssim \epsilon t^{-2+C\epsilon}(t^2-x^2)^\frac14,
\end{equation}
which in turn follows from \eqref{eq:PointwiseBoundspsiplpsimin} in the proof of Theorem~\ref{t:KS}.
Note that we can express the $H$ norm of $\psi$ in terms of $\psi_+$ and $\psi_-$ as 
\[
\| \psi\|_{H}^2 = \left(\frac{t+r}{t-r}\right)^\frac12|\psi_{-}|^2 + 
\left(\frac{t-r}{t+r}\right)^\frac12
|\psi_+|^2 ,
\]
which can be verified by direct computation.

The estimate \eqref{rho-2} is a consequence of the $L^2$ energy bounds on hyperboloids
which follow from the $X_T$ bounds in Section~\ref{s:vf}. More precisely, note that we can replace $\psi$ by $\Omega^{\leq 3} \psi$ in~\eqref{Energy-H}, which yields 
\begin{equation}\label{psi2}
E_H(\Omega^{\leq 3} \psi) \lesssim \epsilon^2 t^{c\epsilon}
\end{equation} 
by Proposition~\ref{prop:EnergyEstGamma}.
\end{proof} 
Our next objective is to establish that the above asymptotic profiles satisfy a precise form 
of the asymptotic formula \eqref{Dirac-asymptotic-nl}:

\begin{lem}\label{l:asympt}
We have the following asymptotic formula:
\begin{equation}\label{exact-Dirac-asymptotic-nl}
\psi(t,vt) = \sum_{\pm} (t^2-(vt)^2)^{-\frac34} e^{i\phi_{\pm}} \rho^{\pm}(t,v) + O(\epsilon (t^2-r^2)^{-\frac78} t^{C\epsilon} ) 
\end{equation}
in the region $\{ t^2-x^2 \geq 1\}$.
\end{lem}

\begin{proof}
Splitting into the $\pm$ components and testing 
against $\sigma_j^{\pm}$, it suffices to show that
\[
\langle \psi(t,vt), \sigma_j^{\pm}(v)\rangle_H = 
(t^2-(vt)^2)^{\frac34} e^{i\phi_{\pm}} \int_H \langle \psi,\psi_v^\pm \rangle_H \, dV_H + O( \epsilon (t^2-r^2)^{-\frac78} t^{C\epsilon}  ).
\]

The phase in $\psi^{\pm}_v$ cancels the one in front, so we are left with 
\[
\langle \psi(t,vt), \sigma_j^{\pm}(v)\rangle_H = 
(t^2-(vt)^2)^{\frac34} \int_H \langle \psi,\chi(y)\sigma^{\pm} + \underline{\sigma}^{\pm} \rangle_H \, dV_H + O(\epsilon (t^2-(vt)^2)^{-\frac78}t^{c\epsilon}  ).
\]
 To prove this we use again the uniform bound \eqref{psi-point} and 
the $L^2$ energy bounds \eqref{psi2}.

The uniform bound \eqref{psi-point} is used to estimate directly the 
contribution of $\underline{\sigma}^\pm$, which has size $(t^2-x^2)^{-\frac14}$ times $\sigma^{\pm}(v)$ and support similar to $\chi$, which has size 
$(t^2-x^2)^{-\frac14}$. Then we can estimate
\[
\begin{aligned}
Err^1:=& \ \left| (t^2-(vt)^2)^{\frac34} \int_H \chi(y)\langle \psi,\underline{\sigma}^{\pm} \rangle_H \, dV_H \right|
\\ \lesssim & \  \epsilon
(t^2-x^2)^{\frac34} (t^2-x^2)
^{\frac12}
t^{-\frac52+C\epsilon} (t^2-x^2)^{-\frac14} (t^2-x^2)^{-\frac34}
\| \sigma^{\pm}(v) \|_{H}
\\
\approx & \ \epsilon t^{-\frac52+C\epsilon} (t^2-x^2)^{\frac14} .
\end{aligned}
\]

The $L^2$ bound \eqref{psi2} is used to estimate the remaining error,
\[
Err^2 := \langle \psi(t,vt), \sigma^{\pm}(v)\rangle_H - 
(t^2-(vt)^2)^{\frac34} \int_H \langle \psi,\chi(y)\sigma^{\pm} \rangle_H \, dV_H.
\]
We denote 
\[
f(w) := \langle \psi(s,w s), \sigma^{\pm}(w)\rangle_H,
\]
where $s$ is chosen such that $(s, w s)$ is on the hyperboloid through $(t, v t)$.   From \eqref{psi2} we have 
\begin{equation}
\| \Omega^{\leq 2} f \|_{L^2(H)} \lesssim \epsilon (t^2-r^2)^{-\frac34} t^{C\epsilon}.   
\end{equation}
Rewriting 
\[
Err^2 = f(v) - (t^2-(vt)^2)^{\frac34} \int_H \chi((t^2-(tv)^2)^{\frac12} d^2(w,v)) f(w) \, dV_H(w),
\]
this can be essentially thought of as the difference 
between $f$ and its local average on the $(t^2-x^2)^{-\frac14}$ scale, up to an error that can be included in $\underline{\sigma}^{\pm}$. 

To understand the role of the scale we start with a spherically symmetric average on the unit scale in the Euclidean setting, where 
in three space dimensions we can estimate
\[
| f(0) - \int \chi_1 f\, dx | \lesssim  
\| \nabla^2 f\|_{L^2}.
\]
Rescaling this to a ball of radius $\mathfrak r$ we get the same bound
\[
| f(0) - \int \chi_{\mathfrak r} f dx | \lesssim   
\mathfrak r^\frac12 \| \nabla^2 f\|_{L^2} .
\]
In the hyperbolic space the Lorentz vector fields play exactly the role of unit derivatives, so applying this to 
$Err^2$ we obtain
\[
|Err^2| \lesssim \epsilon (t^2-r^2)^{{-\frac78}} t^{C\epsilon}.   \]
The conclusion of the lemma follows by adding the bounds for $Err^1$ and $Err^2$.
\end{proof}

The next objective of our asymptotic analysis 
is to prove a rigorous form of the asymptotic equation \eqref{asympt-eqn}.

\begin{lem}\label{lem:asympt-eqn}
 The asymptotic profiles $\rho^{\pm}$ solve the asymptotic equation
\begin{equation}\label{exact-asympt-eqn}
 i  \partial_t \rho^{\pm}(t,v) = v^\alpha A_\alpha \rho^{\pm}(t,v) 
 + e^{2i \phi_\pm} A_\alpha  \sqrt{1-v^2} P^{\pm}_v\gamma^\alpha \rho^{\mp}(t,v)
 + O(\epsilon t^{-\frac{3}{2} + c \epsilon}(1-v^2)^{-\frac14} ).
\end{equation}
 \end{lem}
 
Compared to \eqref{asympt-eqn}, here 
we have the additional second term.
But this is nonresonant, so it will 
not play a significant role in the global
asymptotics for $\psi$.

\begin{proof}
To prove this result we combine the linear 
Dirac equation \eqref{fv} for the wave packets  with the magnetic wave equation for $\psi$, by writing the duality relation 
\[
\langle (\gamma^0 \gamma^\alpha \partial_\alpha + i \gamma^0) \psi, \psi_v^\pm \rangle +   \langle  \psi, (\gamma^0 \gamma^\alpha \partial_\alpha +i \gamma^0) \psi_v^\pm \rangle =  i A_\alpha \langle \gamma^0 \gamma^\alpha \psi, \psi_v^{\pm}\rangle 
- i \langle {\gamma^0} \psi, f_v^{\pm} \rangle,
\]
or equivalently
\[
\partial_\alpha \langle \gamma^0 \gamma^\alpha  \psi, \psi_v^\pm \rangle = i A_\alpha \langle \gamma^0 \gamma^\alpha \psi, \psi_v^{\pm}\rangle 
- i \langle \gamma^0 \psi, f_v^{\pm} \rangle.
\]
Integrating this relation in the region $H_{[t_1,t_2]}$, between two hyperboloids 
$H_{t_1}$ and $H_{t_2}$ through $(t_1,vt_1)$ and $(t_2,vt_2)$, we obtain
\[
\langle(\rho^{\pm}(t_2,v) - \rho^{\pm}(t_1,v)),\sigma^{\pm}(v)\rangle_H = 
\int_{H_{[t_1,t_2]}}  i A_\alpha \langle \gamma^0 \gamma^\alpha \psi, \psi_v^{\pm}\rangle 
-i \langle \gamma^0 \psi, f_v^{\pm} \rangle \,dx dt.
\]
We divide by $t_2-t_1$ and pass to the limit 
$t_2-t_1 \to 0$, to obtain the differential relation
\[
\begin{aligned}
t_1\frac{d}{dt_1} \langle\rho^{\pm}(t_1,v),\sigma^{\pm}(v)\rangle_H = & \  
\int_{H_{t_1}}  \frac{t^2-x^2}{{\sqrt{t^2 + x^2}}}  \left(i A_\alpha \langle \gamma^0 \gamma^\alpha \psi, \psi_v^{\pm}\rangle 
- i  \langle \gamma^0 \psi, f_v^{\pm} \rangle \right) \, {d\sigma}
\\
 = & \  
\int_{H_{t_1}} (t^2-x^2)^2  \left(i A_\alpha \langle \gamma^0 \gamma^\alpha \psi, \psi_v^{\pm}\rangle 
- i \langle \gamma^0 \psi, f_v^{\pm} \rangle \right) \, dV_H.
\end{aligned}
\]
We can simplify the right hand side by peeling off some error terms. We begin with the contribution of $f_v^{\pm}$, which is as in \eqref{fv} with $r_v^{\pm}$ as in \eqref{rv-rep}. Then this contribution has the form
\[
\begin{aligned}
Err^1 =& \ - i \int_{H_{t_1}} (t^2-x^2)^\frac54 e^{i\phi_{\pm}}  
\langle \gamma^0 \psi, \hat \Omega \tilde r_v^{\pm} +\tilde{\tilde r}_v^{\pm} \rangle \, dV_H
\\ 
= & \ - i \int_{H_{t_1}} (t^2-x^2)^\frac54 e^{i\phi_{\pm}}  
\left(-\langle \hat \Omega \gamma^0 \psi, \tilde r_v^{\pm}\rangle +\langle \psi,\tilde{\tilde r}_v^{\pm} \rangle \right)\, dV_H.
\end{aligned}
\]

Using the pointwise bounds for $\psi$ and 
$\hat \Omega \psi$ in the $H$ norm, we obtain
\[
|Err^1| \lesssim \epsilon (t^2-x^2)^{-\frac14} t^{C\epsilon}\| \psi_v^{\pm}\|_{H},
\]
which is acceptable. The pointwise estimates for $\hat{\Omega} \psi$ follow from~\eqref{Energy-H} with $\psi$ replaced by $\Omega^{\leq 3} \psi$ combined with Proposition~\ref{prop:EnergyEstGamma}.

The next error term arises from freezing $A$ 
at $x_1= t_1 v$,
\[
\begin{aligned}
Err^2 =   \int_{H_{t_1}} (t^2-x^2)^2  i (A_\alpha-A_\alpha(t_1,vt_1)) \langle \gamma^0 \gamma^\alpha \psi, \psi_v^{\pm}\rangle  \,  dV_H.
\end{aligned}
\]
Using the pointwise bounds for $\mathcal TA$,
within the support of $\psi_v^\pm$ we have 
\[
|(A_\alpha-A_\alpha(t_1,x_1))| \lesssim 
t_1 \|\mathcal T A\|_{L^\infty} d(x/t,v)
\lesssim \epsilon t_1  t_1^{-1+c\epsilon}(t_1^2-x_1^2)^{-\frac12}
(t_1^2-x_1^2)^{-
\frac14} = \epsilon t_1^{\epsilon} (t^2-x^2)^{-\frac34},
\]
 which suffices in order to estimate $Err^2$,
\[
|Err^2| \lesssim \epsilon(t_1^2-x_1^2)^{-\frac14}\, 
t_1^{c\epsilon}\,  \| \psi_v^{\pm}\|_{H}.
\]

It remains to
look at the leading term in the coefficient of $A_\alpha$, which is
\[
\begin{aligned}
\langle \gamma^0 \gamma^\alpha \psi, \psi_v^{\pm} \rangle= & \ 
\langle \psi, \gamma^0 \gamma^\alpha \psi_v^\pm \rangle
\\
= & \ 
\langle \psi, \gamma^H\gamma^\alpha \psi_v^\pm \rangle_{H}
\\
= & \ 
\langle \psi, (P^+_v - P_v^-)\gamma^\alpha \psi_v^\pm \rangle_{H}.
\end{aligned}
\]
To avoid ambiguities we fix the $+$ sign on the left. Integrating, this gives
\[
\begin{aligned}
\int (t^2-x^2)^2 \langle \gamma^0 \gamma^\alpha \psi, \psi_v^{+} \rangle \, dV_H= & \ 
\sqrt{t^2-x^2}(\langle \rho^+, P^+ \gamma^\alpha \sigma_v^+ \rangle_H
- e^{2i\phi} \langle \rho^-, P^- \gamma^\alpha \sigma_v^+ \rangle_H)
\\ = & \  \frac{x^\alpha}{\sqrt{t^2-x^2}}\langle \rho^+, \sigma_v^{+} \rangle_{H} 
- e^{2i\phi} \langle \rho^-,  \gamma^\alpha \sigma_v^+ \rangle_H
\\ = & \  x^\alpha \langle \rho^+, \sigma_v^{+} \rangle_{H} 
+ e^{2i\phi} \sqrt{t^2-x^2} \langle \gamma^\alpha \rho^-,  \sigma_v^+ \rangle_H.
\end{aligned}
\]
Summing up the estimates for all the terms,
we obtain 
\[
\begin{aligned}
t_1\frac{d}{dt_1} \langle\rho^{+}(t_1,v),\sigma_v^{+}\rangle_H = &\ x^\alpha A_\alpha \langle \rho^+, \sigma_v^{+} \rangle_{H} 
+ e^{2i\phi} \sqrt{t_1^2-x_1^2} A_\alpha\langle \gamma^\alpha \rho^-,  \sigma_v^+ \rangle_H
\\
& \ 
+ O(\epsilon(t_1^2-x_1^2)^{-\frac14}
t_1^{c\epsilon} \| \psi_v^{\pm}\|_{H}),
\end{aligned}
\]
with a similar relation with changed signs.
This implies the conclusion of the lemma.
\end{proof}

\subsection{Global bounds for $\psi$} Our primary objective here is to 
use the asymptotic equation for the asymptotic profiles $\rho^{\pm}$ for $\psi$ in order to establish  global bounds for $\psi$, which in particular will close the $\psi$ part of our bootstrap.

We begin with the global uniform bound for $\psi$ from our Klainerman-Sobolev inequalities, which we recall here,
\begin{equation}
|\psi_{+}| \lesssim \epsilon t^{-\frac32} t^{c\epsilon}, \qquad  |\psi_{-}| \lesssim \epsilon \sqrt{t^2-x^2} t^{-\frac{5}{2}} t^{c\epsilon},
\qquad t-r \geq 1,
\end{equation}
which when switched to the $H$ norm yields
\begin{equation}
\|\psi\|_{H}  \lesssim \epsilon t^{-2} (t^2-x^2)^{\frac14} t^{c\epsilon}.
\end{equation}
This immediately yields a corresponding bound for the asymptotic profiles $\rho^{\pm}$, 
\begin{equation}\label{start-rho}
\| \rho^{\pm}(t,v)\|_{H} \lesssim 
\epsilon t^{C \epsilon} (1-v^2). 
\end{equation}
This is not sufficient for large $t$ because of the $t^{C\epsilon}$ growth. Instead we will use it in order to initialize the asymptotic equation. Precisely, we consider 
the subset of the forward light cone 
\[
 C^{in} =\{(t,x) \in (0,\infty) \times \mathbb{R}^3 \colon t > t_0(v):=(1-v^2)^{-10}, \text{ where } v = x/t \}.
\]
On the boundary of this set,  \eqref{start-rho} yields
\begin{equation}\label{start-rho2}
\| \rho^{\pm}(t,v)\|_{H} \lesssim 
 \epsilon(1-v^2)^{1-C\epsilon}, \qquad t = t_0(v) .
\end{equation}  
This will serve as initial data for the 
asymptotic equation on each ray $x = vt$.
Now we turn our attention to the asymptotic equation \eqref{exact-asympt-eqn}.

A preliminary step here is to eliminate 
the nonresonant terms, which is achieved 
by introducing the corrections
\begin{align*}
    \rho_1^\pm(t,v) = \rho^{\pm}(t,v) -
e^{\pm 2i \phi} A_\alpha P_v^\pm \gamma^\alpha 
\rho^{\mp}(t,v).
\end{align*}

On the one hand these corrections are small and decay at infinity, 
\begin{equation}
\label{del-rho}
\|\rho_1^\pm(t,v) - \rho^{\pm}(t,v)\|_{H}
\lesssim C \epsilon^2 t^{-2+C\epsilon} \sqrt{t^2-x^2}.
\end{equation}
On the other hand, they remove the nonresonant terms in the asymptotic equation. Precisely,
a direct computation using the pointwise bounds for $\nabla A$ and $\rho^{\pm}$ yields
\begin{equation}\label{exact-asympt-eqn+}
  i  \partial_t \rho^{\pm}_1(t,v) = v^\alpha A_\alpha \rho^{\pm}_1(t,v)
 + O((\epsilon+C \epsilon^2)t^{-1}(t^2-t^2 v^2)^{-\frac14}
t^{C\epsilon} ).
\end{equation}
Integrating this equation from $t_0$ to infinity yields the uniform bound
\[
\begin{aligned}
\|\rho_1^{\pm}(t,v)\|_{H} \lesssim & \ \epsilon(1-v^2)^{1-C\epsilon} + \int_{(1-v^2)^{-10}}^\infty t^{-1} (\epsilon+C \epsilon^2)(t^2-t^2 v^2)^{-\frac18} t^{C\epsilon} \, dt
\\
\lesssim & \ (\epsilon+C \epsilon^2)(1- v^2)^{1-C\epsilon}.
\end{aligned}
\]
We can now return to $\rho^{\pm}$ to obtain a similar bound,
\begin{equation}
  \|\rho^{\pm}(t,v)\|_{H} \lesssim   (\epsilon+C \epsilon^2)(1- v^2)^{1-C\epsilon}, \qquad  t^2 - x^2 > t.
\end{equation}
We remark that initially this is valid for $ t \geq t_0(v)$, but the bound \eqref{start-rho}

allows us to extend it to the full range above,
which is our final uniform bound for the asymptotic profile. From here, we can use 
Lemma~\ref{l:asympt} to transfer the bound to $\psi$, which yields
\begin{equation}\label{global-psi}
  \|\psi(t,x)\|_{H} \lesssim  t^{-\frac32}   (\epsilon+C \epsilon^2)(1- v^2)^{\frac14-C\epsilon}, \qquad  t > t_0(v),
\end{equation}
where, assuming $\epsilon$ is small enough, we can also replace $\epsilon+C \epsilon^2$ simply by 
$\epsilon$.  Component-wise this yields
\begin{equation}
|\psi_{+}| \lesssim \epsilon t^{-\frac32} (1-v^2)^{-C\epsilon}, \qquad  |\psi_{-}| \lesssim \epsilon  t^{-\frac{3}{2}}(1-v^2)^{\frac{1}{2} - C \epsilon}, \qquad t \geq t_0(v).
\end{equation}
This can be combined with the global bound 
for $\psi$ from Klainerman-Sobolev inequalities outside $C^{in}$. We summarize the outcome in the following
\begin{lem}\label{psi-final}
The spinor field $\psi$ satisfies the following bounds inside the cone:
\begin{equation}\label{psi-in}
\begin{aligned}
|\psi_{+}| \lesssim & \ \epsilon t^{-\frac32} (1-v^2)^{-C\epsilon}, 
\\
|\psi_{-}| \lesssim & \ \epsilon t^{-\frac32} (1-v^2)^{\frac12 -C\epsilon}.
\end{aligned}
\end{equation}
\end{lem}
This in particular proves the pointwise bound for $\psi$ in 
\eqref{thm-point} in Theorem~\ref{thm:gwp-easy}.

Now we are finally able to close the bootstrap
assumption for $\psi$.  For this we use 
the above bound in the region $C^{in}$, 
and our initial Klainerman-Sobolev bound outside, taking advantage of the better decay of $\psi$ away from the cone and the size of the remaining region,
\[
\| \psi(t)\|_{L^6}^6 \lesssim 
\int_{C_t^{in}} |\psi|^6 \, dx + \int_{\R^3\setminus C_t^{in}} |\psi|^6\,  dx  
\lesssim \epsilon^6 t^{-6} + \epsilon^6 t^{-6 + 6C\epsilon - \frac{2}{3}} + \epsilon^6 t^{-6(1+\delta)+C\epsilon} .
\]
This gives
\begin{equation}
 \| \psi(t) \|_{L^6} \lesssim \epsilon t^{-1}, \end{equation}
as needed in order to close the $\psi$
part of our bootstrap assumption, since the implicit constant is independent of the (large) constant in the bootstrap assumption~\eqref{eq:boot-psi-A}.

To avoid cluttering the proof, the above arguments omit a minor but necessary technical point, which we now explain and 
clarify. Precisely, our bootstrap assumption \eqref{eq:boot-psi-A} is made 
in the set $C_{<T}$, and we need to improve it within the same set. However, 
if $\psi$ is  known only in $C_T$, then the asymptotic profiles $\rho_{\pm}$, whose
definition involves localized integrals 
on hyperboloids, are only defined within a smaller subset $D_{<T} \subset C_{<T}$, 
which can be easily computed as follows:
\[
D_{<T} = \{ (t,x) \in C_{<T}; t^2-x^2 > t, \ t \leq 4T ( 1 - C (t^2-x^2)^{-\frac14})\}
\]
with a fixed large universal constant $C$.
To easily see what is  the above upper boundary of $D_{<T}$, we need to move
down along hyperboloids from the top of $C_T$ for a hyperbolic 
distance $(t^2-x^2)^{-\frac14}$. The full
$C_T$ has size $1$ along hyperboloids, 
so it suffices to select the proportional 
region.

Then our arguments only yield the bound 
\eqref{psi-in} within $D_{<T}$, and we need an additional argument in order to 
cover the remaining region $C_{<T}\setminus D_{< T}$. This is a direct application of the Fundamental Theorem of Calculus along hyperboloids, which is most readily seen using the $H$
norm. Precisely, the bound \eqref{psi-in}
is equivalently written as
\begin{equation}\label{psi-H-DT}
\| \psi(t,x)\|_{H} \lesssim \epsilon t^{-\frac32}(1-v^2)^{\frac14-C\epsilon}.
\end{equation}
On the other hand the energy estimates on hyperboloids, combining Proposition~\ref{prop:EnergyEstGamma}
and \eqref{Energy-H}, yield the $L^2$
bound 
\[
\int_{H \cap C_T} \|(t^2-x^2)^\frac32 \hat \Omega^{\leq 3} \psi\|_{H}^2 dV_H \lesssim \epsilon T^{C\epsilon} ,
\]
which by Sobolev embeddings on $H \cap C_T$
yield the pointwise bound
\[
\|\hat \Omega \psi\|_{H}\lesssim \epsilon (t^2-x^2)^{-\frac32} T^{C \epsilon}.
\]
Here we use the operators $\hat \Omega$ whose transport conveniently generates isometries in the $H$ norm.

Now each point $(t_0,x_0)$ in $C_T$ may be connected along a hyperbolic geodesic to a point $(t_1,x_1)$ in $D_T$ at distance $\lesssim (t^2-x^2)^{-\frac12}$. Applying the fundamental theorem of calculus along this geodesic in the $H$ norm we obtain
\[
\begin{aligned}
\| \psi(t_0,x_0)\|_{H} 
\lesssim & \ \| \psi(t_1,x_1)\|_{H} + 
\epsilon (t^2-x^2)^{-\frac12} (t^2-x^2)^{-\frac32} T^{C \epsilon}
\\
\lesssim & \ \epsilon T^{-\frac32}(1-v^2)^{\frac14-C\epsilon} + 
\epsilon T^{-2+C\epsilon} (1-v^2)^{-2} .
\end{aligned}
\]
Here we see that the contribution of the 
second term is at most comparable if $(1-v^2) > T^{-\frac27}$, so we recover \eqref{psi-H-DT} at $(t_0,x_0)$. On the other hand if $(1-v^2) < T^{-\frac27}$
then we obtain the same outcome simply by combining the energy estimates in Proposition~\ref{prop:EnergyEstGamma}
with the Klainerman-Sobolev inequality in 
Theorem~\ref{t:KS}.

\vspace{1em}

\section{Uniform bounds for $A$}
\label{s:A}
Our primary goal here is  to capture
the $t^{-1}$ decay of $A$, and thus close the $A$ component of our bootstrap bound \eqref{eq:boot-psi-A}.

Recall that the wave equation that $A$ satisfies is 
\[
 \Box A_\mu = - \overline{\psi} \gamma_\mu \psi,
\]
with initial data
\[
A_\mu(0) = a_\mu, \quad \partial_t A_\mu(0) = \dot a_\mu .
\]
We separate the contributions from the 
source and from the initial data, writing
\[
A= A^{hom} + A^{inhom}.
\]
We will obtain separate bounds for the two components. 

\subsection{Bounds for $A^{hom}$.}

Our goal here will be to prove the following: 

\begin{prop}\label{prop:Ahom}
Assume that the initial data $a,\dot a$ for $A^{hom}$ is as in Theorem~\ref{thm:gwp-easy}. Then $A^{hom}$ satisfies the uniform bound
\begin{equation}\label{eq:A-hom}
|A^{hom}(t,x)| \lesssim \frac{\epsilon}{\langle t \rangle \langle t-r\rangle^{\nu/6}  } .   
\end{equation} 
\end{prop}

We remark here the important role played by the $\dot H^{\frac12-\nu} \times \dot H^{-\frac12 - \nu}$ bound
in \eqref{small-data-nu}, where the condition $\nu > 0$ is needed in order to guarantee the uniform $t^{-1}$ decay included in \eqref{eq:A-hom}. 

\begin{rem}\label{rem:mu0}
In the limiting case  $\nu = 0$ 
the bound \eqref{eq:A-hom} is no longer
valid. Instead the best uniform bound
has a logarithmic loss, close to the cone,
\begin{equation}\label{eq:A-hom-limit}
|A^{hom}(t,x)| \lesssim \frac{\epsilon}{\langle t \rangle}  \log \frac{2\langle t \rangle} {\langle t-r\rangle}.  
\end{equation} 
Alternatively, one may retain the $t^{-1}$ decay in the BMO setting, 
\begin{equation}\label{eq:A-hom-BMO}
\|A^{hom}(t,x)\|_{BMO} \lesssim \frac{\epsilon}{\langle t \rangle}.
\end{equation}
\end{rem}

We now prove the above result.

\begin{proof}[Proof of Proposition~\ref{prop:Ahom}]
Since we are solving a homogeneous problem,
the vector field energy bounds are immediate
and without the $t^{C\epsilon}$ loss. In particular we have
\begin{equation}
\| \Gamma^{\leq 9} A^{hom}\|_{X_T} \lesssim \epsilon,
\end{equation}
which by our Klainerman-Sobolev inequalities 
yields
\begin{equation}
| \nabla A^{hom}| \lesssim \langle t\rangle^{-1} 
\langle t-r \rangle^{-\frac12}.
\end{equation}
This immediately implies a similar bound for 
the high frequencies of $A^{hom}$,
\begin{equation}
| P_{\geq 1} A^{hom}| \lesssim \langle t\rangle^{-1} 
\langle t-r \rangle^{-\frac12}.
\end{equation}
It remains to consider low frequencies,
which we index according to the dyadic frequency $\lambda < 1$.

Here we have on one hand the $\dot H^{\frac12} \times \dot H^{-\frac12}$ initial data bounds from \eqref{small-data},
which implies the energy bounds
\begin{equation}\label{hom-energy}
\| \Gamma^{\leq 9} A^{hom}_\lambda [\cdot ]  \|_{L^\infty (\dot H^1 \times L^2)} \lesssim \epsilon \lambda^{\frac12}  .   
\end{equation}
On the other hand from \eqref{small-data-nu}
we also have the better undifferentiated bound
\begin{equation}\label{hom-energy-nu}
\|  A^{hom}_\lambda [\cdot ]\|_{L^\infty (\dot H^1 \times L^2)} \lesssim \epsilon \lambda^{\frac12+\nu}     .
\end{equation}

Using only \eqref{hom-energy}, our proof of the Klainerman-Sobolev inequalities in Theorem~\ref{t:KS} gives the pointwise 
bound 
\begin{equation}
| \nabla A^{hom}_{\lambda}| \lesssim \lambda^\frac12\langle t\rangle^{-1} 
\langle t-r \rangle^{-\frac12}.
\end{equation}
Since localization to frequency $\lambda$ 
means averaging on the $\lambda^{-1}$ scale,
the above bound implies that
\begin{equation}
|  A^{hom}_{\lambda}| \lesssim \lambda^{-\frac12} (t + \lambda^{-1})^{-1} 
( |t-r|+\lambda^{-1})^{-\frac12} \lesssim \frac{1}{t},
\end{equation}
which is tight only  when $t \gtrsim \lambda^{-1}$ and $|t-r| \lesssim \lambda^{-1}$. Here dyadic summation over 
$\lambda < 1$ would lead to \eqref{eq:A-hom-limit}.

On the other hand if we also use the improved bound \eqref{hom-energy-nu},
then the same argument in 
Theorem~\ref{t:KS} combined with interpolation gives a slightly better pointwise bound 
\begin{equation}
| \nabla A^{hom}_{\lambda}| \lesssim \lambda^{\frac12+\frac{\nu}6} \langle t\rangle^{-1} 
\langle t-r \rangle^{-\frac12},
\end{equation}
which in turn implies that
\begin{equation}
|  A^{hom}_{\lambda}| \lesssim \lambda^{-\frac12+\frac{\nu}{6}} (t + \lambda^{-1})^{-1} 
( |t-r|+\lambda^{-1})^{-\frac12}.
\end{equation}
Here dyadic summation for $\lambda \leq 1$
directly gives the bound \eqref{eq:A-hom}.

\end{proof}

\subsection{Bounds for $A^{inhom}$}
To estimate $A^{inhom}$ inside the cone,
we use the pointwise bounds for $\psi$
in order to estimate the source term
\[
G^\alpha = \overline{\psi} \gamma^\alpha \psi. 
\]
Precisely, from Lemma~\ref{psi-final}
we have the bound inside the cone
\begin{equation}\label{Gint}
|G^\alpha| \lesssim \epsilon \langle t\rangle^{-3}
\left( \frac{ | t + r | }{| t-r |}  \right)^{C\epsilon},
\end{equation}
while from Theorem~\ref{t:KS} and~\eqref{eq:PointwBoundCTSmin} we have a better bound outside the cone,
\begin{equation}\label{Gout}
  |G^\alpha| \lesssim \epsilon \langle t+r\rangle^{-3} \langle t-r \rangle^{-\delta}  .
\end{equation}
We use these two bounds together with the fundamental solution of the wave equation
in order to estimate $A^{inhom}$:

\begin{prop} \label{prop:Ainhom}
We have 
\begin{equation}
|A^{inhom}| \lesssim \epsilon t^{-1}   \langle t-r \rangle_-^{-\delta}.
\end{equation}
\end{prop}

Together with Proposition~\ref{prop:Ahom}, 
the above proposition implies the 
global pointwise bound 
\[
\| A(t)\|_{L^\infty} \lesssim \frac{\epsilon}{\langle t\rangle}, 
\]
which in turn closes the bootstrap assumption on $A$, and thus completes the proof of  our global well-posedness result in Theorem~\ref{thm:gwp-easy}.

\medskip

For later use we remark that the bulk 
of the output here comes from a region strictly inside the cone:

\begin{rem}\label{rem:peel-ext}
The main contribution to $A$ comes from the portion of $G$ strictly inside the cone. Precisely, if the source term $G$ is localized
to an exterior region, say 
$G^{ext} := 1_{D^{ext}} G$ where
\[
D^{ext} := \{ t < (1-v^2)^{-10}\}, 
\]
then the corresponding solution $A^{ext}$
to $\Box A^{ext} = G^{ext}$ satisfies 
a better bound
\begin{equation}
|A^{ext}| \lesssim \epsilon t^{-1}   \langle t-r \rangle^{-\delta},
\end{equation}
with decay also inside the cone. This 
is proved in the same way as in the proof of the proposition below.
\end{rem}

\begin{proof}[Proof of Proposition~\ref{prop:Ainhom}]
Here we  use the positivity of the fundamental solution for the wave equation in three space dimensions to reduce to the radial case  solving 
\[
\Box B = G
\]
with  the radial source as in the right hand side of \eqref{Gint}, respectively \eqref{Gout},
\[
G := G^{int} + G^{out}.
\]
Then we must have $|A| \lesssim B$.
We can further reduce 
to the one dimensional case, solving 
\[
\Box_1(r B) = r G
\]
forward in time with zero Cauchy data at $t=0$ and zero boundary condition at $r = 0$.

This has fundamental solution
\[
r B(t,r) = \frac12 \int_{D(r,t) \cap \{t > 0\}} r_1 G(r_1, t_1)\,  dr_1 dt_1 ,
\]
where the domain $D(r,t)$ is the backward 
infinite null rectangle with two vertices 
at $(r,t)$ and $(0,t-r)$.
The rest of the proof is a straightforward computation of the last integral, which is left for the interested reader.
\end{proof}

\section{Modified scattering}

The aim of this final section is to 
study the modified scattering asymptotics 
for $A$ and $\psi$, and thus to 
complete the proof of Theorem~\ref{thm:ms}.
The steps are as follows:

\begin{enumerate}[label=\alph*)]
\item We solve the asymptotic equation for $\rho^{\pm}$, and show that the limit of $\rho^{\pm}$ exists up to a phase rotation.

\item Using the $\psi$ asymptotics, we compute the $A$ asymptotics inside the cone, showing 
that the limit of $t A$ exists at infinity.

\item Using the $A$ asymptotics, we return to 
$\rho^{\pm}$ and compute the phase rotation factor, completing the proof of Theorem~\ref{thm:ms}.

\end{enumerate}

\subsection{ The radiation profile of $\psi$}
\label{sec:rad-psi}
Here we renormalize the asymptotic profiles $\rho^{\pm}$ for the spinor field $\psi$ by setting
\begin{equation}
\tilde\rho^{\pm} (t,v) := \rho^{\pm} (t,v) e^{i \theta(t,v) }   ,
\end{equation}
where the real phase $\theta$ is
\begin{equation}
\theta(t,v) = \int_{t_0(v)}^t \frac{x^\alpha A_\alpha(t,vt)}{t} \, dt ,      
\end{equation}
and where $t_0(v)$ is chosen on a fixed hyperboloid, say $t^2-x^2 = 1$. Then we have the following

\begin{lem}
The limit 
\begin{equation}
\rho^\pm_\infty(v) := \lim_{t \to \infty}  \tilde \rho^{\pm} (t,v)    
\end{equation}
exists for each $|v| < 1$, and satisfies 
the uniform bounds
\begin{equation}\label{rho-inf-point}
\|\rho^{\pm}_\infty(v)\|_{H} \lesssim \epsilon
(1-v^2)^{1-C\epsilon}, 
\qquad \|\tilde \rho^{\pm}(t,v)  -  \rho^\pm_\infty(v) \|_H \lesssim 
\epsilon t^{-\frac1{16}},
\end{equation}
as well as the energy identity
\begin{equation}
\sum_{\pm} \|\rho^{\pm}_{\infty}\|_{L^2(H)}^2
= \| \psi_0\|_{L^2}^2.
\end{equation}
\end{lem}

\begin{proof}
To obtain the asymptotic limit it is more convenient to work with the modified functions $\rho_1^\pm$ introduced earlier,  where the nonresonant contribution to the asymptotic equation is removed.
Setting 
\[
\tilde\rho_1^{\pm} :=  e^{i\theta} \rho_1^{\pm},
\]
the equation \eqref{exact-asympt-eqn+}
becomes
\begin{equation}\label{corrected-asympt-eqn}
 i  t  \partial_t \tilde \rho_1^{\pm}(t,v) =  O(\epsilon (t^2-x^2)^{-\frac1{4}} t^{c\epsilon} ).
\end{equation}
This suffices in order to obtain the 
asymptotic limit
\[
\rho_{\infty}^\pm(v) = \lim_{t \to \infty}
\tilde \rho_1^{\pm}(t,v),
\]
as well as the difference bound
\[
\|\rho_{\infty}^\pm(v) 
- \tilde \rho_1^{\pm}(t,v)\|_{H} \lesssim 
\epsilon (t^2-x^2)^{-\frac1{4}} t^{c\epsilon}
.
\]
Combining this with \eqref{del-rho}
we can return to $\rho^\pm$ and obtain
\begin{equation}\label{rhoinf1}
  \|\rho_{\infty}^\pm(v) -
\tilde \rho^{\pm}(t,v)\|_{H} \lesssim  \epsilon (t^2-x^2)^{-\frac1{4}} t^{c\epsilon} \qquad t^2 -x^2 \geq t.  
\end{equation}
Employing \eqref{start-rho2} and evaluating the above difference bound in $t_0(v) = (1-v^2)^{-10}$, we obtain
\begin{equation}\label{rhoinf2}
   \|\rho_{\infty}^\pm(v)\|_{H}
   \lesssim \epsilon(1-v^2)^{1-C\epsilon}.
\end{equation}
We can also use Lemma~\ref{l:asympt}
to  introduce $\rho^{\pm}_\infty$ in the asymptotic expansion for $\psi$,
\begin{equation}\label{psi-asymptotic+}
\psi(t,x) - \sum_{\pm}
(t^2-x^2)^{-\frac34} e^{i\phi_\pm} 
\rho_\infty^{\pm} e^{-i \theta} 
= O_H(\epsilon (t^2-x^2)^{-\frac{7}{8}} t^{c\epsilon}).
\end{equation}

This allows us to compare $L^2$ norms. 
Given a hyperboloid $H$ and a time slice $t=T$ intersecting it, we can integrate the 
density flux relation for $\psi$ over 
the cup region above $t = 0$ and below 
both the hyperboloid and the time slice. We arrive at
\[
\int |\psi_0|^2\, dx = \int_{H_{<T}}
(t^2 -x^2)^{\frac32} \|\psi\|_{H}^2
\, dV_H + \int_{t = T,\, below \, H} |\psi|^2 \,dx.
\]
We choose $T$ so that the plane $\{t = T\}$ intersects the hyperboloid $H$ at $T = (1-v^2)^{-10}$.
Then the last integral on the right decays to zero as $T \to \infty$ in view of the pointwise bounds for $\psi$ in Theorem~\ref{thm:gwp-easy},
while in the the first integral 
we can substitute $\psi$ as in \eqref{psi-asymptotic+} with a decaying error. We arrive at 
\[
\|\psi_0\|_{L^2}^2 = \lim_{T \to \infty} \int_{H_{<T}} \| \rho^+\|_{H}^2
+\| \rho^-\|_{H}^2\,  d V_H = \int \sum_{\pm} \|\rho_{\infty}^\pm \|_{H}^2 \, dV_H,
\]
which is the desired energy identity.

\end{proof}

\subsection{ The radiation profile of $A$}

The radiation profile for $A$ comes from the source term in the wave equation for $A$, for which we look at the asymptotics for $\psi$ in \eqref{psi-asymptotic+}, namely 
\[
\psi(t,x) = \sum_\pm (t^2-x^2)^{-\frac34} \rho^{\pm}_\infty(v) e^{i\phi_{\pm}} e^{-i\theta} + O_H(\epsilon(t^2-x^2)^{-\frac{7}{8}}t^{c\epsilon})
\]
in the region 
\[
D^{int} : = \{t >(1-|v|^2)^{-10}\},
\]
where we recall that the contributions
from the exterior of this region yield 
output which decays faster than $t^{-1}$ along rays, see Remark~\ref{rem:peel-ext}.

Inserting this into the equation for $A^{inhom}$ we obtain in the same set
the equation
\[
 \Box A^{\mu,inhom} := - \sum_{\pm,\pm}(t^2-x^2)^{-\frac32}  
  \overline{\rho_\infty^{\pm}(v) e^{i\phi_\pm}} \gamma^\mu \rho_\infty^{\pm}(v)e^{i\phi_{\pm}} + 
  O(\epsilon^2(t^2-x^2)^{-\frac{13}{8}}t^{c \epsilon}).
\]
We discard again the contribution of the error term, which has better than $t^{-1}$ decay at infinity inside the cone in timelike directions.
If in the leading term we have two $\rho_\infty^+$ or two $\rho_\infty^-$ then the phases cancel.
But for the mixed terms $\rho_\infty^{+} \rho_\infty^{-}$ they add up, so we expect a lot of cancellation from the oscillations when we solve the wave 
equation. So we separate the principal part into two,
\[
\Box A^\mu_{main} := 1_C (t^2-x^2)^{-\frac32}  \sum_{\pm}
\overline{\rho}_\infty^{\pm}(v)  \gamma^\mu \rho_\infty^{\pm}(v),
\]
respectively 
\[
\Box A^\mu_{err} = 1_C (t^2-x^2)^{-\frac32} \sum_{\pm}
e^{\mp 2i \sqrt{t^2-x^2}} \overline{\rho}_\infty^{\pm}(v) \gamma^\mu \rho_\infty^{\mp}(v).
\]
Both of these are taken with zero Cauchy data. In both terms we 
have completed the region $D^{int}$ to the full cone $C$, at the expense of another error with better than $t^{-1}$
decay.

\bigskip

We first consider the main component, for which we have

\begin{prop}
$A^\mu_{main}$ can be expressed in the form
\[
A^{\mu}_{main} = (t^2-x^2)^{-\frac12}  a^\mu_\infty(v) ,
\]
where $a^\mu_\infty$ solve the equations
\begin{equation}\label{eqn:ainf}
(- 1 - \Delta_H)a^\mu_\infty =  - \frac{v^\mu}{\sqrt{1-|v|^2} } (\|\rho_\infty^+\|_H^2
+ \|\rho_\infty^-\|_H^2),
\end{equation}
and satisfy pointwise bounds
\begin{equation}
\label{eq:ainf-bounds}
\| (1-v^2)^{-\frac12} a^{\mu}_\infty(v)\|_{C^{2,\frac14}(H)} \lesssim \epsilon^2 .
\end{equation}

\end{prop}
Here $-\Delta_H$ is the hyperbolic Laplacian in 3D, which, we recall, has spectrum $[1,\infty)$.
Returning to $A^\mu_{main}:=(t^2-x^2)^{-\frac12}a^\mu_\infty$, we note that the last bound 
in the above proposition implies the vector field bound
\begin{equation}
| \Omega^{\leq 2} A^{\mu}_{main}(v)| \lesssim \epsilon^2 t^{-1} 
\end{equation}
inside the Poincare disk (unit ball).

\begin{proof}
Here the source term is homogeneous of degree $-3$ and supported inside the cone. Then solving the wave equation will give a solution which is homogeneous of degree $-1$ and still supported inside the cone, so we can write it in the form
\[
A^{\mu}_{main} = (t^2-x^2)^{-\frac12}  a^\mu_\infty(v) .
\]

Again we use the $(t^2-x^2)^{-\frac12}$ weight because of the Lorentz invariance.

So our equation now has the form
\[
\Box [(t^2-x^2)^{-\frac12}  a^\mu_\infty(v) ] =  (t^2-x^2)^{-\frac32} b^\mu_\infty(v), 
\qquad b^\mu_{\infty}(v):=
-  \overline{\rho_{\infty}^{\pm}(v)} \gamma^\mu \rho_\infty^{\pm}(v).
\]
We want to rewrite this as an equation connecting $a$ and $b$. This 
equation is Lorentz invariant so it will involve the hyperbolic space Laplacian, where $v$ is the hyperbolic variable in the Klein-Beltrami model.
To find the contribution of the weights we compute
\[
\Box (t^2-x^2)^{-\frac12} = - (t^2-x^2)^{-\frac32}.
\]
Plugging into the equation above, we obtain an equation in the hyperbolic space of the form (see e.g. \cite{tataru-hyp} for a very similar computation)
\begin{equation}\label{deltaH-a}
(- 1 - \Delta_H)a^\mu_\infty =  b^\mu_\infty(v),
\end{equation}
where $\Delta_H$ is the hyperbolic Laplacian in 3D. We can simplify 
the expression for $b^\mu_{\infty}$ by writing
\[
b^\mu_{\infty} = - \langle \gamma^0 \rho_\infty^\pm, \gamma^\mu \rho^{\pm}_\infty \rangle = \mp \langle \gamma^0 \gamma^H \rho_\infty^\pm, \gamma^\mu \rho^{\pm}_\infty \rangle 
= \pm \langle \rho_\infty^\pm, \gamma^\mu \rho^{\pm}_\infty \rangle_H
= \pm \langle \rho_\infty^\pm, P^{\mp}_v \gamma^\mu \rho^{\pm}_\infty \rangle_H.
\]
Now we use the relation \eqref{twist-com} to finally obtain 
\begin{equation}
b^{\mu}_\infty = - \frac{v^\mu}{\sqrt{1-|v|^2}} (\|\rho_\infty^+\|_H^2
+ \|\rho_\infty^-\|_H^2),
\end{equation}
as needed.

Now we consider the size and regularity of 
$b^\mu_\infty$.

\begin{lem}\label{lem:bmu}
The expression $b^\mu_\infty$ satisfies the 
pointwise bounds
\begin{equation}\label{bmu-point}
|b^\mu_\infty| \lesssim   \epsilon^2 (1-v^2)^{\frac32-C\epsilon} , 
\end{equation}
and the $L^2$ bounds 
\begin{equation}\label{bmu2}
 \| (1-v^2)^{-\frac12+C\epsilon}\Omega^{\leq 2} b^\mu_\infty \|_{L^2(H)}
\lesssim  \epsilon^2.
 \end{equation}
\end{lem}

We will not use the $L^2$ bounds directly, but instead via Sobolev embeddings and interpolation. Precisely, we  have the 
following 
\begin{cor}
 The expression $b^\mu_\infty$ satisfies the 
H\"older bounds
\begin{equation}\label{bmu-holder}
\|(1-v^2)^{-\frac12} b^\mu_\infty\|_{C^\frac14(H)} \lesssim   \epsilon^2 .
\end{equation}
\end{cor}
Here neither of the exponents $\frac12$ and $\frac14$ are sharp. This is obtained by using 
(a local form of) the Morrey inequality in \eqref{bmu2}, which gives $C^\frac12$ but with a power of $(1-v^2)$ above $-\frac12$, and then by interpolating  with \eqref{bmu-point} to improve the power of $(1-v^2)$ at the expense of a lower H\"older exponent.

\begin{proof}[Proof of Lemma~\ref{lem:bmu}]
The estimate \eqref{bmu-point} is a direct consequence of \eqref{rho-inf-point}.

For the $L^2$ bound we compare $b^\mu_\infty$
with 
\[
b^\mu(t,v) = - \frac{v^\mu}{\sqrt{1-|v|^2}} (\|\rho^+(t,v)\|_H^2
+ \|\rho^-(t,v)\|_H^2).
\]
On the one hand, from \eqref{rhoinf1} and \eqref{rhoinf2} we have the pointwise difference bound
\begin{equation}\label{interp1}
|b^\mu_\infty(v) - b^\mu(t,v)|  \lesssim  \epsilon^2 (1-v^2)^{\frac{1}{2} - c \epsilon}
(t^2-x^2)^{-\frac14} t^{c\epsilon}.
\end{equation}
On the other hand, we have the $L^2$ bounds on hyperboloids
\[
\|  \Omega^{\leq 3} \rho^{\pm}(t,v)\|_{L^2(H)} \lesssim t^{C\epsilon}, 
\]
which in turn implies 
\begin{equation}\label{interp2}
\|(1-v^2)^{-\frac12+C\epsilon}\Omega^{\leq 3} b^\mu(t,v)\|_{L^2(H)} \lesssim t^{C\epsilon} .
\end{equation}
We combine the bounds \eqref{interp1} and \eqref{interp2} in an interpolation type argument to obtain the desired estimate \eqref{bmu2} for $b^\mu_{\infty}$  (see \cite{IT1}  for a similar argument).

\end{proof}

Now we return to the bounds for $a^{\mu}_\infty$. The same computation as in Proposition~\ref{prop:Ainhom} yields the bound
\[
|(1-v^2)^{-\frac12} a^\mu_{\infty}| \lesssim \epsilon^2.
\]
Inserting this and \eqref{bmu-holder} into  \eqref{deltaH-a} and using local elliptic regularity, we improve the bound for $a_\infty$ say to $C^1$,
\[
\|(1-v^2)^{-\frac12} a^\mu_{\infty}\|_{C^1(H)} \lesssim \epsilon^2.
\]
Reiterating, we finally obtain 
\begin{equation}
\label{eq:HoelderBoundWaveProfile}
\|(1-v^2)^{-\frac12} a^\mu_{\infty}\|_{C^{2,\frac14}(H)} \lesssim \epsilon^2,
\end{equation}
as needed.
\end{proof}

\begin{rem}\label{r:fundamental}
The operator $-1-\Delta_H$ is a nonnegative  invariant operator in the hyperbolic space. Its inverse can be expressed as an integral operator 
with a radial kernel,
\[
(-1-\Delta_H)^{-1} b(v') = \int_H K(v,v') b(v) \, dV(v) ,
\]
which has a known explicit expression, cf.\ \cite{He84},
\[
K(v,v') = K(d) = \frac{1}{4\pi \sinh d}
\]
which is a special case of a hypergeometric function.
Since in the Klein-Beltrami model we have 
\[
d(0,v) = \frac12 \log\frac{1+|v|}{1-|v|},
\]
it follows that 
\[
K(0,v) = \frac{1}{4\pi} \frac{\sqrt{1-|v|^2}}{|v|}.
\]
The behavior near the unit sphere can be used to prove the expected limit behavior of $(1-|v|^2)^{-\frac12}a^\mu_\infty (v)$ as $|v|\to 1$ (we omit the details).
\end{rem}

We now consider the term $A^\mu_{err}$,
which has the source term
\[
c_\mu = 1_C (t^2-x^2)^{-\frac32}  \sum_{\pm}
e^{\mp 2i \sqrt{t^2-x^2}} \overline{\rho}_\infty^{\pm}(v) \gamma^\mu \rho_\infty^{\mp}(v).
\]
For this we will prove that it has a better than $t^{-1}$ decay at infinity,

\begin{prop}
We have, with a small positive universal $\delta$, 
\begin{equation}
|A^\mu_{err}(t,x)| \lesssim \epsilon^2 
\langle t \rangle^{-1} \langle t-r \rangle^{-\delta} .
\end{equation}
\end{prop}
\begin{proof}
The key idea here is that we want to take 
advantage of the oscillations in the phase. 
But this in turn requires some regularity
for the amplitude of the source term, namely the expression 
\[
f_{\infty}^\pm(v) = \overline{\rho}_\infty^{\pm}(v) \gamma^\mu \rho_\infty^{\mp}(v).
\]
We consider this first:

\begin{lem}
The function $f_\infty$ satisfies the pointwise
bound
\begin{equation}\label{finf-inf}
|f_\infty^\pm(v)| \lesssim \epsilon^2 (1-v^2)^{\frac32-C\epsilon} ,   
\end{equation}
and the $L^2$ bound 
\begin{equation}\label{finf-2}
\| (1-v^2)^{-\frac12+C\epsilon} \Omega^{\leq 2} f_\infty^\pm\|_{L^2(H)} \lesssim \epsilon^2  .   
\end{equation}
\end{lem}
The proof is virtually identical to the proof of Lemma~\ref{lem:bmu}, and is omitted. We now return to the proof of the proposition.

\medskip

A direct computation yields
\[
\Box e^{2i \sqrt{t^2-x^2}} = - 4 e^{2i \sqrt{t^2-x^2}} + 6 i(t^2-x^2)^{-\frac12}
e^{2i \sqrt{t^2-x^2}}.
\]
This might suggest that a good approximate solution to our equation
\[
\Box A = 1_C (t^2-x^2)^{-\frac32}  \sum_{\pm}
e^{\mp 2i \sqrt{t^2-x^2}} f_\infty^\pm(v)
\]
should be 
\[
A_{app}= 1_C (t^2-x^2)^{-\frac32}  \sum_{\pm} \pm \frac14
e^{\mp 2i \sqrt{t^2-x^2}} f_\infty^\pm(v),
\]
which has much better $t^{-3}$ decay at infinity. Computing the remaining source term
we get 
\begin{equation}\label{new-source}
\Box (A - A_{app}) = O((t^2-x^2)^{-2}) f_\infty^\pm(v) + O((t^2-x^2)^{-\frac52}) \Delta_H f_\infty^\pm, 
\end{equation}
where $\Delta_H f$ is the hyperbolic space Laplacian, which roughly has form
\[
\Delta_H f_\infty^\pm = \Omega^2 f_\infty^\pm.
\]
The coefficients on the right in \eqref{new-source} have better decay but we still have two problems:
\medskip

(i) Coefficient growth near the cone; but this can be solved by truncating say to the region $t > (1-v^2)^{-10}$, and treating the outside part directly, without the use of the initial $A_{app}$
guess, see Remark~\ref{rem:peel-ext}.

\medskip

(ii) We do not control the pointwise size of $\Delta_H f_\infty^\pm$, only the $L^2$ size.
This can be addressed by dividing $f_\infty^\pm$ into two parts:
\smallskip

(a) a low frequency part for which we control
$\Delta_H f_\infty^\pm$ in $L^\infty$, and

(b) a high frequency part which has smallness in $L^\infty$.

\smallskip

Precisely, given a truncation frequency $\lambda^2 \geq  1$ for the hyperbolic Laplacian, we consider an associated partition of unity
\[
1 = P_{\leq \lambda} + P_{> \lambda},
\qquad P_{\leq \lambda} = \chi_{\leq \lambda^2}(-\Delta_H).
\]
Correspondingly we split 
\[
f_\infty^\pm = P_{\leq \lambda} f_\infty^\pm + 
P_{>\lambda} f_\infty^\pm: = f_{\infty}^{\pm,lo} + f_\infty^{\pm,hi}.
\]
Then for $f_\infty^{\pm,lo}$ we have by  
\eqref{finf-2} and Bernstein's inequality
\begin{equation}
 |\Delta_H f_{\infty}^{\pm,lo}|   \lesssim \epsilon^2\lambda^\frac32 (1-v^2)^{\frac12-C\epsilon}.
\end{equation}
On the other hand for $f_{\infty}^{\pm,hi}$
we get, also by  \eqref{finf-2} and Bernstein's inequality,
\begin{equation}
 |f_{\infty}^{\pm,hi}|   \lesssim \epsilon^2 \lambda^{-\frac12} (1-v^2)^{\frac12-C\epsilon}.
\end{equation}

We insert $f_{\infty}^{\pm,lo}$ into the 
above correction, and solve 
the wave equation with the remaining source term in \eqref{new-source}, 
using pointwise bounds as in the proof of Proposition~\ref{prop:Ainhom}.
This  yields a solution which satisfies a bound of the form
\[
\lesssim \epsilon^2 \lambda^\frac32 (t^2-x^2)^{-1} .
\]
On the other hand we solve the wave equation directly 
for the contribution of $f_{\infty}^{\pm,hi}$, exactly as in 
Proposition~\ref{prop:Ainhom}.
This  yields a solution which satisfies a bound
\[
\lesssim \epsilon^2 \lambda^{-\frac12} (t^2-x^2)^{-\frac12}.
\]

So we arrive at 
\[
|A_{err}^\mu| \lesssim \epsilon^2 \lambda^\frac32 (t^2-x^2)^{-1}+\epsilon^2 \lambda^{-\frac12} (t^2-x^2)^{-\frac12}.
\]
This is true for every $\lambda\geq  1$.
Optimizing the choice of $\lambda$ on each hyperboloid we get 
\[
|A_{err}^\mu| \lesssim  \epsilon^2 (t^2-x^2)^{-\frac58},
\]
which is a better than $t^{-1}$ decay, as claimed in the Proposition.
\end{proof}

We collect our bounds for $A^\mu_{main}$ and $A^\mu_{err}$ and summarize the outcome of this subsection as follows:

\begin{prop}\label{p:asympt-A}
Inside the light cone we have the following asymptotics for $A^{\mu}$:
\begin{equation}
A^\mu = (t^2-x^2)^{-\frac12} a^\mu_\infty
+ O(\epsilon^2 \langle t\rangle^{-1} 
\langle t-r \rangle^{-\delta}).
\end{equation}
\end{prop}
As a consequence, we obtain the expansion \eqref{psi-asymptotic} in Theorem~\ref{thm:ms}. The $C^\frac12$ 
bound for $a^\mu_\infty$ also follows from \eqref{eq:ainf-bounds}.

\subsection{ The radiation profile of $\psi$, revisited}
In Section~\ref{sec:rad-psi} we have defined 
the radiation profile $\rho^\pm_\infty$ of $\psi$ using a somewhat arbitrary initialization
for the asymptotic equation, which possibly yields a relatively rough phase rotation
factor.  Our objective here is to redefine the 
phase rotation factor so that we achieve two objectives:

\begin{itemize}
    \item We uniquely define a canonical radiation profile $\rho^\pm_\infty$ for $\psi$.

    \item For this radiation profile we prove 
    vector field bounds.
\end{itemize}

Before we proceed, we remark that redefining the phase of the radiation profile does not change its size, and thus does not affect the right hand side of the coupling equation \eqref{eqn:ainf}.

To motivate our reset of the radiation profiles $\rho^{\pm}_\infty$, we recall the asymptotic equation
\[
 t\frac{d}{dt} \rho^{\pm}(t,v) = i x_\alpha A^\alpha   \rho^{\pm}(t,v) + O(t^{-\frac14}).
\]
Substituting $A^\alpha$ by $(t^2-x^2)^{-\frac12} a^\alpha_\infty$ we arrive at 
\[
 t\frac{d}{dt} \rho^{\pm}(t,v) \approx \frac{i v_\alpha}{\sqrt{1-v^2}} a^\alpha_\infty   \rho^{\pm}(t,v),
\]
where we have set $v_0 = -1$.
For this we have exact solutions of the form 
\[
\rho^{\pm}(t,v) \approx  \rho^{\pm}_\infty e^{i\log(t^2-x^2) \frac{v_\mu}{2\sqrt{1-v^2}} a^\mu_\infty} e^{i \theta}
\]
with an additional phase rotation factor.
Here we chose the factor $\frac12 \log(t^2-x^2)$
rather that the natural $\log t$ factor in order to have better Lorenz invariance.
Apriori $\rho_{\infty}^\pm$, and thus $\theta$, depends on our initialization  for the asymptotic equation; one such initialization was chosen as a starting point at the beginning of this section. 

\bigskip

{\bf Here we  redefine}  $\rho_{\infty}^\pm$
by setting instead $\theta = 0$ in the above 
formula, which is equivalent to setting
\begin{equation}\label{redef}
 \rho^{\pm}_\infty(v) := \lim_{t \to \infty} \rho^{\pm}(t,v)e^{-i\log(t^2-x^2) \frac{v_\mu}{2\sqrt{1-v^2}} a^\mu_\infty}  .
\end{equation}
This in some sense means that we initialize $\theta$ at infinity.

\bigskip

Corresponding to this choice, we will prove the following:

\begin{prop}
a) The limit in \eqref{redef} exists for each $v \in B(0,1)$.

b) The asymptotic profile $\rho^{\pm}_\infty(v)$
satisfies the pointwise bounds
\begin{equation}\label{good-asympt}
  |\rho^{\pm}(t,v) -  \rho^{\pm}_\infty e^{i\log(t^2-x^2) \frac{v_\mu}{2\sqrt{1-v^2}} a^\mu_\infty} e^{i \theta}|  \lesssim \epsilon \langle t-r \rangle^{-\delta},
\end{equation}
as well as the $L^2$ bounds
\begin{equation}\label{l2-rhoinf}
\| (1-v^2)^{-\frac32+C\epsilon}
\Omega^{\leq 2} \rho^{\pm}_\infty\|_{L^2} 
 \lesssim \epsilon  .  
\end{equation}
\end{prop}
By Sobolev embeddings, \eqref{l2-rhoinf}
gives the $C^{\frac12}$ regularity for 
$\rho_{\infty}^\pm$ in Theorem~\ref{thm:ms}. On the other hand 
the estimate \eqref{good-asympt} combined with Lemma~\ref{l:asympt} yield the asymptotic expansion \eqref{psi-asymptotic} in the same theorem, completing its proof.

\begin{proof}
a) Using the asymptotic equation for 
$\rho^\pm$ in Lemma~\ref{lem:asympt-eqn} and the asymptotic expansion for $A^\mu$ in Proposition~\ref{p:asympt-A} we compute
\[
\begin{aligned}
\frac{d}{dt} [\rho^{\pm}(t,v) e^{-i\log(t^2-x^2) \frac{v_\mu}{2\sqrt{1-v^2}} a^\mu_\infty}] = & \ it^{-1} (x_\mu A^\mu(t,vt)- \frac{v_\mu}{\sqrt{1-v^2}} a^\mu_\infty))  
\rho^{\pm}(t,v)e^{-i\log(t^2-x^2) \frac{v_\mu}{2\sqrt{1-v^2}} a^\mu_\infty}
\\ & \ + O(\epsilon t^{-1} \langle t-r \rangle^{-\delta})
\\ = & \   O(\epsilon t^{-1} \langle t-r \rangle^{-\delta}).
\end{aligned}
\]

The error is integrable along time-like rays $x = vt$, which implies that the limit in \eqref{redef} exists. Furthermore, integrating, we obtain the difference bound
\begin{equation}\label{redef-est}
 |\rho^{\pm}_\infty(v) -  \rho^{\pm}(t,v)e^{-i\log(t^2-x^2) \frac{v_\mu}{2\sqrt{1-v^2}} a^\mu_\infty}| \lesssim \epsilon  \langle t-r \rangle^{-\delta} .
\end{equation}
To obtain vector field bounds for $\rho^{\pm}$,
we use vector field bounds for the second term above. For $\rho^{\pm}$ we can use three vector fields (hyperbolic derivatives), see \eqref{rho-2}. For $a^\mu_\infty$ we can use 
slightly more than two derivatives, see \eqref{eq:ainf-bounds}. Combining the two we obtain 
\begin{equation}\label{redef-vf}
\|  \rho^{\pm}(t,v)e^{-i\log(t^2-x^2) \frac{v_\mu}{2\sqrt{1-v^2}} a^\mu_\infty}\|_{H^{2+\frac14}(H)} \lesssim \epsilon (t^2-x^2)^{c\epsilon}.
\end{equation}
The bound \eqref{l2-rhoinf} then follows by interpolating between \eqref{redef-est} and \eqref{redef-vf}.
\end{proof}

\bibliographystyle{abbrv}
\bibliography{ModifiedScatteringMaxwellDirac3D} 
 
\end{document}